\documentclass[11pt]{article}
\usepackage{amssymb,latexsym,amsmath,amsbsy,amsthm,amsxtra,amsgen,graphicx,dsfont,color}
\usepackage{enumitem}
\usepackage{hyperref}
\usepackage{float}
\newfont{\msbm}{msbm10 at 11pt}

\newcommand\Quote[1]{``#1"}

\newtheorem{Theo}{Theorem}[section]
\newtheorem{Lemma}[Theo]{Lemma}
\newtheorem{Cor}[Theo]{Corollary}
\newtheorem{Prop}[Theo]{Proposition}
\newtheorem{Exm}[Theo]{Example}
\newtheorem{Dfn}[Theo]{Definition}
\newtheorem{Rmk}[Theo]{Remark}
\newtheorem{Assum}{Assumption} 
\newtheorem{Conj}[Theo]{Conjecture}

\begin{document}

\title{Dirichlet eigenfunction and heat kernel estimates on annular domains}
\author{Brian Chao and Laurent Saloff-Coste \\ Cornell University
\\
\\
\textit{In honor of Professor Leonard Gross}}
\maketitle
\begin{abstract}
Motivated by Euclidean boxes, we consider  \Quote{thin} annular domains of the form $U=(a,b)\times U_0\subseteq \mathbb{R}^n$ in polar coordinates, where the spherical base $U_0\subseteq \mathbb{S}^{n-1}$ is an inner uniform domain. We show that, with respect to the measure $\varphi_U^2$ determined by the principal Dirichlet Laplacian eigenfunction $\varphi_U$, such annular domains satisfy volume doubling and Poincar\'{e} inequalities uniformly over all locations and scales. This implies sharp Dirichlet heat kernel estimates expressed in terms of $\varphi_U$. Our results hold uniformly over the collection of all annuli in $\mathbb{R}^n$. We also give matching two-sided bounds for the first Dirichlet Laplacian eigenfunction and eigenvalue for some annular domains including annuli in $\mathbb{R}^n$. Moreover, we prove eigenfunction inequalities for $\varphi_U$ under domain perturbations of $U$. The proofs of our main results utilize eigenfunction comparison techniques due to Lierl and the authors (\cite{lierllsc}, \cite{chaolsc}), small scale $\varphi_U^2$-Poincar\'{e} inequalities \cite{lierllsc}, as well as a discretization technique of Coulhon and Saloff-Coste \cite{coulsc}. Finally, our methods also imply uniform Neumann heat kernel estimates for thin annular domains.
 
\end{abstract}

{\small {{\it AMS 2020 subject classifications}:  Primary 35B20, 35B51, 35J05, 35K08, 35P15;
Secondary 35J25, 60J60, 60J65}

%35B20 Perturbations in context of PDEs
%35B51 Comparison principles in context of PDEs
%35J05 Laplace operator, Helmholtz equation (reduced wave equation), Poisson equation
%35K08 Heat kernel
%35P15 Estimates of eigenvalues in context of PDEs

%35J25 Boundary value problems for second-order elliptic equations
%60J60 Diffusion processes
%60J65 Brownian motion

}

\tableofcontents

\section{Introduction} 
\subsection*{Setting and goals}
Let $U\subseteq \mathbb{R}^n $ be a bounded domain, and let $\Delta=\sum_i \partial _i^2/\partial x_i^2$ be the Dirichlet Laplacian on $U$.
We denote by $\lambda(U)=\lambda_U>0$ the smallest positive eigenvalue of $-\Delta$, and denote by $\varphi_U>0$ the corresponding \textit{principal Dirichlet Laplacian eigenfunction} with eigenvalue $\lambda_U$. We assume that $\varphi_U$ is normalized so that $\|\varphi_U\|_{L^2(U)}=1$.

We let $d_U$ denote the geodesic distance on $U$, and we consider the metric measure space $(U,d_U,\varphi_U^2dx)$, where $dx$ is the Lebesgue measure on $U$. We write $B_U(x,r)$ to denote the metric balls of this space and write $V_{\varphi_U^2}(x,r)$ for the measure of such balls. On $(U,d_U,\varphi_U^2 dx)$, one can consider the heat kernel
$$\widetilde{p}_U(t,x,y):=\frac{e^{\lambda_U t}p^D_U(t,x,y)}{\varphi_U(x)\varphi_U(y)},$$ where $p^D_U(t,x,y)$ is the Dirichlet heat kernel on $U$. In probabilistic terms, the heat kernel $\widetilde{p}_U(t,x,y)$ determines a Brownian motion in $U$ with infinite lifetime, and with invariant measure $\varphi_U^2dx$. 
The motivation of this work is to identify families of \Quote{thin} Euclidean domains $U$ such that for $(U,d_U,\varphi_U^2 dx)$ and $\widetilde{p}_U(t,x,y)$, we can prove statements of the following type:
\begin{enumerate}
    \item $(U,d_U,\varphi_U^2 dx)$ is volume doubling (Definition \ref{vddefn}) and satisfies Poincar\'{e} inequalities (Definition \ref{pidefn}), uniformly over all locations and scales. The conjunction of these two properties then implies that $\widetilde{p}_U(t,x,y)$ satisfies two-sided Gaussian type heat kernel estimates of the form 
\begin{align}
        \forall t>0,x,y\in U,\enspace\frac{e^{\lambda_U t}p^D_{U}(t,x,y)}{\varphi_U(x)\varphi_U(y)}\asymp \frac{\exp(-d_U(x,y)^2/ct)}{\sqrt{V_{\varphi_U^2}(x,\sqrt{t})}\sqrt{V_{\varphi_U^2}(y,\sqrt{t})}}. \tag{$\varphi_U^2$-HKE}\label{HKE}
    \end{align}
    The constant $c$ may differ in the lower and upper bounds. (Throughout this paper we write $A\lesssim B$ to denote $A\leq CB$ for some constant $C$ independent of $A$ and $B$, and we write $A\asymp B$ to denote $A\lesssim B\lesssim A$.)  
    \item There is an explicit expression $\Phi_U$, which we call the \textit{caricature function} of $\varphi_U$, such that $\varphi_U\asymp \Phi_U$.
    \item There are explicit matching upper and lower bounds for $\lambda(U)$.
    \item The eigenfunction $\varphi_U$ is stable under domain perturbations: if $V$ is a suitable larger domain containing $U$, then for any arbitrary domain $U'$ with $U\subseteq U'\subseteq V$, one has eigenfunction inequalities of the form $\varphi_U\lesssim \varphi_{U'}$ and  $\varphi_{U'}\lesssim\varphi_V$ whenever both sides are defined.
\end{enumerate}

Note that combining (\textup{\ref{HKE}}) with explicit estimates for $\varphi_U$ and $\lambda_U$, we obtain essentially explicit Dirichlet heat kernel estimates for $p^D_U(t,x,y)$.

\subsection*{Motivation: Euclidean boxes}
To motivate the above goals, we consider boxes in $\mathbb{R}^n$, which is one of the simplest examples for which Dirichlet Laplacian eigenfunctions and eigenvalues are known explicitly. By rotating and translating if necessary, we consider boxes that are symmetric with respect to each coordinate axis:
$$B=(-a_1,a_1)\times (-a_2,a_2)\times \cdots \times (-a_n,a_n),\hspace{0.1in}a_i\in (0,\infty).$$
The principal Dirichlet  eigenfunction and eigenvalue of $B$ are given by
\begin{align}
        \label{boxphilambda}\varphi_B(x_1,x_2,...,x_n)=\prod_{i=1}^{n}\frac{1}{\sqrt{a_i}}\cos\Big(\frac{\pi x_i}{2a_i}\Big),\hspace{0.1in} \lambda_B=\sum_{i=1}^{n}\Big(\frac{\pi}{2a_i}\Big)^2.
    \end{align}
    We show in Theorem \ref{boxHK} that $(B,d_B,\varphi_B^2 dx)$ satisfies volume doubling, Poincar\'{e} inequalities, and Dirichlet heat kernel estimates (\ref{HKE}), with all constants depending only on $n$. Moreover, regarding the behavior of $\varphi_B$ under domain perturbations, we have the following result.

\begin{Theo}
    \label{rectanglecomparison}
    Let $B_1=\prod_{i=1}^{n}(-a_i,a_i)$ and $B_2=\prod_{i=1}^{n}(-b_i,b_i)$ be two boxes in $\mathbb{R}^n$ with $B_1\subseteq B_2$. Let $U\subseteq \mathbb{R}^n$ denote an arbitrary domain with $B_1\subseteq U\subseteq B_2$, and, for constants $C_1>0,C_2> 1$, consider the following conditions:
    \begin{align}   
        \label{recthyp1}
        &\forall i\in \{1,2,...,n\},\hspace{0.1in}\frac{1}{a_i^2}-\frac{1}{b_i^2}\leq \frac{C_1}{\max _{i}b_i^2}
        \\ \label{recthyp2}  & \forall i\in \{1,2,...,n\},\hspace{0.1in} a_i\geq \frac{1}{C_2}b_i.
    \end{align}
    \begin{enumerate}
        \item Suppose (\ref{recthyp1}) holds. Then $\varphi_U(x)\lesssim \varphi_{B_2}(x)$ for all $x\in U$, where the implied constant depends only on $n$ and an upper bound on $C_1$.
        \item Fix $C_1>0$. For all $C_2>1$ sufficiently close to $1$ depending only on $C_1$ and $n$, if both (\ref{recthyp1}) and (\ref{recthyp2}) hold, then $\varphi_U(x)\gtrsim \varphi_{B_1}(x)$ for all $x\in B_1$, where the implied constant depends only on $n$ and $C_1$.
\end{enumerate}
\end{Theo}
Condition (\ref{recthyp1}) implies condition (\ref{recthyp2}) with $C_2=\sqrt{C_1+1}$. Hence both statements of Theorem \ref{rectanglecomparison} hold if we only assume  the condition (\ref{recthyp1}) for some sufficiently small $C_1>0$; in this case all implied constants depend only on $n$. Theorem \ref{rectanglecomparison} is proven in Section \ref{sectionbox}, and it implies the following corollary.
%In the below corollary, need a_i > b_i/2.
\begin{Cor}
    Consider the setting of Theorem \ref{rectanglecomparison}. Fix $C_1>0$. For all $C_2>1$ sufficiently close to $1$ depending only on $C_1$ and $n$, if both (\ref{recthyp1}) and (\ref{recthyp2}) hold, then
    $$\varphi_U(x)\asymp \varphi_{B_1}(x),$$
    for all $x=(x_1,x_2,...,x_n)$ with $|x_i|<a_i-(b_i-a_i)$ for all $i$. The implied constants depend only on $n$ and $C_1$.
\end{Cor}

Since the eigenfunctions and eigenvalues can be computed explicitly for boxes, the proofs of our results for boxes are relatively straightforward. Inspired by this simple example, we aim to prove similar results for other \Quote{thin} Euclidean domains.

\subsection*{Thin annular domains: annuli in $\mathbb{R}^n$}

We now consider the case when $U\subseteq \mathbb{R}^n$ is an \textit{annular domain} of the form $(a,b)\times U_0\subseteq (0,\infty)\times \mathbb{S}^{n-1}$, where the base $U_0\subseteq \mathbb{S}^{n-1}$ is an \textit{inner uniform domain} (Definition \ref{iudefn}) of the sphere. We focus on the regime when $b/a>1$ is bounded above. Otherwise, the domain $U$ is often (but not always) inner uniform, in which case results of Lierl and Saloff-Coste \cite{lierllsc} already give volume doubling, Poincar\'{e} inequalities, and Dirichlet heat kernel estimates (\ref{HKE}).    

We illustrate our main results for an annulus $A_{a,b}:=(a,b)\times \mathbb{S}^{n-1}$ in $\mathbb{R}^n$, and refer the reader to Theorem \ref{sphericalHKE}, as well as Section \ref{examples}, for other examples of annular domains that satisfy analogous results. 

\begin{Theo}
    \label{illustration} Let $A_{a,b}:=\{x\in \mathbb{R}^n:a<|x|<b\}$ and let $\varphi_{A_{a,b}}$ denote the principal Dirichlet Laplacian eigenfunction of $A_{a,b}$, with eigenvalue $\lambda(A_{a,b})$, and with $\|\varphi_{A_{a,b}}\|_{L^2(A_{a,b})}=1$. 
    \begin{enumerate}
        \item[(a)] $A_{a,b}$ is $\varphi_{A_{a,b}}^2$-volume doubling and satisfies $\varphi_{A_{a,b}}^2$-Poincar\'{e} inequalities, with constants depending only on the dimension $n$.
        \item[(b)] $p^D_{A_{a,b}}(t,x,y)$ satisfies two-sided Dirichlet heat kernel bounds \textup{(\ref{HKE})}, with constants depending only on $n$.
        \item[(c)] Fix $n$. There are caricature functions for $\varphi_{A_{a,b}}$, i.e. there are explicit expressions (\ref{p1}), (\ref{p2}), (\ref{p3}) comparable to $\varphi_{A_{a,b}}$. There are also explicit matching upper and lower bounds (\ref{annulieigval1}) for $\lambda(A_{a,b})$, uniformly for all $0<a<b<\infty$.
        \item[(d)] Set $\varepsilon=b/a-1>0$. For all sufficiently small $\varepsilon>0$ depending only on $n$, we have
        \begin{align}
            \label{illu1}A_{a,b}\subseteq U\subseteq A_{a-(b-a)\varepsilon^2,b+(b-a)\varepsilon^2}\hspace{0.1in}\Rightarrow\hspace{0.1in}\varphi_{A_{a,b}}\lesssim \varphi_U\lesssim \varphi_{a-(b-a)\varepsilon^2,b+(b-a)\varepsilon^2},
        \end{align}
        for any arbitrary domain $U\subseteq \mathbb{R}^n$. Consequently, (\ref{illu1}) implies 
        $$\varphi_U(x)\asymp  \varphi_{A_{a,b}}(x),\hspace{0.1in}\forall x\in A_{a+(b-a)\varepsilon^2,b-(b-a)\varepsilon^2}\subset A_{a,b}.$$
    \end{enumerate}
\end{Theo}

\begin{Rmk}
    \normalfont
(i) When $a=1$ and $b=1+\varepsilon$, Theorem \ref{illustration}(d) allows domain perturbations of $A_{1,1+\varepsilon}=\{x\in \mathbb{R}^n:1<|x|<1+\varepsilon\}$ in the radial direction on scale $\lesssim \varepsilon^3$, as $\varepsilon\to 0$. It can be shown that this is the best possible using our methods. The reason is as follows: to obtain Theorem \ref{illustration}(d), our argument relies on a uniform upper bound on the eigenvalue gap $\lambda(A_{1,1+\varepsilon})-\lambda(B)>0$ for two very thin annuli $A_{1,1+\varepsilon}\subseteq B$, see Lemma \ref{eiggap}. Informally speaking, Lemma \ref{eiggap} shows that if the annulus $B$ is an enlargement of the annulus $A$ at some scale $t>0$, then $\lambda(A)-\lambda(B)$ scales like $t/\varepsilon^3$ as $\varepsilon\to 0$, so we need $t=O(\varepsilon^3)$ to uniformly bound the eigenvalue gap from above. It is not known if the conclusion of Theorem \ref{illustration}(d) holds if the scale $\varepsilon^3$ is replaced by a larger one like $\varepsilon^2$.   

(ii) Both of the enclosing domains in (\ref{illu1}) have spherical base $\mathbb{S}^{n-1}$. When $U_0=\mathbb{S}^{n-1}$ is replaced with some other spherical domain with boundary, our results allow perturbations of $U_0$ as well; see Theorem \ref{perturbthm} for the most general form of Theorem \ref{illustration}(d).
\end{Rmk}

In Theorem \ref{illustration}(a)(b), our contribution is when $b/a\in (1,2]$ is sufficiently close to $1$, where the upper bound \Quote{$2$} is chosen for convenience. 
Indeed, the collection of annuli with $b/a\in (2,\infty)$ is uniformly inner uniform, and the results of \cite{lierllsc} apply. For the proofs of volume doubling (VD) and Poincar\'{e} inequalities (PI), we refer the reader to Sections \ref{VDsubsection} and \ref{PIsubsection} respectively. Once (VD) and (PI) are proven, the heat kernel bounds (\ref{HKE}) immediately follow from general results due to Grigor'yan, Saloff-Coste, and Sturm (\cite{grigoryan}, \cite{saloff-coste}, \cite{sturm}). In Theorem \ref{illustration}(a)(b), we again emphasize that the new analytic contribution is the proof of (VD) and (PI) for thin annular domains $U$ with uniform constants, all under the weighted measure $\varphi^2_{U}dx$; the general implication (VD)+(PI) $\Rightarrow$ (HKE) is already  well-established at the generality of metric measure spaces equipped with a Dirichlet form. We are specifically applying the implication (VD)+(PI) $\Rightarrow$ (HKE) for metric measure spaces to the space $(U,d_U,\varphi^2_U dx)$.

Theorem \ref{illustration}(c)(d) are new to the best of our knowledge; they are proven in Theorem \ref{annulusprofile}, Proposition \ref{annulieigval}, Theorem \ref{perturbthm}, and Corollary \ref{perturbcor}.

\subsection*{Historical background} 

There has been a long history of research on the interplay between volume doubling (VD), Poincar\'{e} inequalities (PI), and two-sided Gaussian-type heat kernel estimates (HKE). For Euclidean space equipped with a second-order uniformly elliptic divergence form operator, (HKE) was proven by Aronson (\cite{aronson1}, \cite{aronson2}). For complete Riemannian manifolds of nonnegative Ricci curvature, (VD) follows from the Bishop-Gromov volume comparison theorem, while (PI) was proven by Buser \cite{buser}. On the other hand, (HKE) for such manifolds were first obtained by Li and Yau \cite{liyau}. Soon after, Grigor'yan \cite{grigoryan} and Saloff-Coste \cite{saloff-coste} independently showed that for Riemannian manifolds, the conjunction of (VD) and (PI) is in fact equivalent to (HKE). Sturm \cite{sturm} then generalized this equivalence to metric measure spaces equipped with a symmetric Dirichlet form. Since then, many versions and variations of the equivalence $\text{(VD)+(PI)}\Leftrightarrow\text{(HKE)}$ have appeared in other settings, for example on graphs and fractal-like spaces.    

In \cite{gyryalsc}, Gyrya and Saloff-Coste considered the metric measure space $(U,d_U,h^2 d\mu)$, where $U\subseteq (X,\mu)$ is an unbounded inner uniform domain of some ambient space $(X,\mu)$ satisfying (HKE), and where $h>0$ is a harmonic function with Dirichlet boundary conditions on $U$. They considered a \textit{Doob $h$-transform} technique and proved (VD) and (PI) with respect to the weighted measure $h^2d\mu$, which implies the kernel $p^D_U(t,x,y)/h(x)h(y)$ on $(U,d_U, h^2 d\mu)$ satisfies (HKE). Following the framework of \cite{gyryalsc}, Lierl and Saloff-Coste \cite{lierllsc} proved analogous results for $(U,d_U,\varphi_U^2 d\mu)$, where $U$ is a bounded inner uniform domain. For bounded domains, $\varphi_U>0$ plays the role of $h$, and the results of \cite{lierllsc} imply that $e^{\lambda_U t}p^D_U(t,x,y)/\varphi_U(x)\varphi_U(y)$ satisfies (\ref{HKE}). More recently, the authors of this paper proved in \cite{chaolsc} stability estimates for $\varphi_U$ under domain perturbations of $U$. For certain inner uniform Euclidean domains $U$, the work \cite{chaolsc} describes explicit estimates (i.e. a \textit{caricature function}) for $\varphi_U$, and hence, explicit estimates for $p^D_U(t,x,y)$. We note that there is a very large literature on the eigenvalues, eigenfunctions, and heat kernels of the Dirichlet Laplacian; earlier works \cite{gyryalsc}, \cite{lierllsc}, \cite{chaolsc} as well as our work focus on the use of volume doubling and Poincar\'{e} inequalities in this context.

Note that \cite{gyryalsc}, \cite{lierllsc}, and \cite{chaolsc} all consider Dirichlet heat kernel or eigenfunction estimates on inner uniform domains. However, inner uniformity is sufficient, but often not necessary, to obtain such estimates. For instance, the collection of Euclidean boxes or annuli in $\mathbb{R}^n$ uniformly satisfy (\ref{HKE}) even though they are not uniformly inner uniform (Theorems \ref{illustration}, \ref{boxHK}). This suggests the equivalence $(\text{VD})+(\text{PI})\Leftrightarrow(\text{HKE})$ of (\cite{grigoryan}, \cite{saloff-coste}, \cite{sturm}) has the potential to be applied to many other families of \Quote{thin} domains that may fail to be inner uniform. For example in Figure \ref{HKEconj}, the families of domains $U_{\varepsilon,L}\subseteq \mathbb{R}^2$ parametrized by a small $\varepsilon>0$ and large $L>0$ (along their higher dimensional analogues) are all expected to satisfy (VD) and (PI) uniformly.

 \begin{figure}
    
  \centering
  \includegraphics[width=0.57\textwidth]{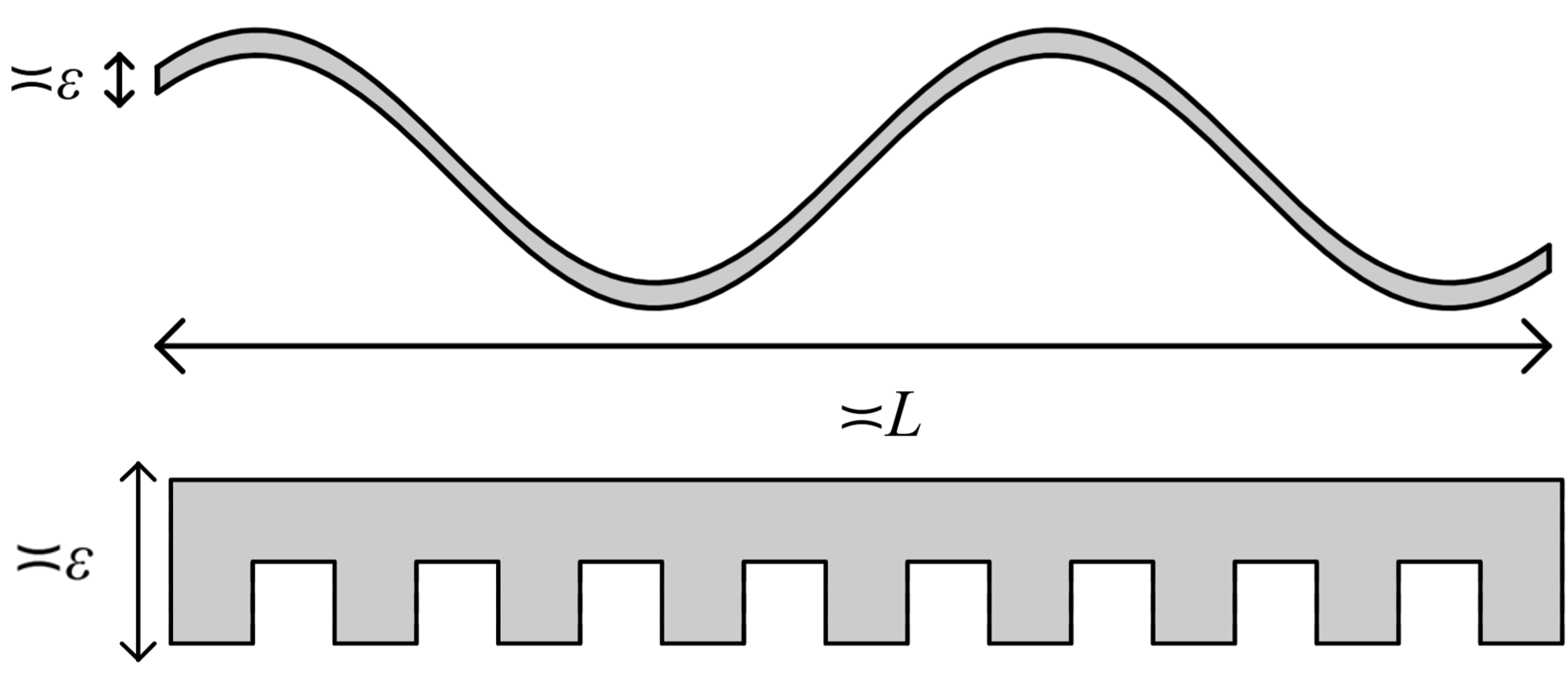}
  \caption{Thin bounded Euclidean domains $U_{\varepsilon,L}$ parametrized by length parameters $\varepsilon>0$ and $L>0$. These domains are expected to satisfy volume doubling and Poincar\'{e} inequalities with respect to the weighted measure $\varphi_{U_{\varepsilon,L}}^2 dx$, uniformly in the regime $\varepsilon\to 0$ and/or $L\to \infty$. }
  \label{HKEconj}
\end{figure}
Let us write $(\varphi_U^2\text{-VD})$ and $(\varphi_U^2\text{-PI})$ to denote (VD) and (PI) with respect to the weighted measure $\varphi_U^2 dx$; these notions are defined in Definitions \ref{vddefn} and \ref{pidefn}. At present, it seems out of reach to develop a general theory showing that large classes of \Quote{thin} domains (such as those in Figure \ref{HKEconj}) uniformly satisfy ($\varphi^2_U$-VD) and ($\varphi_U^2$-PI). Even if such a theory were developed, there would still be many other \Quote{thin} domains for which ($\varphi_U^2$-VD) and ($\varphi_U^2$-PI) fail, hence requiring different methods for obtaining sharp Dirichlet heat kernel bounds. For instance, in Section \ref{counterexample}, we prove that the family of all bounded convex domains in $\mathbb{R}^n$ fails to uniformly satisfy $(\varphi_U^2\text{-VD})$. 

\subsection*{Thin annular domains: volume doubling, Poincar\'{e} inequalities, and Dirichlet heat kernel estimates}

We thus restrict our attention to annular domains of the form $U=(a,b)\times U_0$ and try to prove results similar to Theorem \ref{illustration} for $U$. In particular, \textit{we attempt to understand sufficient conditions on the annular domain $U$ for \textup{(}$\varphi_U^2\textup{-VD)}$ and \textup{(}$\varphi_U^2$\textup{-PI)} to uniformly hold true.} We assume that the base $U_0\subseteq \mathbb{S}^{n-1}$ is an inner uniform domain. The notion of inner uniformity was independently introduced by Martio and Sarvas \cite{martiosarvas}, as well as Jones \cite{jones1}. Inner uniform domains form a wide collection of domains, including Lipschitz domains in $\mathbb{R}^n$ \cite[Page~73]{jones1}, non-tangentially accessible (NTA) domains, and even domains with irregular boundary points or fractal boundary. It follows from general results of \cite{lierllsc} that $(\varphi_{U_0}^2\text{-VD})$ and $(\varphi_{U_0}^2\text{-PI})$ hold uniformly over all inner uniform domains $U_0\subseteq \mathbb{S}^{n-1}$, which is equivalent to the Dirichlet heat kernel estimates ($\varphi_{U_0}^2$-HKE).    

In this paper, we focus on the regime where $b/a\in (1,2]$ and where $\text{diam}(U_0)$ is uniformly bounded below, see the left of Figure \ref{cover}. (The upper bound \Quote{$2$} in $b/a\leq 2$ is only chosen for convenience; see Remark \ref{thinremark}.) These conditions can be interpreted as a thin spherical shell with a spherical base $U_0\subseteq \mathbb{S}^{n-1}$ that is not too small.
Outside of this regime, $U$ can either be inner uniform or still be very thin (see the middle and right of Figure \ref{cover}, respectively); the first case is addressed by the much more general framework of \cite{lierllsc}, while we do not address the latter case in this paper.

\begin{figure}
  \centering
  \includegraphics[width=0.8\textwidth]{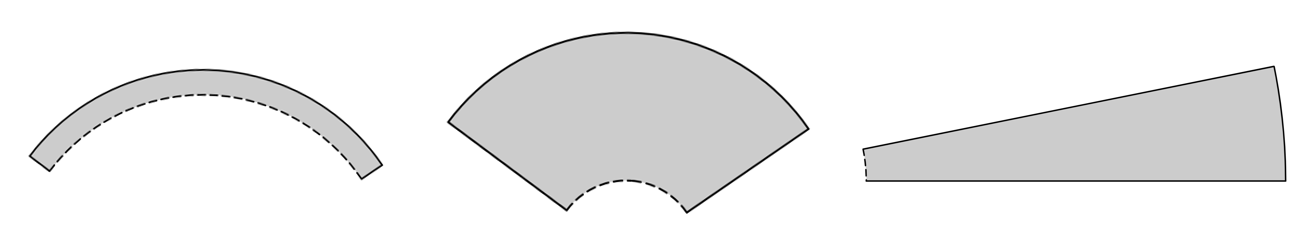}
  \caption{Depicted above are domains in $\mathbb{R}^2$ of the form $U=(1,b)\times U_0$, where the dashed line represents the angular base $U_0\subseteq \mathbb{S}^{1}$. Our work addresses $(\varphi_U^2\text{-VD})$, $(\varphi_U^2\text{-PI})$, and (\ref{HKE}) for \Quote{thin} domains in $\mathbb{R}^n$ generated by an inner uniform base $U_0$, with $\text{diam}(U_0)$ uniformly bounded below (for example, the left figure). The middle figure (with $b$ bounded away from $1$ and $\infty$, and with $\textup{diam}(U_0)\asymp 1$) is inner uniform, hence it is covered by results of \cite{lierllsc}; our methods apply in this case as well. When $U_0$ has small diameter (for example, the rightmost figure), $U$ is not inner uniform; our results do not apply uniformly in the regime where $\text{diam}(U_0)\to 0$. On the other hand, in the regime where $\text{diam}(U_0)$ is uniformly bounded below and $b\to \infty$, the rightmost figure is still inner uniform \cite[Page~129]{gyryalsc}, and the results of \cite{lierllsc} again apply; in this regime, the eigenfunction $\varphi_U$ has large-scale behavior that depends on $U_0$ (c.f. \cite[Proposition~6.31]{gyryalsc}).}
  \label{cover}
\end{figure}

For thin spherical shells $U=(a,b)\times U_0\subseteq \mathbb{R}^n$ with an inner uniform base $U_0\subseteq \mathbb{S}^{n-1}$, our results (Theorems \ref{vdthm} and \ref{PIthm}) show that they satisfy ($\varphi_U^2$-VD) and ($\varphi_{U}^2$-PI) uniformly, and thus imply the following result.
\begin{Theo}
\label{sphericalHKE}
Let $U_0\subseteq \mathbb{S}^{n-1}$ be a domain. Consider the domain $U=(a,b)\times U_0\subseteq \mathbb{R}^n$. Suppose that $b/a\in (1,2]$ and $\textup{diam}(U_0)$ is uniformly bounded below by $d>0$. Furthermore, suppose that $U_0$ is $(C_0,c_0)$-inner uniform, and $U$ is $(C_0,c_0)$-locally inner uniform up to scale $b-a$. Then the two-sided Dirichlet heat kernel estimates (\textup{\ref{HKE}}) hold for the domain $U$, with all implicit constants depending only on $n, C_0,c_0$ and $d$.
\end{Theo}

Note that the heat kernel estimates given by Theorem \ref{sphericalHKE} are sharp in the sense that the Doob $h$-transform $e^{\lambda_Ut}p^D_U(t,x,y)/\varphi_U(x)\varphi_U(y)$ satisfies matching two-sided Gaussian heat kernel bounds. However, this is much different from essentially explicit estimates on the Dirichlet heat kernel $p^D_U(t,x,y)$ itself, because the estimates on $p^D_U(t,x,y)$ are then expressed in terms of $\lambda_U$ and $\varphi_U$, and we often do not have explicit two-sided estimates for these quantities.  

We compare Theorem \ref{sphericalHKE} with the earlier Theorem \ref{illustration}. Note that the conclusion of Theorem \ref{sphericalHKE} covers Theorem \ref{illustration}(a)(b) because Theorem \ref{sphericalHKE} applies when $U_0=\mathbb{S}^{n-1}$. However, there is no analog of Theorem \ref{illustration}(c) because in general, we do not have explicit estimates for $\lambda(U_0)$ and $\varphi_{U_0}$; see Example \ref{vonkochexample} for a discussion of this issue. Regarding eigenfunction perturbations, Theorem \ref{perturbthm} in Section \ref{perturbannulisection} gives a result more general than Theorem \ref{illustration}(d); for conciseness we do not describe that result in Theorem \ref{sphericalHKE}.

\begin{Rmk}
    \normalfont
    It is natural to ask in Theorem \ref{sphericalHKE}, which assumptions are technical consequences of our methods, and which assumptions are expected to be truly necessary. 
    \begin{enumerate}
        \item[(i)] The assumption that $b/a\in (1,2]$ is needed to ensure that the radial part of $\varphi_U$ behaves like a \Quote{nice} tent-shaped function (imagine $\sin(x)$ for $x\in [0,\pi]$), see Lemma \ref{phiapprox} for a rigorous formulation. When $2<b/a\to \infty$, it is shown in Theorem \ref{annulusprofile} that the \Quote{peak} of the radial part of $\varphi_U$ moves toward the left endpoint of the interval $(a,b)$, meaning that the radial part is not uniformly \Quote{tent-shaped} in the regime $b/a\in (2,\infty)$, and our proof does not apply for large values of $b/a$. However, as long as $\text{diam}(U_0)$ is uniformly bounded below,  $(a,b)\times U_0$ will be uniformly inner uniform in the regime $b/a\in (2,\infty)$, and the conclusion of Theorem \ref{sphericalHKE} continues to hold because of general results of \cite{lierllsc} (in this case, the constant $d$ in Theorem \ref{sphericalHKE} is already controlled by the inner uniformity constants $C_0$ and $c_0$).
        \item[(ii)] We expect the uniform lower bound on $\text{diam}(U_0)$ to be necessary. In Section  \ref{counterexample}, we show that bounded domains $U\subseteq \mathbb{R}^2$ formed by angular sectors with angle $\theta$ are not uniformly ($\varphi_U^2$-VD) as $\theta\to 0$. Thus, we expect that in higher dimensions, shrinking the spherical base $U_0$ to a singleton should also cause ($\varphi_U^2$-VD) to fail to hold uniformly. Then ($\varphi_U^2$-HKE) cannot hold uniformly because of the equivalence ($\varphi_U^2$-VD) + ($\varphi_U^2$-PI) $\iff$ ($\varphi_U^2$-HKE).  
        \item[(iii)] We expect to be able to relax the requirement that $U_0\subseteq \mathbb{S}^{n-1}$ is inner uniform. Indeed, as will be seen later, our proof strategy is to begin with a spherical base $U_0\subseteq \mathbb{S}^{n-1}$ that satisfies ($\varphi_{U_0}^2$-VD) and ($\varphi_{U_0}^2$-PI), and then \Quote{transfer} these properties over to the domain $U=(a,b)\times U_0$. Our choice to work with inner uniform spherical domains $U_0$ is therefore in part because we already know, due to \cite{lierllsc}, that such $U_0$ satisfy weighted volume doubling and Poincar\'{e} inequalities with respect to $\varphi_{U_0}^2$. One may instead consider working with non-inner uniform families of spherical domains and attempt to obtain similar results. For example, one can consider a family of thin planar rectangular strips embedded onto the sphere $\mathbb{S}^{2}$ in some reasonable manner to be made precise. We do not pursue this here. 
        \item[(iv)] We also expect to be able to relax the assumption of local inner uniformity of the domain $U$. In essence, we only need this assumption to obtain the more essential property $(\varphi_U^2$-PI) on all sufficiently small metric balls (Theorem \ref{lierlPI}).
    \end{enumerate}
\end{Rmk}

Because the domains $U$ in Theorem \ref{sphericalHKE} satisfy ($\varphi_U^2$-VD) and ($\varphi_U^2$-PI) uniformly, we also obtain the following complementary heat kernel estimates.

\begin{Theo}
    \label{complementary}
    In the setting of Theorem \ref{sphericalHKE}, the Dirichlet heat kernel $p^D_U(t,x,y)$ of the domain $U=(a,b)\times U_0$ satisfies the following estimate: for all $t\geq \textup{diam}(U)^2$,
    $$
        \sup_{x,y\in U}\Bigg|\frac{e^{t\lambda_U }p^D_U(t,x,y)}{\varphi_U(x)\varphi_U(y)}-1\Bigg|\leq c_1e^{-c_2 t/\textup{diam}(U)^2},$$
        where the constants $c_1,c_2$ depend only on $n, C_0,c_0,$ and $d$.
\end{Theo}

Theorem \ref{complementary} follows from Theorem \ref{equilibrium} and Lemma \ref{diamlemma}, and it can be understood as follows. A Brownian motion inside $U$ killed upon hitting the boundary $\partial U$ survives for at least time $t$ with a probability decaying like $e^{-\lambda_U t}$. The term $e^{\lambda_U t}p^D_U(t,x,y)$ can then be interpreted as a rescaled version of the Dirichlet heat kernel $p^D_U(t,x,y)$, with rescaling to account for the decaying survival probability. Theorem \ref{complementary} asserts that $e^{\lambda_U t}p^D_U(t,x,y)$ equilibrates to its limit $\varphi_U(x)\varphi_U(y)$ after a time of order $\text{diam}(U)^2$, and the convergence then occurs at an exponential rate decaying in $t/\text{diam}(U)^2$. For thin annular domains $U$, the killing at the boundary occurs in time $t\asymp 1/\lambda_U\asymp (b-a)^2$ (see Lemma \ref{eigvaldecomp}), which is much faster than the time scale $t\asymp \text{diam}(U)^2$ for mixing in space. The heat kernel estimate of Theorem \ref{complementary} simultaneously captures both of these phenomena.

\begin{Rmk}
    \normalfont
    In later sections, our results will often require $\lambda(U_0)$, the first Dirichlet eigenvalue of $U_0$, to be uniformly bounded above. Because $U_0$ is inner uniform, $U_0$ contains a ball with radius comparable to $\text{diam}(U_0)$; this follows from \cite[Lemma~3.20]{gyryalsc} and Lemma \ref{diamlemma}. It follows that $\lambda(U_0)\lesssim 1/\text{diam}(U_0)^2$. Therefore, to ensure that $\lambda(U_0)$ is uniformly bounded above, it is sufficient to assume $\text{diam}(U_0)$ is uniformly bounded below. 
\end{Rmk}

\subsection*{Proof strategy}

Let us give a brief overview of the proof strategy used to obtain our main results, such as Theorem \ref{illustration}. By dilation arguments (Proposition \ref{scaling}) we may assume that $a=1$ and $b=1+\varepsilon$ for some $\varepsilon\in (0,1]$. Next, we observe in the parameter regime of interest, we have the \Quote{factorization}
\begin{align}
    \label{factorization}
    \varphi_U(r,\theta)\asymp \varphi_I(r)\cdot \varphi_{U_0}(\theta),
\end{align}
where $\varphi_I$ is the Dirichlet \textit{Laplacian} eigenfunction of the interval $I=(1,1+\varepsilon)$. The comparison (\ref{factorization}) is proven in Lemma \ref{lierlcompare} and uses a result of \cite{lierllsc}. In particular, (\ref{factorization}) implies that $\varphi_U^2 dx$ is comparable to a product measure $\varphi_I^2(r)\varphi_{U_0}^2(\theta)r^{n-1}drd\theta$, with the first and second factors satisfying ($\varphi_I^2$-VD) and ($\varphi_{U_0}^2$-VD), respectively. This enables us to prove ($\varphi_U^2$-VD) in Theorem \ref{vdthm}. Once the volume doubling property ($\varphi_U^2$-VD) is obtained, we then combine it with (i) small scale $\varphi_U^2$-Poincar\'{e} inequalities \cite{lierllsc}, and (ii) a discretization technique of Coulhon and Saloff-Coste \cite{coulsc}, to obtain ($\varphi_U^2$-PI) uniformly for all locations and scales (Theorem \ref{PIthm}). We conclude that $U$ satisfies ($\varphi_U^2$-VD) and ($\varphi_U^2$-PI), and hence satisfies (\ref{HKE}) by results of \cite{grigoryan}, \cite{saloff-coste}, \cite{sturm}. As a matter of fact, our methods also apply in the easier case when the weight $\varphi_U^2$ is replaced by a positive constant function on $U$. For example, when  $\varphi_U\equiv \text{vol}(U)^{-1/2}$ is the \text{first Neumann eigenfunction} of $U$, we obtain Neumann heat kernel estimates for thin annular domains; see Section \ref{neumann}.   

To then prove estimates for $\varphi_U$ under domain perturbations, we observe that ($\varphi_U^2$-VD) and ($\varphi_U^2$-PI) together imply Theorem \ref{equilibrium}, a large-time Dirichlet heat kernel estimate (see also Theorem \ref{complementary}). This estimate allows us to adapt the method of the authors' earlier work \cite{chaolsc} and prove the perturbation estimates as in Theorem \ref{illustration}(d).

\subsection*{Organization of paper}

In Section \ref{examples}, we give several more examples which follow from our main results. In Section \ref{notation}, we describe commonly used notations and conventions. In Section \ref{preliminaries}, we define some basic notions to be used in the proofs. Section \ref{mainresults} consists of the proofs of the main results. In Section \ref{neumann} of the Appendix, we describe the application of our methods to Neumann heat kernel estimates on thin annular domains. In Section \ref{counterexample} of the Appendix, we construct an explicit  counterexample to uniform $\varphi_U^2$-volume doubling, and formulate Conjecture \ref{conjVD}.  

\subsection{Examples}
\label{examples}
\begin{Exm}
    \label{harmonicexample}
    \normalfont
    On $\mathbb{R}^n$, choose a subset of $k$ coordinates for some $k\in \{1,2,...,n\}$. By relabeling, we may assume that the coordinates chosen are the first $k$ coordinates $x_1,...,x_k$. Let $P_k$ denote the restriction of the harmonic polynomial $x_1x_2\cdots x_k$ to $\mathbb{S}^{n-1}$. We have that $P_k$ is the principal Dirichlet eigenfunction of the domain $U_k:=\mathbb{S}^{n-1}\cap \{x_i>0\text{ for all }1\leq i\leq k\}$ with eigenvalue $k(k+n-2)$. This assertion follows from, for example, \cite[Exercise~10.19]{grig}.
     Note that  $P_k$ is not normalized to have $L^2(U_k)$-norm $1$, but in a fixed dimension $n$, we have $P_k\asymp P_k/\|P_k\|_{L^2(U_k)}$ since $\|P_k\|_{L^2(U_k)}\asymp 1$. 

    For instance, in $\mathbb{R}^2$, the principal Dirichlet eigenfunction of $\{x_1^2+x_2^2=1,x_1>0,x_2>0\}$ is given in polar coordinates by $\theta\mapsto \cos(\theta)\sin(\theta)$ with eigenvalue $4$. The eigenfunction and eigenvalue of $\{x_1^2+x_2^2=1,x_2>0\}$ are equal to $\theta\mapsto \sin(\theta)$ and $1$, respectively. 
    
    To express $P_k$ in terms of intrinsic geometric quantities associated to the sphere, define
    $E_i:=\mathbb{S}^{n-1}\cap \{x\in \mathbb{R}^n:x_i=0\}$
    to be the equator of $\mathbb{S}^{n-1}$ cut out by the hyperspace $\{x_i=0\}$, and note that 
    $$x_i\asymp \text{dist}(x,E_i),\enspace \forall x=(x_1,x_2,...,x_n)\in \mathbb{S}^{n-1},\enspace x_i>0.$$
    The distance to $E_i$ is taken with respect to the Riemannian distance on $\mathbb{S}^{n-1}$. It follows that 
    \begin{align}
        \label{PStheta}
        P_k(\theta)\asymp \prod_{i=1 }^{k}\text{dist}(\theta,E_i),\enspace \forall \theta\in U_k\subseteq \mathbb{S}^{n-1}.
    \end{align}
    The formula (\ref{PStheta}) is a caricature function for the first Dirichlet eigenfunction of $U_k$. It captures the fact that $P_k$ decays linearly as $\theta\in U_k$ approaches each $E_i$ and decays faster than linearly as $\theta$ approaches the intersection of multiple equators. 
    
    Our main results applied to the domains $U_k$ read as follows.

    \begin{Theo}
        \label{Smain1}
        Consider the domain $U_{a,b}:=(a,b)\times U_k\subseteq \mathbb{R}^n$. Let $\varphi_{U_k}$ (resp. $\varphi_{U_{a,b}}$) denote the principal Dirichlet eigenfunction of $U_k$ (resp. $U_{a,b}$), with $L^2$-norm $1$. 
        \begin{enumerate}
            \item[(a)] $U_{a,b}$ satisfies \textup{(}$\varphi_{U_{a,b}}^2$\textup{-VD)} and \textup{(}$\varphi_{U_{a,b}}^2$\textup{-PI)} with constants depending only on $n$.
            \item[(b)] $U=U_{a,b}$ satisfies \textup{(\ref{HKE})}, with constants depending only on $n$.
             \item[(c)] Suppose $b/a\in (1,2]$. We have, with all implied constants depending only on $n$,
        \begin{align}
            \nonumber
        \varphi_{U_{a,b}}(r,\theta)\asymp \frac{1}{a^{(n-1)/2}}\frac{\min\{r-a,b-r\}}{(b-a)^{3/2}}\prod_{i=1}^{k}\textup{dist}(\theta,E_i).
        \end{align}
        \item[(d)] With $\lambda(U_k)=k(k+n-2)$ and $C_1,C_2$ defined as in (\ref{C1C2}),
        \begin{align}
        \label{Uabeigval}\frac{C_1(n,b/a)}{(b-a)^2}+\frac{\lambda(U_k)}{b^2}\leq \lambda(U_{a,b})\leq \frac{C_2(n,b/a)}{(b-a)^2}+\frac{\lambda(U_k)}{a^2}.
        \end{align}
        \end{enumerate}
    \end{Theo}
    When $b/a\in [2,\infty)$, Theorem \ref{Smain1}(a)(b) already follows from \cite{lierllsc}. Our contribution to Theorem \ref{Smain1}(a)(b) is for thin annular domains where $b/a\in (1,2]$, in which case the desired results follow from Theorems \ref{vdthm} and \ref{PIthm}, respectively. Theorem \ref{Smain1}(c)(d) is obtained by combining (\ref{PStheta}), Lemma \ref{lierlcompare}, and Lemma \ref{eigvaldecomp}.
    
    With notation as above, we also define $$U_k(\eta):=\{\theta\in \mathbb{S}^{n-1}:\text{dist}(\theta,U_k)<\eta\}$$
    to be the $\eta$-neighborhood of the domain $U_k\subseteq \mathbb{S}^{n-1}$. Regarding the stability of $\varphi_{U_{a,b}}$ under domain perturbations, we have the following result, which allows perturbations in both the radial and angular directions. 

    \begin{Theo}
        \label{Smain2}
        Fix two nonnegative functions $\{f(\varepsilon)\}_{\varepsilon>0}$ and $\{g(\varepsilon)\}_{\varepsilon>0}$ with $\max\{f(\varepsilon),g(\varepsilon)\}\in [0,\varepsilon^3]$ for all $\varepsilon\in (0,1]$. 
        Put $a'=a-af(b/a-1)$ and $b'=b+ag(b/a-1)$ so that $a'\leq a<b\leq b'$.
        For all sufficiently small $\varepsilon>0$ and $\eta>0$ both depending only on $n$, if $U\subseteq \mathbb{R}^n$ is any domain with
        $$(a,b)\times U_k\subseteq U\subseteq (a',b')\times U_k(\eta),$$
        then, with all implied constants depending only on $n$,
        \begin{align*}
        \varphi_{(a,b)\times U_k}\lesssim \varphi_U \text{ on }(a,b)\times U_k\hspace{0.3in}\text{ and }\hspace{0.3in} \varphi_U\lesssim \varphi_{(a',b')\times U_k(\eta)} \text{ on }U.
        \end{align*}
    \end{Theo}

    In Theorem \ref{Smain2}, while we have $\varphi_{(a',b')\times U_k(\eta)}(r,\theta)\asymp \varphi_{(a',b')}(r)\varphi_{U_k(\eta)}(\theta)$ due to Lemma \ref{lierlcompare}, we do not have explicit estimates on $\varphi_{U_{k}(\eta)}(\theta)$. Thus in contrast to Theorem \ref{illustration}(d), we are unable to give a rigorous proof of the inequality $\varphi_U\asymp \varphi_{(a,b)\times U_k}$ on some compact subregion of $(a,b)\times U_k$, even though we expect it to hold. This illustrates the general principle that, for domains $A\subseteq U\subseteq B$, the inequalities $\varphi_A\lesssim \varphi_U\lesssim \varphi_B$ imply $\varphi_U\asymp \varphi_A$ only if we can prove the reverse inequality $\varphi_B\lesssim \varphi_A$ on an appropriate subdomain of $A$, which in turn requires explicit estimates on $\varphi_B$.  
\end{Exm}

\begin{Exm}
    \label{unitcircle}
    \normalfont On the unit circle $\mathbb{S}^{1}$, we have $\Delta_{\mathbb{S}^{1}}=\partial ^2/\partial \theta^2$. Thus, given any angular interval $(0,\theta_1)\subseteq \mathbb{S}^1$ where $\theta_1\in (0,2\pi)$, we can compute that
    $$\varphi_{(0,\theta_1)}(\theta)=\sqrt{\frac{2}{\theta_1}}\sin\Big(\frac{\pi \theta}{\theta_1}\Big),\hspace{0.1in}\lambda_{(0,\theta_1)}=\frac{\pi^2}{\theta_1^2}.$$
    In the case when $\theta_1\geq \theta'$ is uniformly bounded from below by a small $\theta'>0$, our main results apply to domains of the form $(a,b)\times (0,\theta_1)$, giving results completely analogous to Theorems \ref{illustration}, \ref{Smain1}, and \ref{Smain2}. All implied constants now depend only on $\theta'$. For brevity, we do not rewrite in full the analogues of those results. To illustrate Theorem \ref{Smain2} in this setting with an example, consider the annular domains
    $$A:=(1,1+\varepsilon)\times (0,3\pi/4)\subseteq \mathbb{R}^2\hspace{0.3in}\text{ and }\hspace{0.3in}B:=(1-\varepsilon^3,1+\varepsilon+\varepsilon^3)\times (-\eta,3\pi/4+\eta)\subseteq \mathbb{R}^2.$$
    By Lemma \ref{lierlcompare}, the Dirichlet Laplacian eigenfunction $\varphi_A$ of $A=(1,1+\varepsilon)\times (0,3\pi/4)$ (expressed in polar coordinates $(r,\theta)$) has boundary behavior comparable to that of the Euclidean box $(1,1+\varepsilon)\times (0,3\pi/4)$, and similarly for $B$.  
    For all sufficiently small $\varepsilon>0$ and $\eta>0$, for any arbitrary domain $U$ with $A\subseteq U\subseteq B$, we have the eigenfunction inequalities
    \begin{align}
        \label{AB1}
        & \varphi_A\lesssim \varphi_U\text{ on }A\hspace{0.3in}\text{ and }\hspace{0.3in} \varphi_U\lesssim \varphi_B \text{ on }U,
    \end{align}
    which implies that with implied constants \textit{independent} of $\varepsilon$ and $\eta$,
    \begin{align}
        \label{AB2}
        \varphi_U\asymp \varphi_A \text{ on } (1+\varepsilon^3,1+\varepsilon-\varepsilon^3)\times (\eta,3\pi/4-\eta)\subset A. 
    \end{align}

        \begin{figure}
        \centering \includegraphics[width=0.35\textwidth]{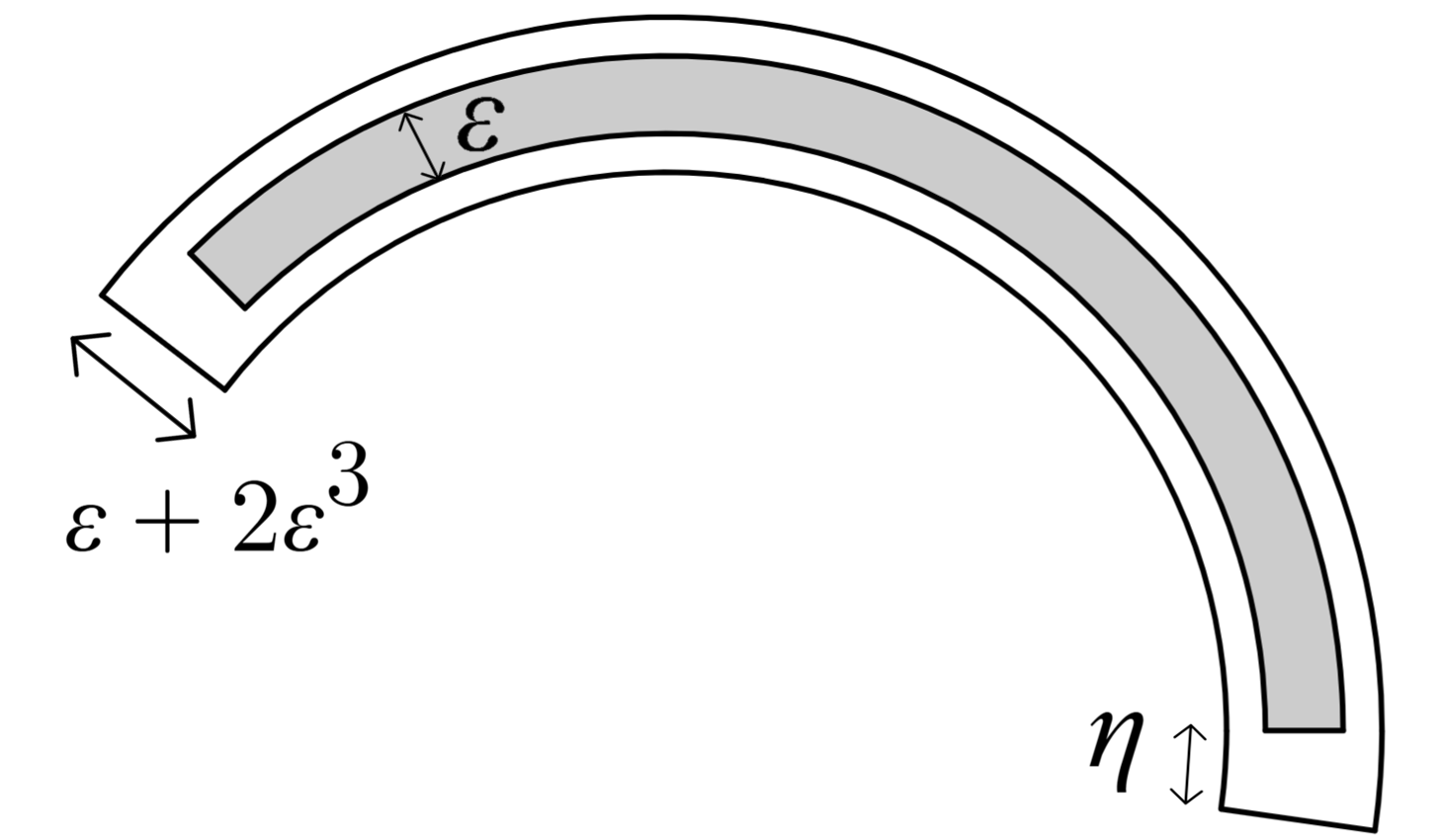}
  \caption{The annular domains $A$ and $B$ from Example \ref{unitcircle}. When these two domains become sufficiently close to each other, the behavior of $\varphi_U$ for any arbitrary domain $U$ with $A\subseteq U\subseteq B$ \Quote{stabilizes} in the sense of (\ref{AB1}) and (\ref{AB2}).}
  \label{3pi4pic}
\end{figure}
In the above example, the angle $3\pi/4$ may be replaced by any other angle $\theta \in (0,2\pi]$ that is uniformly bounded away from $0$. (If $\theta=2\pi$, then we perturb $A=(1,1+\varepsilon)\times (0,2\pi)$ only in the radial direction and take $B=(1-\varepsilon^3,1+\varepsilon+\varepsilon^3)\times (0,2\pi)$.)        
\end{Exm}

\begin{Exm} 
    \label{2sphere}
    \normalfont
    On the $2$-sphere $\mathbb{S}^2$, we may consider spherical triangles, which are determined by three great circles on the sphere. The class of spherical triangles $T\subseteq \mathbb{S}^2$ with angles $\alpha_1,\alpha_2,\alpha_3$ uniformly bounded below by $\alpha>0$ is uniformly inner uniform. We can consider a thin spherical shell $U=(a,b)\times T\subseteq  \mathbb{R}^3$ generated by a spherical triangle, with $b/a\in (1,2]$ and diameter $\text{diam}(T)>d>0$ uniformly bounded below. Then our results imply that ($\varphi_U^2$-VD), ($\varphi_U^2$-PI), and (\ref{HKE}) are uniformly satisfied, with all constants depending only on $\alpha,d>0$.   

    Moreover, by \cite[Example~2.3]{chaolsc}, we have the following caricature function for $\varphi_T$:
    \begin{align}
        \label{triangleprofile}
        \varphi_T\asymp \frac{d_1 d_2 d_3 (d_1+d_3)^{\pi/\alpha_2-2}(d_2+d_3)^{\pi/\alpha_1-2}(d_1+d_2)^{\pi/\alpha_3-2}}{\textup{diam}(T)^{\pi/\alpha_1+\pi/\alpha_2+\pi/\alpha_3-2}},
    \end{align}
    where $d_i$ denotes the geodesic distance to the side of $T$ opposite to $\alpha_i$. The implied constants in (\ref{triangleprofile}) depend only on $\alpha$. Thus (\ref{triangleprofile}) yields a caricature function for $\varphi_U$, namely
    $$\varphi_{U}(r,\theta)\asymp \frac{1}{a^{(n-1)/2}}\frac{\min\{r-a,b-r\}}{(b-a)^{3/2}}\varphi_T(\theta).$$
    For domain perturbations, an analog of Theorem \ref{Smain2} holds with $U_k$ replaced by $T$, and with $U_k(\eta)$ replaced by $T(\eta):=\{\theta\in \mathbb{S}^2:\text{dist}(\theta,T)<\eta\}$. 
    
    We discuss, with a specific example, a variation of Theorem \ref{Smain2} where $T(\eta)$ is replaced by some other spherical domain $T_{\eta}$ for which we have good estimates for its eigenfunction. Express $\mathbb{S}^2$ in spherical coordinates as $(\theta,\phi):=(\cos\theta \sin\phi,\sin\theta \sin\phi,\cos\phi)$ where $\theta\in [0,2\pi)=\mathbb{R}/2\pi \mathbb{Z}$ and $\phi\in [0,\pi)$. Consider the spherical triangles
    $$T=\{(\theta,\phi)\in \mathbb{S}^2:\theta \in (0,\theta_1),\phi \in (0,\pi/2)\},$$
    and 
    $$T_{\eta}:=\{(\theta,\phi)\in\mathbb{S}^2:\theta\in (-\eta,\theta_1),\phi\in (0,\pi/2)\}.$$
   See Figure \ref{sphericaltriangles}. The angles of $T$ are $\pi/2,\pi/2$ and $\theta_1$, while the angles of $T_{\eta}$ are $\pi/2,\pi/2,$ and $\theta_1+\eta$. By (\ref{triangleprofile}), we have explicit expressions comparable to $\varphi_T$ and $\varphi_{T_{\eta}}$. The more general Theorem \ref{perturbthm} implies that Theorem \ref{Smain2} holds with $U_k$ and $U_k(\eta)$ replaced by $T$ and $T_\eta$. Moreover, the conclusion of Theorem \ref{Smain2} implies that on any subregion $K\subseteq T$ at scale $\eta$ away from $\{\theta=0\}$ (see Figure \ref{sphericaltriangles}), we have 
   $$\varphi_U\asymp \varphi_{(a,b)\times T}\text{ on }(a+(a-a'),b-(b'-b))\times K.$$
    All of the implied constants depend only on a uniform lower bound on the angle $\theta_1$.

    \begin{figure}
    
  \centering
  \includegraphics[width=0.4\textwidth]{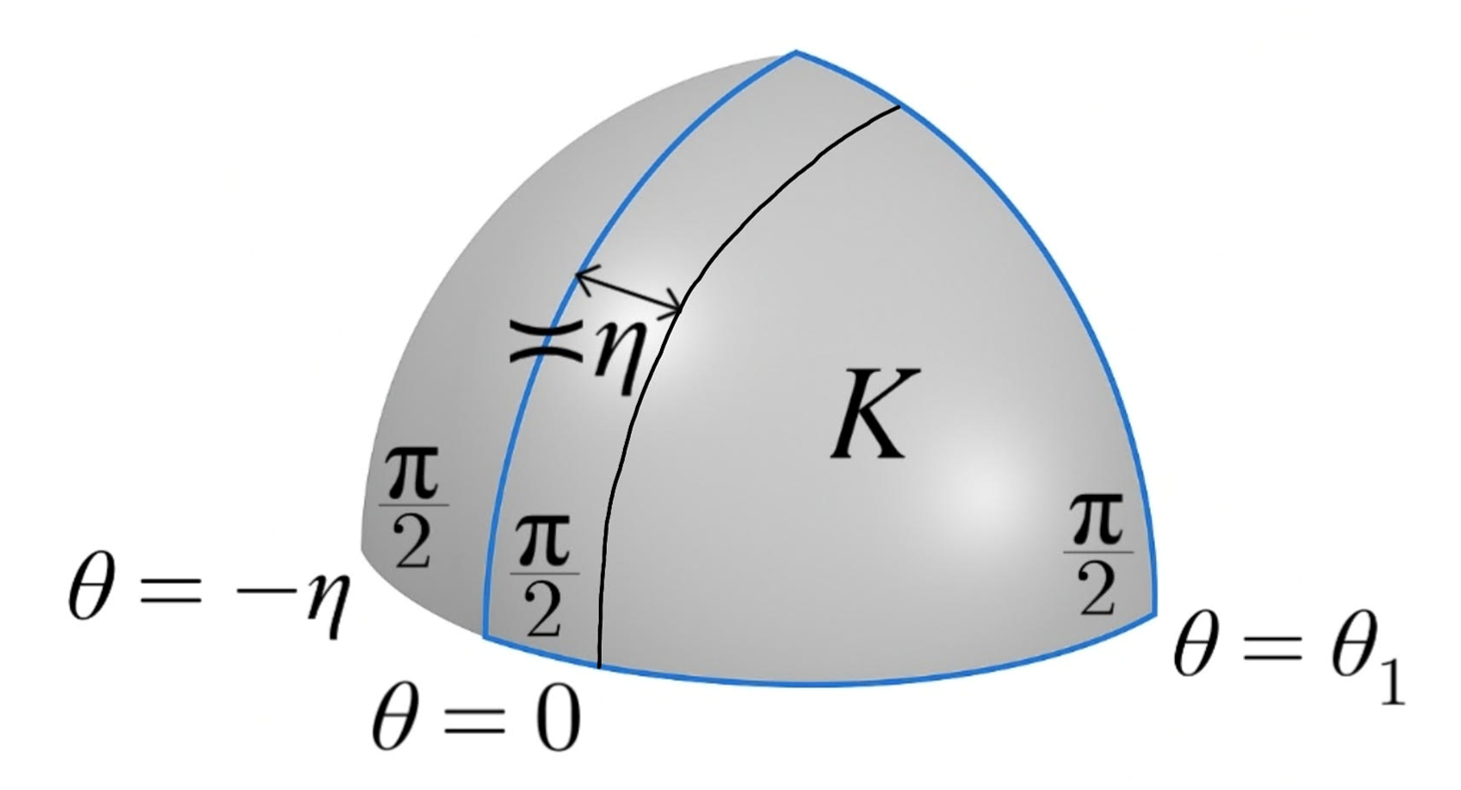}
  \caption{The spherical triangles $T$ and $T_{\eta}$, and a subregion $K\subseteq T$ at distance $\asymp \eta$ away from one side of $T$. The triangle $T_{\eta}$ is the largest spherical triangle depicted above, i.e. the entirety of the gray shape. The triangle $T$ is $T_{\eta}$  with the portion from $\theta=-\eta$ to $\theta =0$ removed.}
  \label{sphericaltriangles}
\end{figure}
\end{Exm}

\begin{Exm} 
    \label{vonkochexample}
    \normalfont
    A Von Koch snowflake in $\mathbb{R}^2$ is an inner uniform domain. This is a consequence of the more general fact that \textit{uniform} domains in $\mathbb{R}^n$, such as Euclidean balls, are preserved under quasiconformal mappings, see \cite[Page~384]{martiosarvas}. A direct proof of inner uniformity can be found in \cite[Proposition~6.30]{gyryalsc}.

    We can thus embed a copy of the Von Koch snowflake from $\mathbb{R}^2$  into $\mathbb{S}^2$ using some quasiconformal map, such that the image has diameter $\asymp 1$ and such that inner uniformity is preserved. Call the image of this map $\mathcal{V}\mathcal{K}\subseteq \mathbb{S}^2$. Our main results then hold in the case where $U=(a,b)\times \mathcal{V}\mathcal{K}\subseteq \mathbb{R}^3$, with $b/a\in (1,2]$. This yields results analogous to Theorems \ref{Smain1}(a)(b) and \ref{Smain2} for the domain $U$. In particular, a thin spherical shell $U$ with base $\mathcal{V}\mathcal{K}\subseteq \mathbb{S}^2$ satisfies (\ref{HKE}) uniformly.
    
    However, unlike the previous examples, we do not have satisfactory two-sided explicit estimates for $\varphi_{\mathcal{V}\mathcal{K}}$ or $\lambda(\mathcal{V}\mathcal{K})$. The best bound on $\lambda(U)$ our results yield is given by (\ref{Uabeigval}), with $U_k$ replaced by $\mathcal{V}\mathcal{K}$. We point to \cite{LP} and \cite{GL} for some work on estimates of Dirichlet eigenfunctions and eigenvalues on the snowflake in Euclidean space. 
\end{Exm}

\subsection{Notation}
    \label{notation}
$|x|$ is the Euclidean norm of $x\in \mathbb{R}^n$
\\
$A_{a,b}=\{x\in \mathbb{R}^n:a<|x|<b\}$
\\
$X\subseteq Y$ only means $X$ is a subset of $Y$, and does not imply whether or not $X$ can equal $Y$
\\
$dx,dy$, or $dz$ denote integration with respect to Lebesgue measure on $\mathbb{R}^n$ unless otherwise stated
\\
$|U|$ denotes the Lebesgue measure of $U\subseteq \mathbb{R}^n$
\\
$\sigma_{n-1}(U)$ denotes the surface measure of $U\subseteq \mathbb{S}^{n-1}$
\\
$\sigma_{n-1}(\mathbb{S}^{n-1})=2\pi^{n/2}/(\frac{n}{2}-1)!$ is the surface measure of $\mathbb{S}^{n-1}\subseteq \mathbb{R}^n$
\\
$\lambda_k(U)$ is the $k$th Dirichlet eigenvalue of the Dirichlet Laplacian on $U$, counting multiplicity
\\
$\lambda(U),\lambda_U,\lambda_1(U)$ all denote the smallest positive Dirichlet Laplacian eigenvalue of $U\subseteq \mathbb{R}^n$
\\
$A\lesssim B$ denotes $A\leq CB$ where $C$ is some constant depending only on important parameters
\\
$A\asymp B$ denotes $A\lesssim B$ and $B\lesssim A$ 
\\ $A_n\sim B_n$ if $A_n/B_n\to 1$ as $n\to\infty$

We will also adopt the convention that, in a chain of inequalities, the constants can change from line to line.

\subsection{Preliminaries}
\label{preliminaries}

In this section, we gather a number of well-known facts and results for later use. We begin with the following proposition, which records some basic properties of the eigenvalues, eigenfunctions, and heat kernel of the Dirichlet Laplacian; see for example \cite[Chapter~10]{grig}.

\begin{Prop}
    \label{lambda}
    Let $(M,\mathbf{g})$ be a Riemannian manifold with Riemannian measure $\mu$. Let $U$ be a precompact domain of $M$. Let $-\Delta_{\mathbf{g}}$ denote the Dirichlet Laplacian on $U$.
    \begin{enumerate}
        \item $-\Delta_{\mathbf{g}}$ admits a nondecreasing sequence of positive eigenvalues $\{\lambda_k\}_{k=1}^{\infty}=\{\lambda_k(U)\}_{k=1}^{\infty}$ with $\lim_{k\to\infty}\lambda_k=\infty$, corresponding to eigenfunctions $\varphi_k\in L^2(U,d\mu)$. The eigenvalues are counted with multiplicity.
        \item The Dirichlet heat kernel $p^D_U(t,x,y)$ can be expressed as
$$p^D_U(t,x,y)=\sum_{k=1}^{\infty}e^{-\lambda_k t}\varphi_k(x)\varphi_k(y),$$
where the series converges absolutely and uniformly on $(\varepsilon,\infty)\times U\times U$ for any $\varepsilon>0$.
        \item The first Dirichlet eigenvalue $\lambda(U):=\lambda_1(U)$ is given by the variational formula
        \begin{align}
        \label{rayleigh}
        \lambda(U)=\inf\Bigg\{\frac{\int_U |\nabla f|_{\mathbf{g}}^2 d\mu}{\int_U f^2 d\mu}:f\in W^{1,2}_0(U),f\neq 0\Bigg\}.
        \end{align}
        \item For any two precompact domains $U_1\subseteq U_2$ of $M$, their first Dirichlet eigenvalues satisfy domain monotonicity $\lambda(U_2)\leq \lambda(U_1)$.
    \end{enumerate}
\end{Prop}

Next, we introduce the geodesic distance $d_U$ on $U$, which enables us to consider the closure $\widetilde{U}$ with respect to $d_U$. This turns $U$ into a metric space with its own intrinsic geometry; balls in $(U,d_U)$ may differ significantly from those defined with respect to the Riemannian distance on the whole space.  

The definitions of the rest of this section are collected from the works \cite{gyryalsc} and \cite{lierllsc}.

\begin{Dfn}
    \normalfont
    Let $U$ be a domain of a Riemannian manifold $(M,\mathbf{g})$. For $x,y\in U$, we write $d_U(x,y)$ to denote the geodesic distance on $U$, i.e. the infimum of lengths of rectifiable curves in $U$ connecting $x$ to $y$.
\end{Dfn}

It is known that (see for example \cite[Proposition~2.13]{gyryalsc}) the geodesic distance on $U$ may be expressed as
\begin{align}
    \label{variationaldist}
    d_U(x,y)=\sup\{u(x)-u(y):u\in\mathcal{C}(U)\cap \mathcal{F}_{\text{loc}}(U),\enspace |\nabla u|_{\mathbf{g}}\leq 1\},
\end{align}
where 
\begin{align*}
    \mathcal{F}_{\text{loc}}(U):=\{u\in L^2_{\text{loc}}(U,d\mu):\forall\text{ compact }K\subseteq U\text{, }\exists\text{ }u^{\#}\in W^{1,2}(M)\text{ with }u=u^{\#}\big|_K\text{ a.e.}\}.
\end{align*}

\begin{Dfn}
    \normalfont
    \begin{enumerate}
        \item[(i)] We write $\widetilde{U}$ to denote the closure of $U$ with respect to the geodesic distance $d_U$. 
        \item[(ii)] Given an open set $V\subseteq U$, we let $V^{\#}$ be the largest open set in $\widetilde{U}$ which is contained in the closure of $V$ in $\widetilde{U}$ (with respect to $d_U$) and whose intersection with $U$ is $V$.
        \item[(iii)] We denote by $B_U(x,r)\subseteq U$ and $B_{\widetilde{U}}(x,r)\subseteq \widetilde{U}$ the open balls in $U$ and $\widetilde{U}$, respectively.
    \end{enumerate}
\end{Dfn}

\begin{Dfn}
    \label{iudefn}
    \normalfont
    Let $U$ be a domain of a Riemannian manifold $(M,\mathbf{g})$. We say that $U$ is $(C_0,c_0)$\textit{-inner uniform} if there are constants $C_0\in [1,\infty)$, $c_0\in (0,1]$ such that for any $x,y\in U$, there exists a continuous curve $\gamma:[0,1]\to U$ connecting $x$ to $y$ with length at most $C_0d_U(x,y)$, and such that for any $z\in \gamma([0,1])$,
    \begin{align}
        \label{iu1}
        d(z,\partial U)\geq c_0\frac{d_U(x,z)d_U(y,z)}{d_U(x,y)}.
    \end{align}
    We call such a curve a $(C_0,c_0)$\textit{-inner uniform curve}. 
\end{Dfn}
Inner uniform domains were first introduced by Martio and Sarvas \cite{martiosarvas} as well as Jones \cite{jones1}. More properties and examples of inner uniform domains may be found in \cite{martiosarvas}, \cite{jones1}, \cite{gyryalsc}, and \cite{lierllsc}.

Because $\max\{d_U(x,z),d_U(y,z)\}\geq d_U(x,y)/2$, it is easily verified that (\ref{iu1}) is equivalent to the condition that
\begin{align}
    \label{iu2}
    d(z,\partial U)\geq \frac{c_0}{2}\min\{d_U(x,z),d_U(y,z)\}.
\end{align}

We now recall a notion of \textit{local inner uniformity} from \cite[Definitions~3.7~and~3.10]{lierllsc}.
\begin{Dfn}
    \label{liu}
    \normalfont 
Fix $c_0\in (0,1]$ and $C_0\in [1,\infty)$. For a domain $U$ and $z\in \widetilde{U}$, we define
$$R(U,z):=\sup\Bigg\{R\in \Big(0,\frac{\textup{diam}_U}{8}\Big):\begin{matrix}
    \textup{any two points } x,y\in B_{\widetilde{U}}(z,8R)\text{ can be}
    \\ \textup{ connected by a } (C_0,c_0)\textup{-inner uniform curve in }U
\end{matrix}\Bigg\}.$$
For nonempty domains $V\subseteq U$, we say that $U$ is $(C_0,c_0)$\textit{-locally inner uniform near $V$} if for any point $z\in V^{\#}$, we have $R(U,z)>0$. We say that $U$ is $(C_0,c_0)$\textit{-locally inner uniform up to scale $R>0$ near $V$} if for any $z\in V^{\#}$, $R(U,z)\geq R$.
\end{Dfn}

Consider a domain $U$ that is locally inner uniform, and let $B_{\widetilde{U}}(z,r)$ be a small boundary ball with $z\in \widetilde{U}$. Locally on the boundary ball, let $h>0$ be a \Quote{locally harmonic} function with Dirichlet boundary conditions on $\widetilde{U}$. A rather general boundary Harnack principle from \cite[Proposition~5.10]{lierllsc} shows that the ratios of $\varphi_U$ and of $h$ are locally comparable. We first define the class of permissible harmonic functions in Definition \ref{FUV}, then in Proposition \ref{bhp}, we state a simpler version of the boundary Harnack principle due to \cite{lierllsc}. 

\begin{Dfn}
    \label{FUV}
    \normalfont
    For two domains $V\subseteq U$, we define
    $$\mathcal{F}^0_{\text{loc}}(U,V)=\Bigg\{f\in L^2_{\text{loc}}(U):\begin{matrix} \forall W\subseteq V \text{ relatively compact in }\widetilde{U} \text{ with }d_U(W,U\backslash V)>0,\\ \exists f^{\#}\in W^{1,2}_0(U)\text{ such that }f=f^{\#}\text{ a.e. on }W.
    \end{matrix}\Bigg\}$$
\end{Dfn}

\begin{Prop}
    \label{bhp}
    Let $(M^n,\mathbf{g})$ be a complete Riemannian manifold of nonnegative Ricci curvature. Let $U\subseteq (M,\mathbf{g})$ be a bounded domain and $V\subseteq U$. Suppose $U$ is $(C_0,c_0)$-locally inner uniform at scale $R$ near $V$. Let $h$ be any positive harmonic function on $V$ with $h\in \mathcal{F}^0_{\textup{loc}}(U,V)$. There exist constants $A,B\geq 1$ such that for any geodesic ball $B_U(x,r)$ with $r\in (0,R/A)$ and $B_U(x,Ar)\subseteq V^{\#}$, we have for all $y_1,y_2\in B_U(x,r)$,
    \begin{align*}
        \frac{\varphi_U(y_1)}{\varphi_U(y_2)}\leq B\frac{h(y_1)}{h(y_2)}.
    \end{align*}
    The constant $A$ depends only on $C_0,c_0$. The constant $B$ depends only on $C_0,c_0,n$ and an upper bound on $\lambda_UR^2$.
\end{Prop}

\section{Proof of main results}
    \label{mainresults}

\subsection{Laplacian eigenfunctions and eigenvalues on annular domains}
\label{structure}

Consider a bounded domain $U\subseteq \mathbb{R}^n$ of the form 
$$U=(a,b)\times U_0=\{(r,\theta)\in\mathbb{R}^{n}:r\in (a,b),\theta\in U_0\},$$ where $U_0\subseteq \mathbb{S}^{n-1}\subseteq \mathbb{R}^n$ is a region on the $(n-1)$-dimensional sphere. As in the title of Section \ref{structure} , we call such domains $U$ \textit{annular}  because when $U_0=\mathbb{S}^{n-1}$, $U$ coincides with an annulus. 

The goals of this section are to give detailed two-sided bounds for $\varphi_U$ and $\lambda_U$ where $U$ is an annular domain. We summarize the main results of this section. In Lemma \ref{lierlcompare}, we find a more explicit expression uniformly comparable to $\varphi_U$. In Proposition \ref{annulieigval}, we find two-sided bounds for $\lambda_U$, uniformly over all annuli $U\subseteq \mathbb{R}^n$. In Lemma \ref{eigvaldecomp}, we do the same for more general annular domains $U$. In Theorem \ref{annulusprofile}, we find a completely explicit expression uniformly comparable to $\varphi_U$ for any annulus $U$. 

In what follows, we drop the superscript or subscript \Quote{$U$} when convenient, to simplify notation when no confusion arises. The Dirichlet Laplacian eigenfunctions $\{\varphi_k^U\}_{k=1}^{\infty}$of $U$ can be taken to be of the form $\varphi^U_k(r,\theta)=f_k(r)g_k(\theta)$, where $g_k(\theta)$ are the Dirichlet Laplacian eigenfunctions of $U_0$ with respect to $\Delta_{\mathbb{S}^{n-1}}$. Note that in general, $g_k(\theta)$ need not be the eigenfunction corresponding to $\lambda_k(U_0)$ since the Dirichlet eigenvalues $\lambda_k$ are by convention taken to be in increasing order. However, $g_1(\theta)$ does correspond to $\lambda_1(U_0)$.

The Laplacian $\Delta$ on $\mathbb{R}^n$, in polar coordinates, is given by
\begin{align}
    \label{laplacian}
    \Delta=\frac{1}{r^{n-1}}\frac{\partial }{\partial r}\Bigg(r^{n-1}\frac{\partial }{\partial r}\Bigg)+\frac{1}{r^2}\Delta_{\mathbb{S}^{n-1}}=\frac{\partial ^2}{\partial r^2}+\frac{n-1}{r}\frac{\partial }{\partial r}+\frac{1}{r^2}\Delta_{\mathbb{S}^{n-1}}.
\end{align}
Thus, writing $\varphi_U(r,\theta)=f(r)g(\theta)$, computing $\Delta \varphi_U=\Delta f(r)g(\theta)$, and separating variables, we obtain
\begin{align}
    \nonumber & -\Delta_{\mathbb{S}^{n-1}} g=\lambda(U_0)g, \hspace{1.in}g|_{\partial U_0}=0
    \\  & \label{sepvar}
    -f''-\frac{n-1}{r}f'+\frac{\lambda(U_0)}{r^2}f=\lambda_Uf,\hspace{0.1in}f(a)=f(b)=0.
\end{align}
That is, $g(\theta)$ is simply the principal Dirichlet eigenfunction of $U_0\subseteq \mathbb{S}^{n-1}$. The radial part $f(r)$ of $\varphi_U$ is the principal Dirichlet eigenfunction of a one-dimensional Schr\"{o}dinger operator on an interval $(a,b)$; the eigenvalue of $f(r)$ equals $\lambda_U$.

The following proposition relates the principal Dirichlet eigenfunctions of an annular domain and its radial dilation.

\begin{Prop}
    \label{scaling}
    Let $U=(a,b)\times U_0\subseteq \mathbb{R}^n$ be a domain. For any $c>0$, put $cU:=(ca,cb)\times U_0\subseteq \mathbb{R}^n$. Then 
    \begin{align*}
        \varphi_{cU}(r,\theta)=\frac{1}{c^{n/2}}\varphi_U(r/c,\theta),\hspace{0.2in} \lambda(cU)=\frac{1}{c^2}\lambda(U).
    \end{align*}
\end{Prop}

The proof of Proposition \ref{scaling} is just checking that the proposed expression for $\varphi_{cU}$ satisfies (\ref{sepvar}), with $(a,b)$ replaced by $(ca,cb)$.

We aim to show in Lemma \ref{lierlcompare} that $f(r)$ is comparable to the Dirichlet Laplacian eigenfunction of the interval $(a,b)$. To accomplish this, we use a result of Lierl and Saloff-Coste \cite[Theorem~7.13]{lierllsc}. In a simplified form, their result roughly states the following. Consider two second-order differential operators $L_1=\Delta$ and $L_2=\Delta+L_2'$, where $\Delta$ is the Laplacian (or even a second-order uniformly elliptic operator), and where $L_2'$ are the lower order terms of $L_2$. Consider the operators $L_1$ and $L_2$ on a bounded inner uniform domain of $\mathbb{R}^n$ with Dirichlet boundary conditions. Then if $\varphi_{L_i}$ denotes the ground state eigenfunction of $L_i$, then $\varphi_{L_1}\asymp \varphi_{L_2}$. 

However, we describe one small difficulty that arises from trying to apply \cite[Theorem~7.13]{lierllsc}. While the differential operator on the left-hand side of (\ref{sepvar}) shares the same second-order part as the Laplacian $-\Delta f=-f''$ on a bounded interval, the reference measure for (\ref{sepvar}) is $r^{n-1}dr$, while it is just Lebesgue measure $dr$ for the Laplacian. Hence, \cite[Theorem~7.13]{lierllsc} does not apply to our situation directly. Before stating Lemma \ref{lierlcompare}, we first prove Lemma \ref{noweight}, which removes the weight $r^{n-1}$ from $r^{n-1}dr$. This in turn will allow us to apply \cite[Theorem~7.13]{lierllsc} in Lemma \ref{lierlcompare}.

\begin{Lemma}   
    \label{noweight}
    Let $I=(a,b)$ be a bounded interval with $0<a<b<\infty$. Let $U=I\times U_0\subseteq \mathbb{R}^n$ and let $\varphi_U(r,\theta)=f(r)g(\theta)$ be the principal Dirichlet Laplacian eigenfunction of $U$, normalized so that $\|\varphi_U\|_{L^2(U)}=1$. Then $f(r)=r^{-(n-1)/2}\widetilde{f}(r)$, where $\widetilde{f}(r)$ (normalized so that $\|\widetilde{f}\|_{L^2(I,dr)}=1$) is the eigenfunction associated to the eigenvalue
    \begin{align}   
        \label{noweight1}
        \lambda=\inf\Big\{\int_I\phi'(r)^2+\frac{\alpha+\lambda(U_0)}{r^2}\phi(r)^2dr:\phi\in C^{\infty}_c(I),\int_I \phi^2 dr=1\Big\},
    \end{align}
    where $\alpha =(n-3)(n-1)/4$. In fact, $\lambda=\lambda_U$.
\end{Lemma}

\begin{proof}
    The idea of the proof is to use a unitary map to transform $r^{n-1}dr$ to $dr$. On the interval $I:=(a,b)$, consider the unitary map $\mathcal{U}:L^2(I,r^{n-1}dr)\to L^2(I,dr)$ given by $\mathcal{U}f(r)=r^{(n-1)/2}f(r)$. Also define the operators
    \begin{align*}
        & \Delta_{\text{radial}}f=\frac{1}{r^{n-1}}\frac{\partial }{\partial r}\Big(r^{n-1}\frac{\partial f}{\partial r}\Big)=\frac{\partial ^2f}{\partial r^2}+\frac{n-1}{r}\frac{\partial f}{\partial r}
        \\ & \widetilde{\Delta}_{\text{radial}}f=\frac{\partial^2 f}{\partial r^2}-\frac{(n-1)(n-3)}{4r^2}f
    \end{align*}
    A direct calculation shows that $\widetilde{\Delta}_{\text{radial}}=\mathcal{U}\Delta_{\text{radial}}\mathcal{U}^{-1}$. This implies that $f$ solves (\ref{sepvar}) if and only if $\mathcal{U}f$ satisfies
    \begin{align}
        \label{Uf}
        -\widetilde{\Delta}_{\text{radial}}\mathcal{U}f+\frac{\lambda(U_0)}{r^2}\mathcal{U}f=\lambda_U\cdot \mathcal{U}f,\hspace{0.1in}\mathcal{U}f(a)=\mathcal{U}f(b)=0.
    \end{align}
    Hence $\mathcal{U}f$ is the principal Dirichlet eigenfunction of $-\widetilde{\Delta}_{\text{radial}}+\lambda(U_0)/r^2$ on $(a,b)$, which corresponds exactly to the eigenvalue $\lambda$ from (\ref{noweight1}). (We thus deduce from (\ref{Uf}) that the expression in (\ref{noweight1}) in fact equals $\lambda_U$.) Defining $\widetilde{f}:=\mathcal{U}f$ so that $f=\mathcal{U}^{-1}\widetilde{f}$, the proof is complete.  
\end{proof}

In Lemma \ref{lierlcompare}, we argue that in the regime where $b/a\in (1,2]$, the radial part of $\varphi_U(r,\theta)$ is comparable to the principal Dirichlet Laplacian eigenfunction of $(a,b)$, 

\begin{Lemma}
    \label{phiapprox}
    \label{lierlcompare} Suppose $b/a\in (1,2]$. Let $U\subseteq \mathbb{R}^n$ be a domain of the form $U=(a,b)\times U_0\subseteq (0,\infty)\times \mathbb{S}^{n-1}$ and let $\varphi_U(r,\theta)$ be the principal Dirichlet Laplacian eigenfunction of $U$ normalized so that $\|\varphi_U\|_{L^2(U)}=1$. Then $$\varphi_U(r,\theta)\asymp \frac{1}{a^{(n-1)/2}}\sqrt{\frac{2}{b-a}}\cos\Big(\frac{\pi (r-\frac{a+b}{2})}{b-a}\Big)g(\theta)\asymp  \frac{1}{a^{(n-1)/2}}\frac{\min\{r-a,b-r\}}{(b-a)^{3/2}}g(\theta),$$
     where $g:U_0\to [0,\infty)$ is the principal Dirichlet eigenfunction of $\Delta_{\mathbb{S}^{n-1}}$ on $U_0$ normalized so that $\|g\|_{L^2(U_0)}=1$. The implied constants depend only on the dimension $n$ and a uniform upper bound on $\lambda(U_0)$.
\end{Lemma}

\begin{proof} First assume $a=1$ and $b\in (1,2]$. We may express $\varphi_U(r,\theta)=f(r)g(\theta)$, where $f,g$ are chosen so that $\int_1^{b}f(r)^2r^{n-1}dr =1$ and $\int_{U_0}g(\theta)^2d\sigma_{n-1}(\theta)=1$.

On the Sobolev space $W_0^{1,2}(a,b)$, consider the Dirichlet forms
\begin{align}
    \label{bilinear}
    \mathcal{E}_1(f,g)=\int_{a}^{b}f'(r)g'(r)dr,\hspace{0.2in}\mathcal{E}_2(f,g)=\int_a^{b}f'(r)g'(r)+\frac{\alpha+\lambda(U_0)}{r^2}f(r)g(r)dr,
\end{align}
where $\alpha=(n-3)(n-1)/4$.
(In (\ref{bilinear}), we abuse notation and write $f,g$ to denote generic functions $f,g\in W_0^{1,2}(1,b)$ even though $f$ and $g$ were already defined.) For $i=1,2$, we also let $\widetilde{f}_i$ denote the first eigenfunction of the differential operator associated with $\mathcal{E}_i(\cdot,\cdot)$. In particular, $\widetilde{f}_1(r)$ is the principal Dirichlet Laplacian eigenfunction of the interval $(a,b)$, and it is given explicitly by
\begin{align*}
    \widetilde{f}_1(r):=\sqrt{\frac{2}{b-a}}\cos\Big(\frac{\pi(r-\frac{a+b}{2})}{b-a}\Big)\asymp \frac{\min\{r-a,b-r\}}{(b-a)^{3/2}}.\end{align*}
On the other hand, by Lemma \ref{noweight}, $f(r)=r^{-(n-1)/2}\widetilde{f}_2(r)$, so $f(r)\asymp \widetilde{f}_2(r)$ because $r\in [1,2]$. 

We now apply \cite[Theorem~7.13]{lierllsc}, which implies the principal Dirichlet eigenfunctions of the forms $\mathcal{E}_1$ and $\mathcal{E}_2$ (each normalized so that $\int_{a}^{b}\widetilde{f}_i(r)^2dr =1$) are comparable, that is,
\begin{align}
    \label{1.02}
    \frac{1}{A}\cdot \widetilde{f}_1(r)\leq \widetilde{f}_2(r)\leq A\cdot \widetilde{f}_1(r),\hspace{0.1in} r\in [a,b],
\end{align}
where $A\geq 1$ is a constant depending only on an upper bound on
$$\frac{\alpha+\lambda(U_0)}{r^2}\cdot \text{diam}((1,b))^2\leq (\alpha+\lambda(U_0))\cdot \Big(\frac{b}{a}-1\Big)^2.$$
(This claim about the constant $A$ is just by carefully keeping track of the constants and definitions in \cite{lierllsc}. Indeed, in \cite[Theorem~7.13]{lierllsc}, we can take the parameters $C_8(\mathcal{E})$ and $R$ to be comparable to $(\alpha+\lambda(U_0))r^{-2}$ and $\text{diam}((1,b))$, respectively.) By assumption, $1\leq b/a\leq 2$. Therefore, we can assume $A$ depends only on the dimension $n$ and an upper bound on $\lambda(U_0)$. 

Using (\ref{1.02}) and the earlier observation that $f(r)\asymp \widetilde{f}_2(r)$, we get 
\begin{align}
    \label{1.03}\varphi_U(r,\theta)=f(r)g(\theta)\asymp_{n,A}\sqrt{\frac{2}{b-1}}\cos\Big(\frac{\pi(r-\frac{1+b}{2})}{b-1}\Big)g(\theta)\asymp\frac{\min\{r-1,b-r\}}{(b-1)^{3/2}}g(\theta).
\end{align}
This proves the desired result when $a=1$. For arbitrary $(a,b)$ with $b/a\in (1,2]$, we use scaling properties of Laplacian eigenfunctions (Proposition \ref{scaling}) to get
$$\varphi_U(r,\theta)=\varphi_{(a,b)\times U_0}(r,\theta)=\varphi_{a(1,b/a)\times U_0}(r,\theta)=\frac{1}{a^{n/2}}\varphi_{(1,b/a)\times U_0}(r/a,\theta),$$
and we finish the proof by using (\ref{1.03}), substituting $b/a,r/a$ for $b,r$.
\end{proof}

Next, in Proposition \ref{annulieigval}, we give explicit upper and lower bounds for the first Dirichlet eigenvalue $\lambda(A)$ for an annulus $A_{a,b}=\{x\in \mathbb{R}^n:a<|x|<b\}$, uniformly for any choices of $0<a<b<\infty$. In particular, Proposition \ref{annulieigval} shows that $\lambda(A_{a,b})\asymp_n (b-a)^{-2}$, and moreover
$$\lambda(A_{a,b})\sim \frac{\pi^2}{(b-a)^2},\enspace \text{ as }\enspace b/a\to 1.$$
Proposition \ref{annulieigval} will not be used outside of this subsection, but may be of some independent interest; it is also mentioned in Theorem \ref{illustration}(c). 

\begin{Prop}
    \label{annulieigval}
    Let $\lambda(A_{a,b})$ be the first Dirichlet Laplacian eigenvalue of $A_{a,b}\subseteq \mathbb{R}^n$. Let
    \begin{align}
        \label{C1C2}
        &C_1(n,x):=\begin{cases}
             \max\Big\{\frac{n\pi^2}{4}\Big(1-\frac{1}{x}\Big)^2,\pi^2+\frac{(n-1)(n-3)}{4}\Big(1-\frac{1}{x}\Big)^2\Big\}, & n\geq 3
             \\  \max\Big\{\frac{n\pi^2}{4}\Big(1-\frac{1}{x}\Big)^2,\pi^2+\frac{(n-1)(n-3)}{4}(x-1)^2\Big\} & n = 2
        \end{cases}
        \\ \nonumber & C_2(n,x):= \begin{cases}
            \min\Big\{(n\pi)^2,\pi^2+\frac{(n-1)(n-3)}{4}(x-1)^2\Big\}, & n \geq 3
            \\   \min\Big\{(n\pi)^2,\pi^2+\frac{(n-1)(n-3)}{4}\Big(1-\frac{1}{x}\Big)^2\Big\}, & n = 2.
        \end{cases}
    \end{align}
    For any dimension $n\geq 2$ and $0<a<b<\infty$,
    \begin{align}
        \label{annulieigval1} \frac{C_1(n,b/a)}{(b-a)^2}\leq  \lambda(A_{a,b})\leq \frac{C_2(n,b/a)}{(b-a)^2}. 
    \end{align}
\end{Prop}

\begin{proof}
   We first use domain monotonicity of the Dirichlet eigenvalue (Proposition \ref{lambda}) to get upper and lower bounds on $\lambda(A_{a,b})$. Let $p=((a+b)/2,0,...,0)\in A_{a,b}$, and note that the Euclidean ball $B(p,(b-a)/2)\subseteq A_{a,b}$. Since any Euclidean ball of radius $R>0$ contains an $n$-dimensional cube $C$ of side length $2R/\sqrt{n}$, domain monotonicity of Dirichlet eigenvalue implies
   \begin{align}
        \label{annulieigval2}
       \lambda(A_{a,b})\leq \lambda(C)=n\cdot \frac{\pi^2}{(2R/\sqrt{n})^2}=\frac{n^2\pi^2}{(b-a)^2}.
   \end{align}
   Since $A_{a,b}$ is also contained inside a cube of side length $2b$, we get
   \begin{align}
       \label{annulieigval3}
       \lambda(A_{a,b})\geq \frac{n\pi^2}{(2b)^2}.
   \end{align}
    To sharpen the already obtained upper and lower bounds, we argue using Rayleigh quotients. First, we note that because $\varphi_{A_{a,b}}$ is a radial function, we have
    $$\lambda(A_{a,b})=\inf_{f\in W^{1,2}_0(A_{a,b})}\frac{\int_{A_{a,b}}|\nabla f|^2 dx}{\int_{A_{a,b}} f^2 dx}=\inf\Big\{\frac{\int_{A_{a,b}}|\nabla f|^2 dx}{\int_{A_{a,b}} f^2 dx}:f\text{ radial}\Big\}=\inf_{f\in C^{\infty}_c(a,b)}\frac{\int_a^{b}f'(r)^2r^{n-1}dr}{\int_a^{b}f(r)^2r^{n-1}dr}.$$
    Consider $g(r):=r^{(n-1)/2}f(r)$. A calculation similar to the proof of Lemma \ref{noweight} shows that for any sufficiently smooth function $f$ such that $f(a)=f(b)=0$, we have
    $$\frac{\int_a^{b}f'(r)^2 r^{n-1}dr}{\int_a^b f(r)^2 r^{n-1}dr}=\frac{\int_{a}^{b}g'(r)^2 +\frac{(n-3)(n-1)}{4r^2}g(r)^2 dr}{\int_{a}^{b}g(r)^2dr}.$$
    It follows that when $n\geq 3$ (so that $(n-3)(n-1)\geq 0$),
    \begin{align*}
         \frac{(n-3)(n-1)}{4b^2}+\inf_{g\in C^{\infty}_c(a,b)}\frac{\int_a^{b}g'(r)^2 dr}{\int_{a}^b g(r)^2 dr}\leq \lambda(A_{a,b})\leq \inf_{g\in C^{\infty}_c(a,b)}\frac{\int_a^{b}g'(r)^2 dr}{\int_{a}^b g(r)^2 dr}+\frac{(n-3)(n-1)}{4a^2}.
    \end{align*}
    Since the above infimum is just the Dirichlet Laplacian eigenvalue of the interval $(a,b)$, we obtain
    \begin{align}
        \label{annulieigval4}
         \frac{(n-3)(n-1)}{4b^2}+\frac{\pi^2}{(b-a)^2}\leq \lambda(A_{a,b})\leq\frac{\pi^2}{(b-a)^2}+\frac{(n-3)(n-1)}{4a^2}.
    \end{align}
    Combining (\ref{annulieigval2}), (\ref{annulieigval3}), (\ref{annulieigval4}) and rearranging gives the desired conclusion. The case for $n=2$ is similar.
\end{proof}

For a general annular domain $(a,b)\times U_0$, we also prove two-sided bounds on $\lambda(U)$ in terms of $\lambda(A_{a,b})$ and $\lambda(U_0)$.

\begin{Lemma} 
    \label{eigvaldecomp}
Let $U_0\subseteq \mathbb{S}^{n-1}$ be a domain. In polar coordinates, consider the domain $U=(a,b)\times U_0\subseteq \mathbb{R}^n$. We have
\begin{align}
    \label{eigvaldecomp1}
\lambda(A_{a,b})+\frac{\lambda(U_0)}{b^2}\leq \lambda(U)\leq \lambda(A_{a,b})+\frac{\lambda(U_0)}{a^2}.
\end{align}
\end{Lemma}

\begin{proof}
    Let $I:=(a,b)\subseteq \mathbb{R}$ and let $\alpha:=(n-1)(n-3)/4$. By Lemma \ref{noweight}, we have
    \begin{align}
        \label{intervaleigval}
       \lambda_U=\lambda(I\times U_0)=\inf_{ f\in C^{\infty}_c(I)}\Bigg\{\frac{\int_{a}^{b} f'(r)^2dr}{\int_a^{b}f(r)^2dr}+(\alpha+\lambda(U_0))\frac{\int_a^{b}\frac{f(r)^2}{r^2}dr}{\int_a^{b}f(r)^2dr}\Bigg\}
    \end{align}
    and, since $\lambda(\mathbb{S}^{n-1})=0$,
     \begin{align}
        \label{intervaleigval1}
       \lambda(A_{a,b})=\lambda(I\times \mathbb{S}^{n-1})=\inf_{ f\in C^{\infty}_c(I)}\Bigg\{\frac{\int_{a}^{b} f'(r)^2dr}{\int_a^{b}f(r)^2dr}+\alpha\frac{\int_a^{b}\frac{f(r)^2}{r^2}dr}{\int_a^{b}f(r)^2dr}\Bigg\}.
    \end{align}
    (\ref{intervaleigval}) and (\ref{intervaleigval1}) together imply (\ref{eigvaldecomp1}).
\end{proof}

Next, we begin to prepare the proof of Theorem \ref{annulusprofile}, which gives explicit expressions comparable to $\varphi_U$ for any annulus $U\subseteq \mathbb{R}^n$. We precede Theorem \ref{annulusprofile} by several preparatory lemmas. 

\begin{Lemma}
    \label{maxeig}
    Let $U\subseteq \mathbb{R}^n$ be any domain with $\lambda(U)>0$, with the corresponding Dirichlet Laplacian eigenfunction $\varphi_U$ normalized so that $\|\varphi_U\|_{L^2(U)}=1$. Then
    $$\frac{1}{|U|}\leq \|\varphi_U\|_{L^{\infty}(U)}^2\lesssim \lambda(U)^{n/2},$$
    where the implied constant depends only on $n$.
\end{Lemma}

\begin{proof}
    The desired statement follows from, for example, \cite[Lemma~4.2]{chaolsc} and its proof.
\end{proof}

\begin{Lemma} \label{phi}Let $U\subseteq \mathbb{R}^n$ be a domain with $\lambda(U)>0$, with corresponding Dirichlet Laplacian eigenfunction $\varphi_U$. Let $B(z,r)$ be a Euclidean ball contained inside $U$. Then for all $x,y\in B(z,r/2)$, we have 
$$\varphi_U(x)\lesssim \varphi_U(y),$$
where the implied constant depends only on $n$ and an upper bound on $\lambda(U)r^2$.
\end{Lemma}

\begin{proof} Define $u:(0,r^2)\times B(z,r)\to \mathbb{R}$ by  $u(t,x)=e^{-\lambda_U t}\varphi_U(x)$. Then $u(t,x)$ is a nonnegative weak solution to the heat equation $\partial_t u=\Delta u$. Thus, the parabolic Harnack inequality implies that for some constant $H\geq 1$ depending only on $n$,
$$\sup_{Q_-}u(t,x)\leq H \inf_{Q_+}u(t,x),$$
where $Q_+=(r^2/4,r^2/2)\times B(z,r/2)$ and $Q_-=(3r^2/4,r^2)\times B(z,r/2)$. Thus,
$$\varphi_U(x)\leq H\exp(\lambda_U r^2/4)\varphi_U(y),$$ and the lemma follows.
\end{proof}

We now state and prove Theorem \ref{annulusprofile}.

\begin{Theo}
    \label{annulusprofile}
    Consider an annulus $A_{a,b}=\{x\in \mathbb{R}^n:a<|x|<b\}$ in $\mathbb{R}^n$. Let $\varphi_{A_{a,b}}$ denote the principal Dirichlet Laplacian eigenfunction of $A_{a,b}$, normalized so that $\|\varphi_{A_{a,b}}\|_{L^2(A_{a,b})}=1$.
    \begin{enumerate}
        \item (Thin annuli) Suppose $a/b\in (1/2,1)$. Then 
        \begin{align}
            \label{p1}
            \varphi_{A_{a,b}}(x)\asymp \frac{1}{a^{n/2+1}}\cdot \frac{\min\{|x|-a,b-|x|\}}{(b/a-1)^{3/2}}.
        \end{align}
        \item (Non-thin annuli) Suppose $a/b\in (0,1/2]$. Then if $n\geq 3$, 
        \begin{align}  
            \label{p2}
            \varphi_{A_{a,b}}(x)\asymp \frac{1}{b^{n/2}}\Big(1-\Big(\frac{a}{|x|}\Big)^{n-2}\Big)\Big(1-\frac{|x|}{b}\Big)
        \end{align}
        while if $n=2$, 
        \begin{align}
            \label{p3}
            \varphi_{A_{a,b}}(x)\asymp \frac{1}{b}\cdot \frac{1}{\log\Big(1+\frac{b}{4a}\Big)}\log\Big(\frac{|x|}{a}\Big)\Big(1-\frac{|x|}{b}\Big).
        \end{align}
    \end{enumerate}
    The implied constants in (\ref{p1}) and (\ref{p2}) depend only on $n$, while the implied constants in (\ref{p3}) are absolute.
\end{Theo}

\begin{Rmk}
    \normalfont
    \label{thinremark}
    In Theorem \ref{annulusprofile}, the choice of the constant $1/2$ to distinguish between \Quote{thin annuli} and \Quote{non-thin annuli} is essentially arbitrary, and can be replaced by any fixed number in the interval $(0,1)$. 
\end{Rmk}

\begin{proof}
    The case when $a/b\in (1/2,1)$ is a consequence of Lemma \ref{lierlcompare}. Thus we prove the result when $a/b\in (0,1/2)$. By a scaling argument to be mentioned at the end of the proof, we may assume for now that $b=1$. Observe that the annulus $A_{a,b}$ is locally inner uniform on a small neighborhood near the origin, say $\{x\in\mathbb{R}^n:a<|x|<a+1/100\}$, where $1/100$ represents any sufficiently small number depending at most on the dimension $n$, chosen for convenience. We apply the boundary Harnack principle (Proposition \ref{bhp}) on this neighborhood; the hypothesis on $\lambda(A_{a,b}) R^2$ is satisfied because both $\lambda(A_{a,b})$ and $R^2$ are $\lesssim_n 1$ in the regime $a/b\in (0,1/2)$. By a chaining argument using Lemma \ref{phi}, we may enlarge this neighborhood from $a+1/100$ to $a+1/4$, so that
    \begin{align}
        \label{1a}
        \frac{\varphi_{A_{a,b}}(y_1)}{\varphi_{A_{a,b}}(y_2)}\asymp_n\frac{h(y_1)}{h(y_2)},\hspace{0.2in}\forall y_1,y_2\in\{a<|x|\leq a+1/4\}.
    \end{align}
    We choose $y_2$ with $|y_2|=a+1/4$. Lemma \ref{maxeig} and Proposition \ref{annulieigval} give
    \begin{align*}
        \varphi_{A_{a,b}}(y_2)\lesssim \|\varphi_{A_{a,b}}\|_{L^{\infty}(A_{a,b})}\asymp 1.
    \end{align*}
    On the other hand, because $\varphi_{A_{a,b}}$ is a radial function, we may assume without loss of generality that $y_2$ lies on the first coordinate axis. Define a ball $V=\{(x_1-(a+1/4))^2+x_2^2+\cdots+x_n^2<(1/4)^2\}\subseteq A_{a,b}$ centered on the first coordinate axis, with center $y_2$. By an eigenfunction comparison result in \cite[Theorem~4.2]{chaolsc}, as well as Lemma \ref{maxeig}, we get $\varphi_{A_{a,b}}(y_2)\gtrsim \varphi_V(y_2)\gtrsim_n 1$. It follows that $\varphi_{A_{a,b}}(y_2)\asymp 1$.
    
    Thus choosing the harmonic functions $h(x)=|x|^{-(n-2)}-a^{-(n-2)}$ if $n\geq 3$ and $h(x)=\ln(|x|/a)$ if $n=2$, we get from (\ref{1a}) that 
    \begin{align}
        \label{1b}
        \varphi_{A_{a,b}}(x)\asymp_n \frac{h(x)}{h(y_2)}\asymp \begin{cases}
        \displaystyle 1-(\frac{a}{|x|})^{n-2}, &n\geq 3
        \\ \displaystyle \frac{1}{\ln(1+1/(4a))}\ln(|x|/a),& n=2
    \end{cases}
    \end{align}
    for any $x$ with $a<|x|\leq a+1/4$. On the region $\{a+1/4<|x|<1\}$, we claim that eigenfunction comparison from \cite{chaolsc} implies 
    \begin{align}
        \label{1c}
        \varphi_{A_{a,b}}(x)\asymp_n 1-|x|.
    \end{align}
    The upper bound in (\ref{1c}) follows because $\varphi_{A_{a,b}}(x)\lesssim_n \varphi_{B(0,1)}(x)\asymp_n 1-|x|$, while the lower bound follows from a similar chain of inequalities applied to an appropriate ball $B\subseteq A_{a,b}$ tangent to $\partial A_{a,b}$. Note that (\ref{1c}) holds regardless of whether $n=2$ or $n\geq 3$. Combining (\ref{1b}) and (\ref{1c}) into a single global expression, we thus get that when $b=1$, for all $x\in A_{a,b}$,
    \begin{align*}
        \varphi_{A_{a,b}}(x)\asymp_n\begin{cases}
        \displaystyle \Big\{1-(\frac{a}{|x|})^{n-2}\Big\}(1-|x|), &n\geq 3
        \\ \displaystyle \frac{1}{\ln(1+1/(4a))}\ln\Big(\frac{|x|}{a}\Big)\cdot (1-|x|),& n=2.
        \end{cases}
    \end{align*}
    The general case follows from scaling properties of Dirichlet Laplacian eigenfunctions: for any domain $U$ and $c>0$, we have $\varphi_{cU}(x)=c^{-n/2}\varphi_{U}(x)$ where $cU$ is any homothetic copy of $U$ dilated by $c>0$.
\end{proof}

\subsection{Estimate for the eigenvalue gap of two annuli}

In this short section, we aim to prove in Lemma \ref{eiggap} that if $A_{a,b}\subseteq A_{a',b'}$ are two Euclidean annuli, then under suitable assumptions on $a,a',b,b'$, we can make $\lambda_{A_{a,b}}-\lambda_{A_{a',b'}}$ uniformly bounded.

\begin{Lemma}
    \label{hadalemma}
    Let $A_{a,b}=\{x\in \mathbb{R}^n:a<|x|<b\}$ denote an annulus with radii $0<a<b$. Let $\{b_{\varepsilon}\}_{\varepsilon>0}$ be nonnegative real numbers with $b_{\varepsilon}\to 0$ as $\varepsilon \to 0$. Define
    \begin{align*}
        \Phi_1(t)=\lambda(A_{1,1+t}),\hspace{0.5in}\Phi_2(t)=\lambda(A_{1-t,1+\varepsilon+b_{\varepsilon}}).
    \end{align*}
    For all $t\in (0,1]$, we have $|\Phi'_1(t)|\asymp t^{-3}$. For all $t\in [0,1/4)$ and for all $\varepsilon>0$ sufficiently small so that $\varepsilon+b_{\varepsilon}\leq 1$, $|\Phi_2'(t)|\asymp (\varepsilon+b_{\varepsilon}+t)^{-3}$. The implied constants depend only on the dimension $n$. 
\end{Lemma}

\begin{proof}
    Note that $\varphi_{A_{1,1+t}}$ is a radial function. Thus, by Hadamard's variational formula \cite[Page~369, Problem~15]{evans}, we have 
    \begin{align}
        \label{hada1}
       \frac{d\Phi_1(t)}{dt}=-\int_{\{|x|=1+t\}} \Bigg|\frac{\partial \varphi_{A_{1,1+t}}}{\partial r} \Bigg|^2 dS_t, 
    \end{align}
    where $dS_t$ is the surface measure on the sphere $\{x\in \mathbb{R}^n:|x|=1+t\}$. Lemma \ref{phiapprox} implies that
    \begin{align}
        \label{hada2}
         \Bigg|\frac{\partial \varphi_{A_{1,1+t}}}{\partial r} (r,\theta)\Bigg|\asymp \frac{1}{t^{3/2}\sqrt{\sigma_{n-1}(\mathbb{S}^{n-1})}},
    \end{align}
    with the implied constant depending only on $n$. Combining (\ref{hada1}) and (\ref{hada2}) gives the claim about $\Phi_1(t)$. The proof for $\Phi_2(t)$ is similar.
\end{proof}

\begin{Lemma}
    \label{eiggap}
    Let $\{a_{\varepsilon}\}_{\varepsilon>0}$ and $\{b_{\varepsilon}\}_{\varepsilon>0}$ be nonnegative real numbers, and assume that for some constants $C_1,C_2>0$, $a_{\varepsilon}\in [0,C_1\varepsilon^3)$ and $b_{\varepsilon}\in [0,C_2\varepsilon^3)$ for all $\varepsilon\in (0,1)$. Define
    \begin{align*}
        & A=A_{1,1+\varepsilon}=\{x\in \mathbb{R}^n:1<|x|<1+\varepsilon\},
        \\ & B = A_{1-a_{\varepsilon},1+\varepsilon+b_{\varepsilon}}=\{x\in \mathbb{R}^n:1-a_{\varepsilon}<|x|<1+\varepsilon+b_{\varepsilon}\}.
    \end{align*}
    For all sufficiently small $\varepsilon>0$ such that $\varepsilon+b_{\varepsilon}\leq 1$ and $a_{\varepsilon}\leq 1/4$, 
    $$0<\lambda_A-\lambda_B\leq C,$$
    for some finite constant $C\in (0,\infty)$ depending only on $C_1,C_2$ and $n$. 
\end{Lemma}

\begin{proof}
    The inequality $\lambda_A-\lambda_B>0$ is trivially due to domain monotonicity of the Dirichlet eigenvalue (Proposition \ref{lambda}). 
    
    For the other inequality, the mean value theorem guarantees that for some $\xi_1\in (\varepsilon,\varepsilon+b_{\varepsilon})$ and $\xi_2\in (0,a_{\varepsilon})$,
    \begin{align*}
        0<\lambda_A-\lambda_B &\leq  |\lambda_A-\lambda(A_{1,1+\varepsilon+b_{\varepsilon}})|+|\lambda(A_{1,1+\varepsilon+b_{\varepsilon}})-\lambda_B|
        \\ & =|\Phi_1'(\xi_1)|\cdot b_{\varepsilon}+|\Phi'_2(\xi_2)|\cdot a_{\varepsilon}.
    \end{align*}
   Thus by Lemma \ref{hadalemma}, $$\lambda_A-\lambda_B\lesssim \frac{b_{\varepsilon}}{\varepsilon^3}+\frac{a_{\varepsilon}}{(\varepsilon+b_{\varepsilon})^3},$$
   and we finish the proof by using the assumption that $a_{\varepsilon},b_{\varepsilon}\lesssim \varepsilon^3$.
\end{proof}

\begin{Rmk}
    \normalfont
    Since $a_{\varepsilon}\in (0,C_1\varepsilon^3)$ and $b_{\varepsilon}\in (0,C_2\varepsilon^3)$, it is clear that $\varepsilon$ in Lemma \ref{eiggap} can be chosen to be sufficiently small depending only on $C_1$ and $C_2$.   
\end{Rmk}

\subsection{Volume doubling with respect to the measure \texorpdfstring{$\varphi_U^2$}{TEXT}}
\label{VDsubsection}

In this section, we show in Theorem \ref{vdthm} that annular domains of the form $(a,b)\times U_0$ are uniformly volume doubling when they are thin (for example when $b/a\in (1,2]$).  

\begin{Dfn} 
    \normalfont \label{vddefn} 
    Suppose $U$ is a bounded domain of some Riemannian manifold with Riemannian measure $\mu$. Let $\widetilde{U}$ denote the closure of $U$ with respect to the geodesic distance $d_U$ on $U$. For $x\in \widetilde{U}$ and $r>0$, let $B_{\widetilde{U}}(x,r)\subseteq \widetilde{U}$ denote a geodesic ball in $\widetilde{U}$, with the convention that the radii of balls are taken to be minimal. We say that $U$ is \textit{$\varphi_U^2$-volume doubling}, or $(\varphi_U^2\text{-VD)}$, if there is a constant $D_0\geq 1$ such that, for all $x\in \widetilde{U}$ and $r>0$, 
    $$\int_{B_{\widetilde{U}}(x,2r)}\varphi_U^2d\mu\leq D_0\int_{B_{\widetilde{U}}(x,r)}\varphi_U^2d\mu.$$
\end{Dfn}

It follows from \cite{lierllsc} that inner uniform domains $U$ in large classes of metric measure spaces are $\varphi^2_U$-volume doubling; see \cite[Theorem~6.12(ii)]{lierllsc}. We rewrite this result in lesser generality in Theorem \ref{lierlVD}. 

\begin{Theo}    
    \label{lierlVD}
    Suppose $U_0\subseteq \mathbb{S}^{n-1}$ is a $(C_0,c_0)$-inner uniform domain with nonempty boundary. Then $U_0$ satisfies $(\varphi_{U_0}^2\textup{-VD})$, with volume doubling constant depending only on $n,C_0,c_0$. The same holds for $U_0=\mathbb{S}^{n-1}$ and  $\varphi_{U_0}^2\equiv 1/\sqrt{\sigma_{n-1}(\mathbb{S}^{n-1})}$, with constants depending only on $n$.
\end{Theo}

In order to prove volume doubling on a domain of the form $U=(a,b)\times U_0\subseteq [0,\infty)\times \mathbb{S}^{n-1}$ in polar coordinates, it is natural to prove that the metric on $U$ is controlled by the metrics on $(a,b)$ and $U_0$.

\begin{Lemma}
\label{distcomp0}
Let $U_0\subseteq \mathbb{S}^{n-1}$ be a domain and set $U=(1,1+\varepsilon)\times U_0 \subseteq \mathbb{R}^n$. For $x,y\in U$, by abuse of notation we write $d_{U_0}(x,y)$ in place of  $d_{U_0}(x/\|x\|,y/\|y\|)$. Suppose $\varepsilon\in (0,1]$. Then 
    \begin{align}
        \label{distcomp}
        \max\Big\{d_{U_0}(x,y),\Big|\|x\|-\|y\|\Big|\Big\}\leq d_U(x,y)\leq 4\max\Big\{d_{U_0}(x,y),\Big|\|x\|-\|y\|\Big|\Big\}.
    \end{align}
    
\end{Lemma}

\begin{proof}
    Let $\widetilde{\mathbf{g}}=(\widetilde{g}_{ij})$ and $(\theta_1,\theta_2,...,\theta_{n-1})$ denote the Riemannian metric and local coordinates on $\mathbb{S}^{n-1}$, respectively. For any function $u(r,\theta)=u(\theta)$ depending only on $\theta\in\mathbb{S}^{n-1}$, 
    \begin{align*}
    |\nabla u|^2=\Big(\frac{\partial u}{\partial r}\Big)^2+\frac{\widetilde{g}^{ij}}{r^2}\cdot \frac{\partial u}{\partial \theta^i}\frac{\partial u}{\partial \theta^j}\leq \widetilde{g}^{ij}\cdot \frac{\partial u}{\partial \theta^i}\frac{\partial u}{\partial \theta^j}= |\nabla u|^2_{\widetilde{\mathbf{g}}},
    \end{align*}
    and therefore by (\ref{variationaldist}),
    \begin{align*}
        d_U(x,y)\geq \sup\{u(x)-u(y):u\in C(U_0)\cap \mathcal{F}_{\text{loc}}(U_0),|\nabla u|_{\widetilde{\mathbf{g}}}\leq 1\}=d_{U_0}(x,y).
    \end{align*}
    Similarly, by considering radial functions, $d_U(x,y)\geq \big|\|x\|-\|y\|\big|$. This gives the left-hand inequality in (\ref{distcomp}). 

    Now let $x,y\in U$ be any pair of points and let $\overline{\gamma}:[0,1]\to U_0$ be any path with $\overline{\gamma}(0)=x/\|x\|,\overline{\gamma}(1)=y/\|y\|$. Lift $\overline{\gamma}$ up to a path $\gamma$ in $U$ with $\gamma(0)=x,\gamma(1)=y$ by defining $$\gamma(t)=((1-t)\|x\|+t\|y\|,\overline{\gamma}(t))=:(\gamma_0(t),\gamma_1(t),...,\gamma_{n-1}(t)).$$ Letting $L(\gamma)$ denote the length of a curve $\gamma$, we compute
    \begin{align*}
        L(\gamma) & = \int_0^{1}\sqrt{\sum_{i,j=1}^{n-1}r^2\widetilde{g}_{ij}\frac{d\gamma_i}{dt}\cdot \frac{d\gamma_j}{dt}+\Big(\frac{d\gamma_0}{dt}\Big)^2} dt
        \\ & \leq 2\int_0^{1}\sqrt{\widetilde{g}_{ij}\frac{d\gamma_i}{dt}\cdot \frac{d\gamma_j}{dt}}+\Big|\frac{d\gamma_0}{dt}\Big| dt \\ & = 2L(\overline{\gamma})+2\big|\|x\|-\|y\|\big|.
    \end{align*}
    Thus, $d_U(x,y)\leq 2L(\overline{\gamma})+2\big|\|x\|-\|y\|\big|$. Taking infimum over all paths $\overline{\gamma}$ that stay in $U_0$, we get $d_U(x,y)\leq 2d_{U_0}(x,y)+2\big|\|x\|-\|y\||\leq 4\max\{d_{U_0}(x,y),\big|\|x\|-\|y\|\big|\}.$
\end{proof}

We are now ready to give the proof of volume doubling with respect to $\varphi_U^2$.

\begin{Theo}
    \label{vdthm}
    Let $U_0=\mathbb{S}^{n-1}$ or let $U_0\subseteq \mathbb{S}^{n-1}$ be a $(C_0,c_0)$-inner uniform domain. In polar coordinates, consider the domain $U=(a,b)\times U_0$, where $b/a\in (1,2]$. Then $U$ satisfies $(\varphi_{U}^2\textup{-VD})$, with volume doubling constant depending only on $n$, $C_0,c_0$, and a uniform upper bound on $\lambda(U_0)$.  
\end{Theo}

\begin{proof}
    In this proof, we will often write $z$ to denote either $z\in U$ or $z/\|z\|\in U_0$ when there is no ambiguity. Also, observe that if $U\subseteq \mathbb{R}^n$ satisfies $(\varphi_U^2\textup{-VD})$, then $cU$ satisfies $(\varphi_{cU}^2\textup{-VD})$ (with the same constant), so by scale-invariance we may assume $a=1$ and $b=1+\varepsilon$. By Lemma \ref{distcomp0}, for any $z\in \widetilde{U}$ and $s>0$, we have
    \begin{align}
        \label{vd0}
        \Big(\|z\|-\frac{s}{4},\|z\|+\frac{s}{4}\Big)\times B_{\widetilde{U_0}}(z,s/4)\subseteq B_{\widetilde{U}}(z,s)\subseteq (\|z\|-s,\|z\|+s)\times B_{\widetilde{U_0}}(z,s).
    \end{align}
    Thus, by (\ref{vd0}) and Lemma \ref{lierlcompare},
    \begin{align}
        \nonumber \int_{B_{\widetilde{U}}(z,2s)}\varphi_{U}^2 dx & \lesssim \int_{B_{\widetilde{U_0}}(z,2s)}\varphi_{U_0}^2 d\sigma_{n-1}\cdot \int_{\|z\|-2s}^{\|z\|+2s} \frac{\min\{r-1,1+\varepsilon-r\}^2}{\varepsilon^{3}}\cdot r^{n-1}dr
        \\ \label{vd1}& \lesssim \int_{B_{\widetilde{U_0}}(z,s/4)}\varphi_{U_0}^2d\sigma_{n-1}\cdot \int_{\|z\|-s/4}^{\|z\|+s/4} \frac{\min\{r-1,1+\varepsilon-r\}^2}{\varepsilon^{3}}\cdot r^{n-1}dr.
    \end{align}
    We interpret $\|z\|+2s$ as $\min\{\|z\|+2s,1+\varepsilon\}$ and $\|z\|-2s$ as $\max\{\|z\|-2s,1\}$, and likewise for $\|z\|\pm s/4$. The inequality (\ref{vd1}) follows because of Theorem \ref{lierlVD}, and because of the following fact: if $\varphi_{(a,b)}$ denotes the principal Dirichlet Laplacian eigenfunction of the bounded interval $(a,b)$, then $(a,b)$ is $\varphi_{(a,b)}^2$-volume doubling with a volume doubling constant uniform over all bounded intervals, see for example Theorem \ref{boxHK}. Using (\ref{vd0}) again on (\ref{vd1}), we obtain
    \begin{align*}
        \int_{B_{\widetilde{U}}(z,2s)}\varphi_{U}^2 dx\lesssim  \int_{B_{\widetilde{U}}(z,s)}\varphi_{U}^2 dx,
    \end{align*}
    as desired.
\end{proof}

\subsection{Poincar\'{e} inequalities with respect to the measure \texorpdfstring{$\varphi_U^2$}{TEXT}}
\label{PIsubsection}

In this subsection, we prove that all annular domains $U$ of interest, including those in Theorem \ref{illustration} and Section \ref{examples}, all satisfy a family of scale-invariant Poincar\'{e} inequalities with respect to $\varphi_U^2$. We note that the proof to be given makes use of $(\varphi_U^2\textup{-VD})$, which was proved earlier in Theorem \ref{vdthm}, Section \ref{VDsubsection}.

\begin{Dfn}
    \label{pidefn} 
    In the setting of Definition \ref{vddefn}, we say that $U$ satisfies $\varphi_U^2$\textit{-Poincar\'{e} inequalities}, or $(\varphi_U^2\textup{-PI})$, if there exist constants $P\geq 1$ and $C\geq 1$ such that  
    \begin{align}
    \label{PIphi}
    \forall f\in W^{1,2}(U,\varphi_U^2 d\mu),\hspace{0.1in}\forall B_U(x,r)\subseteq \widetilde{U},\hspace{0.2in}\int_{B_U(x,r)}|f-f_B|^2 \varphi_U^2 d\mu\leq Pr^2\int_{B_U(x,Cr)}|\nabla f|^2 \varphi_U^2 d\mu.
\end{align}
\end{Dfn}

It is known from \cite{lierllsc} that Poincar\'{e} inequalities of the form (\ref{PIphi}) hold for inner uniform domains $U\subseteq \mathbb{S}^{n-1}$, or even just for locally inner uniform domains, as in Theorem \ref{lierlPI} below.

\begin{Theo}
    \label{lierlPI} Let $U_0\subseteq \mathbb{S}^{n-1}$ be a domain, and, in polar coordinates, consider the domain $U=(1,1+\varepsilon)\times U_0$.
    \begin{enumerate}
        \item  Suppose $U_0$ is a $(C_0,c_0)$-inner uniform domain. $U_0$ satisfies $(\varphi_{U_0}^2\textup{-PI})$ with $C=1$ and $P\geq 1$ depending on $C_0,c_0,n$. The same holds for $U_0=\mathbb{S}^{n-1}$ and $\varphi_{U_0}^2\equiv 1/\sqrt{\sigma_{n-1}(\mathbb{S}^{n-1})}$, with $C=1$ and $P\geq 1$ depending only on $n$.
        \item Suppose $U$ is $(C_0',c_0')$-locally inner uniform up to scale $\varepsilon$ (see Definition \ref{liu}).  Then, $U$ satisfies $(\varphi_U^2\textup{-PI})$ up to scale $\varepsilon$. That is, fixing an auxiliary integer $N$, (\ref{PIphi}) holds for all $r\in (0,N\varepsilon)$, with $C=1$ and $P$ depending only on $n,C'_0,c'_0,N$. 
    \end{enumerate}
\end{Theo}

Our goal is to prove the following theorem, which improves the second item of Theorem \ref{lierlPI}. 

\begin{Theo}
    \label{PIthm}
    Let $U_0\subseteq \mathbb{S}^{n-1}$ be a domain. For $\varepsilon\in (0,1]$, put $U=(1,1+\varepsilon)\times U_0\subseteq \mathbb{R}^n$. Suppose $U_0$ is $(C_0,c_0)$-inner uniform, and $U$ is $(C_0,c_0)$-locally inner uniform up to scale $\varepsilon$. Then, $U$ satisfies $\varphi_U^2$-Poincar\'{e} inequalities (\ref{PIphi}) for any ball $B_U(x,r)\subseteq \widetilde{U}$, with $C$ an absolute constant, and with $P$ depending on $n,C_0,c_0,$ and a uniform upper bound on $\lambda(U_0)$.
\end{Theo}

\begin{Rmk}
    \normalfont
    \begin{enumerate}
        \item The requirement that $\varepsilon\in (0,1]$ is flexible and chosen for convenience; we can replace the upper bound $1$ by any fixed larger number.
        \item Recall the balls $B_U(x,r)$ are by definition subsets of $U$, and they are taken to be of minimal radius. Thus $B_U(x,r)=B_U(x,\text{diam}_U)=U$ for all sufficiently large $r$, and Theorem \ref{PIthm} can be described as either $(\varphi_U^2\textup{-PI})$ up to all scales, or $(\varphi_U^2\textup{-PI})$ up to scale $\text{diam}_U\asymp 1$.
        \item Once Theorem \ref{PIthm} is proven, we can in fact take $C=1$ in (\ref{PIphi}) at the cost of replacing $P$ by a larger constant $P'$. This is possible due to a general Whitney decomposition argument of Jerison \cite{jerison} (see also \cite[Section~5.3.2]{asti}). Jerison's argument only requires the metric measure space $(U,d_U,\varphi_U^2)$ to be volume doubling, which we proved in Theorem \ref{vdthm}. We note that the work \cite{jerison} was carried out in the setting of vector fields on $\mathbb{R}^n$ satisfying H\"{o}rmander's bracket-generating condition. After \cite{jerison}, there have been subsequent developments in establishing the Poincar\'{e} inequality with $C=1$ in other settings. We refer, for example, to the works of G. Lu in \cite{lu1} and \cite{lu2}. In \cite{lu1}, Lu proves the Poincar\'{e} inequality (with $C=1$) for H\"{o}rmander vector fields, with respect to certain weighted measures. On the other hand \cite{lu2} gave a simpler proof of the Poincar\'{e} inequality in \cite{jerison}, using a different covering argument.          
        \item Theorem \ref{PIthm} also holds when $U_0=\mathbb{S}^{n-1}$, in which case $\lambda(U_0)=0$ is trivially upper bounded, and all implied constants depend only on $n$.
    \end{enumerate}
    
\end{Rmk}

Before proving Theorem \ref{PIthm}, we give heuristics for its validity when $U=\{a<|x|<b\}$ is a thin annulus. By scale invariance of (\ref{PIphi}) with respect to scaling the domain $U$ in the radial direction, we may assume that  $U=\{1<|x|<1+\varepsilon\}$ with $\varepsilon\in (0,1]$.

First, observe that by Theorem \ref{lierlPI}, $(\varphi_U^2\textup{-PI})$ holds at all scales $r$ with $0<r\lesssim \varepsilon$. Even without Theorem \ref{lierlPI}, we can deduce $(\varphi_U^2\textup{-PI})$ at small scales as follows. Theorem \ref{annulusprofile} yields an explicit expression comparable to $\varphi_U$, which behaves like a tent function in the radial direction. At scale $r\lesssim \varepsilon$, one can imagine $B_U(x,r)$ as contained inside a small $n$-dimensional box with side lengths $\asymp \varepsilon$. The small box is $(C_0,c_0)$-inner uniform. (Note that the entire annulus fails to be uniformly inner uniform as $\varepsilon\in (0,1]$ approaches $0$.) Since $(\varphi_U^2\textup{-PI})$ is known to hold if $U$ is an $n$-dimensional box (Theorem \ref{boxHK}), and the $\varphi_U$ for such a box is also comparable to a tent function, we deduce (\ref{PIphi}) for the annulus at scale $\varepsilon$. In the above discussion, the normalization constant in  $\varphi_U$ so that $\|\varphi_U\|_{L^2(U)}=1$ is unimportant, since it is trivial that the validity of $\varphi_U^2$-Poincar\'{e} inequalities does not change by multiplying $\varphi_U$ by a constant.

To make the above reasoning for $U=\{a<|x|<b\}$ completely rigorous, one could follow the line of reasoning developed in \cite{gyryalsc} and \cite{lierllsc}, keeping all the calculations localized to scale $\varepsilon$ using local inner uniformity of the annulus $U$ at the same scale $\varepsilon$. Of course the arguments of \cite{gyryalsc} and \cite{lierllsc} are much more general and do not require explicit estimates on $\varphi_U$.  

To then obtain (\ref{PIphi}) for all scales, we follow a discretization argument of Coulhon and Saloff-Coste \cite{coulsc}. We imagine the annulus $U=(1,1+\varepsilon)\times \mathbb{S}^{n-1}$ as being \textit{quasi-isometric} to $\mathbb{S}^{n-1}$, which will allow the techniques of \cite{coulsc} to deduce $(\varphi_U^2\textup{-PI})$ from $(\varphi_{\mathbb{S}^{n-1}}^2\textup{-PI})$, for scales larger than $\varepsilon$. To run this discretization argument in our setting, we need the following already proven facts:
\begin{enumerate}
    \item[(i)] $(\varphi_U^2\textup{-PI})$ at scales $\lesssim \varepsilon$ (Theorem \ref{lierlPI}, due to \cite{lierllsc});
    \item[(ii)]  $\varphi_U^2$-volume doubling (Theorem \ref{vdthm});
    \item[(iii)] The weighted measure $\varphi_U^2(x)dx$ is comparable to a well-behaved product measure of the form $f(r)g(\theta)r^{n-1}dr d\theta$ in polar coordinates (Lemma \ref{lierlcompare});
    \item[(iv)] $U_0\subseteq \mathbb{S}^{n-1}$ satisfies $(\varphi_{U_0}^2\textup{-VD})$ and $(\varphi_{U_0}^2\textup{-PI})$ (Theorems \ref{lierlVD} and \ref{lierlPI}, which are due to \cite{lierllsc}). 
\end{enumerate}

We now define some necessary notions from \cite{coulsc} to prove Theorem \ref{PIthm} in greater generality than just the annulus. 

Fix $\varepsilon \in (0,1]$ and let $U=(1,1+\varepsilon)\times U_0$, where $U_0\subseteq \mathbb{S}^{n-1}$ is a spherical domain with possibly nonempty boundary. 
Let $X:=\{x_i\}$ denote a maximal $\varepsilon$-separated subset of $U$. That is, 
for all $x_i,x_j\in X$, we have $d_U(x_i,x_j)\geq \varepsilon$, and $X$ is a maximal subset of $U$ with this property with respect to set inclusion. We call $X=\{x_i\}$ an $\varepsilon$\textit{-net} of $U$. 

Let $y_i=x_i/\|x_i\|$ be the projection of the points $x_i$ to $\mathbb{S}^{n-1}$; we will also call $Y:=\{y_i\}$ an $\varepsilon$-net of $U_0$. Technically speaking, $Y$ may not be exactly an $\varepsilon$-net of $U_0$, but thanks to Lemma \ref{distcomp0}, $Y$ satisfies the following two properties:
\begin{align}
    \label{Y1}& \forall y_i,y_j\in Y,\hspace{0.1in}d_{U_0}(y_i,y_j)\geq \frac{\varepsilon}{4},
    \\ \label{Y2}&\forall z\in U_0, \enspace \exists y_i\in Y,\enspace d_{U_0}(z,y_i)\leq \varepsilon,
\end{align}
which are enough for our argument to go through. 

We make $X=\{x_i\}$ into a weighted graph $(X,m_U)$ as follows. We take the vertex set to be $\{x_i\}$, and declare that two vertices $x_i,x_j$ are connected by an edge (denoted $x_i\sim x_j$) if $d_U(x_i,x_j)\leq 2\varepsilon$. The weight on each vertex $x_i$ is given by
$$m_U(x_i):=\int_{B_U(x_i,\varepsilon)}\varphi_U^2 dx.$$

Similarly, $Y=\{y_i\}$ also has a weighted graph structure $(Y,m_{U_0})$. We declare $y_i\sim y_j$ if and only if $x_i,x_j$ are connected by an edge in $X$. The weight on $y_i$ is given by
$$m_{U_0}(y_i):=\int_{B_{U_0}(y_i,\varepsilon)}\varphi_{U_0}^2 d\sigma_{n-1}.$$
We equip both $X$ and $Y$ with the graph distance, which we denote by $d(\cdot,\cdot)$. It is clear the graph structures of $X$ and $Y$ coincide, so a discrete ball $B(x,m)$ is the same in both $X,Y$. In fact, the weights on $X$ and $Y$ are also comparable, as the following lemma shows.

\begin{Lemma}
    \label{weightcomp}
    Suppose $\varepsilon \in (0,1]$. For any $x_i\in U$ and its projection $y_i\in U_0$, $m_U(x_i)\asymp m_{U_0}(y_i)$. The implied constants depend only on $n$ and a uniform upper bound on $\lambda(U_0)$.
\end{Lemma}

Lemma \ref{weightcomp} is an easy consequence of Lemmas \ref{lierlcompare} and \ref{distcomp0}, and we omit its proof.

We now explain a correspondence between functions defined on $U$ and those defined on $X$. For each $f:U\to \mathbb{R}$, we associate a function $\widetilde{f}:X\to \mathbb{R}$ defined via taking ball averages of $f$ with respect to the weighted measure $\varphi_U^2$, that is,
$$\forall x\in X,\enspace \widetilde{f}(x):=\frac{\int_{B_U(x,\varepsilon)}f(y)\varphi_U^2(y)dy}{\int_{B_U(x,\varepsilon)}\varphi_U^2(y)dy}.$$
Conversely, functions defined on $X$ can be associated to functions defined on $U$ as follows. Let $\{\theta_x\}_{x\in X}$ be a \textit{partition of unity} indexed by the $\varepsilon$-net $X$, satisfying 
\begin{align}
    \label{pou}
    \forall x\in X,\hspace{0.1in}\theta_x\equiv 1\text{ on }B_U(x,\varepsilon/2),\hspace{0.15in}\theta_x\equiv 0\text{ on }B_U(x,3\varepsilon/2)^c,\hspace{0.15in}\|\nabla\theta_x\|_{\infty}\leq C/\varepsilon,
\end{align}
where $C$ is an absolute constant. Such a partition of unity is constructed in Section \ref{partitionofunity} of the Appendix. Given $f:X\to \mathbb{R}$, we associate  $\hat{f}:U\to \mathbb{R}$ by setting
$$\hat{f}(y):=\sum_{x\in X}f(x)\theta_x(y),\hspace{0.1in}y\in U.$$
Likewise, functions $f:U_0\to \mathbb{R}$ give rise to functions $\widetilde{f}:Y\to \mathbb{R}$, and vice versa. We define the partition of unity $\{\theta_y\}_{y\in Y}$ so that
\begin{align}   
    \label{pouY}
    \forall y\in Y,\enspace \theta_y\equiv 1 \text{ on }B_{U_0}(y, \varepsilon/8),\enspace \theta_y\equiv 0\text{ on }B_{U_0}(x,3\varepsilon/8)^c,\enspace \|\nabla \theta_y\|_{\infty}\leq C/\varepsilon.
\end{align}
Also, since the $\varepsilon$-nets $X$ and $Y$ are in bijection, we will identify functions $f:Y\to \mathbb{R}$ with $f:X\to \mathbb{R}$ in the canonical way.

Given a function $\psi(x)$ on some metric measure space, we let $\psi_{\varepsilon}(x)$ denote the average of $\psi$ over a ball of center $x$ and radius $\varepsilon$. Given a ball $B$, we let $\psi_B$ denote the average value of $\psi$ over $B$. When writing notation for ball averages such as $\psi_{\varepsilon}$ and $\psi_B$, the ambient metric measure space will be clear from context. 

Given a function $f$ on the weighted graph $(X,m_{U})$, and $E\subseteq X$, we denote $L^p(E,m_U)$ norms by
$$\|f\|_{p,E}:=\Big(\sum_{x\in E}|f(x)|^p m_U(x)\Big)^{1/p},$$
and we let
$$\delta f(x):=\Big(\sum_{y\sim x}|f(y)-f(x)|^2\Big)^{1/2}$$
denote the norm of the discrete gradient of $f$.

Finally, let $N=\max_{x\in X}\#\{x'\in X:x'\sim x\}$ be the maximum number of neighbors of any vertex in the weighted graph $(X,m_U)$. It is possible to show that, because of $\varphi_U^2$-volume doubling from Theorem \ref{vdthm}, the number $N$ is uniformly bounded above by a constant depending only on the $\varphi_U^2$-volume doubling constant; see (\ref{6.4.3}) below for an explanation.

With all of the notations in place, we begin to prepare the proof of Theorem \ref{PIthm}. We begin with Lemma \ref{coulsc5.3}, which is essentially \cite[Lemma~5.3]{coulsc}, adapted for our purposes. (The proof of Lemma \ref{coulsc5.3} was not written out in \cite{coulsc}, so we write it here.)
\begin{Lemma}
\label{coulsc5.3}
Let $U\subseteq M$ be a bounded domain of a Riemannian manifold $M$ with Riemannian measure $d\mu$. Let $\varphi_U$ denote the principal Dirichlet eigenfunction of $U$. Let $V(x,r)$ denote the measure of the ball $B_{U}(x,r)$ with respect to $\varphi_U^2 d\mu$. Suppose $U$ satisfies $(\varphi_U^2\textup{-VD})$ (Definition \ref{vddefn}) with constant $D$. Also suppose $U$ satisfies $L^p$-Poincar\'{e} inequalities up to scale $3\varepsilon$, that is, for any $\psi\in W^{1,p}(U,\varphi_U^2 d\mu)$,
\begin{align}
    \label{5.3.0}
    \forall r\in (0,3\varepsilon],\hspace{0.1in}\forall B_U(x,r)\subseteq U,\hspace{0.1in}\int_{B_U(x,r)}|\psi-\psi_B|^p\varphi_U^2 d\mu\leq P(p,r)\int_{B_U(x,2r)}|\nabla \psi|^p \varphi_U^2 d\mu.
\end{align}
Then, for any $x,y\in U$ with $d_U(x,y)<2\varepsilon$,
$$V(x,\varepsilon)|\psi_{\varepsilon}(x)-\psi_{\varepsilon}(y)|^p \leq 2^pD^{2+\log_2 3} \cdot P(p,3\varepsilon)\int_{B_U(x,6\varepsilon)}|\nabla \psi|^p \varphi_U^2 d\mu.$$
\end{Lemma}

\begin{proof}
    Abbreviating $B_U(x,\varepsilon)=B(x,\varepsilon)$, $\varphi_U^2(w)d\mu(w)=dw$, and $\varphi_U^2(z)d\mu(z)=dz$,
    \begin{align*}
         |\psi_{\varepsilon}(x)-\psi_{\varepsilon}(y)|^p
         & = \Bigg|\frac{1}{V(x,\varepsilon)V(y,\varepsilon)}\int_{B(x,\varepsilon)}\int_{B(y,\varepsilon)}\psi(z)-\psi(w) dw dz\Bigg|^p 
        \\ & \leq \frac{1}{V(x,\varepsilon)V(y,\varepsilon)}\int_{B(x,\varepsilon)}\int_{B(y,\varepsilon)}|\psi(z)-\psi(w)|^pdwdz.
    \end{align*}
    Since $d_U(x,y)\leq 2\varepsilon$, $B(y,\varepsilon)\subseteq B(x,3\varepsilon)$. Thus for any $\alpha\in \mathbb{R}$,
    \begin{align}
        \nonumber V(x,\varepsilon)|\psi_{\varepsilon}(x)-\psi_{\varepsilon}(y)|^p&\leq \int_{B(x,3\varepsilon)}\int_{B(x,3\varepsilon)} |\psi(z)-\psi(w)|^p \frac{dw dz}{V(y,\varepsilon)}
        \\ \nonumber & \leq 2^{p-1} \int_{B(x,3\varepsilon)}\int_{B(x,3\varepsilon)}|\psi(z)-\alpha|^p+|\psi(w)-\alpha|^p \frac{dwdz}{V(y,\varepsilon)}
        \\ \nonumber & \leq 2^{p}\frac{V(x,3\varepsilon)}{V(y,\varepsilon)}\int_{B(x,3\varepsilon)}|\psi(z)-\alpha|^p dz
        \\ \label{5.3.1}& \leq 2^pD^{2+\log_2 3}\int_{B(x,3\varepsilon)}|\psi(z)-\alpha|^p dz
    \end{align}
    In (\ref{5.3.1}), the hypothesis of volume doubling guarantees that $V(x,3\varepsilon)/V(y,\varepsilon)\leq D^{2+\log_2 3}$; see for example \cite[Proposition~3.2]{chaolsc}.
     Setting $\alpha=\psi_{3\varepsilon}(x)$  and applying the Poincar\'{e} inequality (\ref{5.3.0}) finishes the proof.
\end{proof}

We also give a version of \cite[Lemma~6.4]{coulsc} below as Lemma \ref{gradcomp}; the proof is adapted from \cite{coulsc}.

\begin{Lemma}
    \label{gradcomp}
    Let $U_0=\mathbb{S}^{n-1}$ or let $U_0\subseteq \mathbb{S}^{n-1}$ be a $(C_0,c_0)$-inner uniform domain. In polar coordinates, consider the domain $U=(1,1+\varepsilon)\times U_0$ where $\varepsilon\in (0,1]$. Let $X=\{x_i\}$ and $Y=\{y_i\}$ be the $\varepsilon$-nets of $U$ and $U_0$ described at the beginning of this subsection, respectively.
    \begin{enumerate}
        \item Suppose $U$ is  $(C_0,c_0)$-locally inner uniform up to scale $\varepsilon$. For any $\psi\in W^{1,2}(U,\varphi_U^2 dx)$, and for any discrete ball $B(x,m)\subseteq X$, we have
        \begin{align}
            \label{gradcomp1}\|\delta\widetilde{\psi}\|^2_{2,B(x,m)} & \leq C_1\varepsilon^2\|\nabla \psi\|^2_{L^2(B_U(x,4\varepsilon m),\varphi_U^2 dx)},
        \end{align}
        where the constant $C_1$ depends only on $n,C_0,c_0,$ and a uniform upper bound on $\lambda(U_0)$. 
        \item For any $z\in U_0$, pick a point $\overline{z}\in Y$ with $d_{U_0}(z,\overline{z})\leq \varepsilon$. For all $r\geq \varepsilon$ and any function $f:Y\to \mathbb{R}$, 
        \begin{align*}
            \|\nabla \hat{f}\|_{L^2(B_{U_0}(z,r),\varphi_{U_0}^2 d\sigma_{n-1})}^2\leq \frac{C_2}{\varepsilon^2}\|\delta f\|^2_{2,B(\overline{z},10r)},
        \end{align*}
        where the constant $C_2$ depends only on $n, C_0,c_0$.
    \end{enumerate} 
    
\end{Lemma}

\begin{proof}
    For the first item, we combine Lemma \ref{coulsc5.3} and Poincar\'{e} inequalities at scale $\varepsilon$ (Theorem \ref{lierlPI}, 2) to write
    \begin{align}
        \nonumber \|\delta\widetilde{\psi}\|^2_{2,B(x,m)}& =\sum_{z\in B(x,m)}\sum_{y\sim z}|\widetilde{\psi}(y)-\widetilde{\psi}(z)|^2 m_U(z) 
        \\ \label{6.4.1}& \leq 4D^{2+\log_2 3}\cdot P\varepsilon^2\sum_{z\in B(x,m)}\sum_{y\sim z} \int_{B_U(z,6\varepsilon)}|\nabla \psi|^2 \varphi_U^2 dx
        \\ \label{6.4.2}& \lesssim \varepsilon^2 \int_{B_U(x,4\varepsilon m)}|\nabla \psi|^2 \varphi_U^2 dx. 
    \end{align}
    To go from (\ref{6.4.1}) to (\ref{6.4.2}), we used the fact that a discrete ball of radius $m$ is roughly a continuous ball of radius $m\varepsilon$, and that
\begin{align}
    \label{6.4.3}
    \sup_{z\in X}\#\{y\in X:y\sim z\}\lesssim 1, \hspace{0.1in}\sum_{z\in B(x,m)}\mathbf{1}_{B_U(z,6\varepsilon)}\lesssim \mathbf{1}_{B_U(x,4\varepsilon m)},
\end{align}
where both of the implied constants in (\ref{6.4.3}) depend only on the $\varphi_U^2$-volume doubling constant $D$. This can be proven via the method of \cite[Proposition~1.2.16]{muruganthesis} (which in turn builds on ideas from \cite{kanai}, \cite{kanai2}, and \cite{coulsc}); we omit the details. The implied constant $C_1$ in (\ref{gradcomp1}) then clearly depends only on the $\varphi_U^2$-volume doubling constant of $U$ and the constant $P$. Thanks to Theorems \ref{vdthm} and \ref{lierlPI}, the parameters which $C_1$ depends on are as claimed.

For the second item, first note that choosing $\overline{z}\in Y$ is possible due to the property (\ref{Y2}) of $Y=\{y_i\}$. Since $\{\theta_y\}_{y\in Y}$ is a partition of unity, $\sum_{y\in Y}\nabla \theta_y=0$, so that
\begin{align*}
    \nabla \hat{f}(\cdot )=\sum_{y\in Y}(f(y)-f(x))\nabla \theta_y(\cdot ),\enspace \forall x\in Y.
\end{align*}
Thus, for any $x\in Y$ and $z\in B_{U_0}(x,\varepsilon)$,
\begin{align}
    \label{6.4.4}|\nabla \hat{f}(z)|&\lesssim \frac{1}{\varepsilon}\cdot \sup\{|f(y)-f(x)|:d(y,x)\leq 2\}
    \\ \label{6.4.5} & \lesssim \frac{1}{\varepsilon}\sum_{z\in Y:d(z,x)\leq 2}\delta f(z).
\end{align}
(Elementary arguments show that the right hand side of (\ref{6.4.4}) can be upper bounded by an absolute constant multiple of (\ref{6.4.5}).) 
From (\ref{6.4.5}) we obtain
\begin{align*}
    \int_{B_{U_0}(z,r)} |\nabla \hat{f}(\xi)|^2 \varphi_{U_0}^2(\xi)\sigma_{n-1}(d\xi) &\leq \sum_{x\in Y:x\in B_{U_0}(\overline{z},5 r)}\int_{B_{U_0}(x,\varepsilon)} |\nabla \hat{f}(\xi)|^2 \varphi_{U_0}^2(\xi) \sigma_{n-1}(d\xi)
    \\ & \lesssim \frac{1}{\varepsilon^2} \sum_{x\in Y:x\in B_{U_0}(\overline{z},5 r)} \sum_{z:d(z,x)\leq 2}|\delta f(z)|^2 m_{U_0}(x).
\end{align*}
We obtain the desired conclusion by enlarging the radius $5r$ and multiplying the right-hand side by a large constant depending only on $N$ (the uniformly bounded number of neighbors of each vertex in $X$). 
\end{proof}

Because $U_0$ satisfies $(\varphi_{U_0}^2\textup{-PI})$, its discretization $(Y,m_{U_0})$ also satisfies a discrete Poincare inequality. 

\begin{Lemma}
    \label{dPI} 
    Let $U_0=\mathbb{S}^{n-1}$ or let $U_0\subseteq \mathbb{S}^{n-1}$ be a $(C_0,c_0)$-inner uniform domain. The weighted graph $(Y,m_{U_0})$ satisfies a discrete Poincar\'{e} inequality: for any ball $B(x,m)\subseteq Y$ and  $f:Y\to \mathbb{R}$,
    \begin{align*}
        \sum_{y\in B(x,m)}|f(y)-f_{B(x,m)}|^2 m_{U_0}(y) \leq P_Y m^2\sum_{y\in B(x,Cm)}\delta f(y)^2 m_{U_0}(y),
    \end{align*}
    where $C$ is an absolute constant, and the constant $P_Y$ depends only on $n,C_0,c_0$.
\end{Lemma}

\begin{proof}
    This follows from the proof method of \cite[Proposition~6.10]{coulsc}; we include the details here for completeness. Throughout this proof, we abbreviate $dz=d\sigma_{n-1}(z)$. Given a discrete ball $B(x,m)\subseteq Y$,  $f:Y\to \mathbb{R}$, and $\alpha\in \mathbb{R}$, 
    write
    \begin{align}
        \nonumber & \sum_{y\in B(x,m)}|f(y)-\alpha|^2 m_{U_0}(y)
        \\ \nonumber & = \sum_{y\in B(x,m)}\int_{B_{U_0}(y,\varepsilon)}|f(y)-\alpha|^2\varphi_{U_0}^2(z)dz
        \\ \label{pi1}& \lesssim \sum_{y\in B(x,m)}\int_{B_{U_0}(y,\varepsilon/8)}|f(y)-\alpha|^2
        \varphi_{U_0}^2(z) dz
        \\ \nonumber & \lesssim \sum_{y\in B(x,m)}\int_{B_{U_0}(y,\varepsilon/8)}|f(y)-\hat{f}(z)|^2\varphi_{U_0}^2(z)dz+\sum_{y\in B(x,m)}\int_{B_{U_0}(y,\varepsilon/8)}|\hat{f}(z)-\alpha|^2\varphi_{U_0}^2(z)dz
        \\ \nonumber & = I+II
    \end{align}
    We used $(\varphi_{U_0}^2\textup{-VD})$ (Theorem \ref{lierlVD}) to obtain (\ref{pi1}). Since the balls $\{B_{U_0}(y,\varepsilon/8):y\in Y\}$ are pairwise disjoint by (\ref{Y1}), they are all contained in a larger continuous ball $B_{U_0}(x,Cm\varepsilon)$. Thus,
    \begin{align*}
        II \leq \int_{B_{U_0}(x,Cm\varepsilon)}|\hat{f}(z)-\alpha|^2 \varphi_{U_0}^2(z)dz.
    \end{align*}
    Choose $\alpha$ to be the average value of $\hat{f}$ over the ball $B_{U_0}(x,Cm\varepsilon)$ with respect to $\varphi_{U_0}^2(z)dz$. Applying Poincar\'{e} inequality from Theorem \ref{lierlPI}, and then applying Lemma \ref{gradcomp},
    \begin{align*}
        II& \leq P'm^2\varepsilon^2\int_{B_{U_0}(x,C'm\varepsilon)}|\nabla \hat{f}(z)|^2\varphi_{U_0}^2(z)dz
     \lesssim m^2 \|\delta f\|^2_{2,B(x,C''m)}.
    \end{align*}
    For the term $I$, we have
    \begin{align}
          \nonumber \int_{B_{U_0}(y,\varepsilon/8)}|f(y)-\hat{f}(z)|^2 \varphi_{U_0}^2(z)dz &=\int_{B_{U_0}(y,\varepsilon/8)}\Big|f(y)-\sum_{t\in Y}f(t)\theta_t(z)\Big|^2\varphi_{U_0}^2(z) dz 
         \\ & \label{pi2.1}  = \int_{B_{U_0}(y,\varepsilon/8)}\Big|\sum_{t\in Y}(f(y)-f(t))\theta_t(z)\Big|^2\varphi_{U_0}^2(z)dz.
    \end{align}
    Note that (\ref{Y1}) and (\ref{pouY}) imply that if $d_{U_0}(z,y)<\varepsilon/8$, then $\{t\in Y:\theta_t(z)\neq 0\}\subseteq \{t\in Y:t\sim y\}\cup \{y\}$. Hence
    \begin{align*}
        \nonumber \int_{B_{U_0}(y,\varepsilon/8)}|f(y)-\hat{f}(z)|^2 \varphi_{U_0}^2(z)dz & = \int_{B_{U_0}(y,\varepsilon/8)}\Big|\sum_{t\in Y}(f(y)-f(t))\theta_t(z)\Big|^2\varphi_{U_0}^2(z)dz
        \\ & \lesssim \int_{B_{U_0}(y,\varepsilon/8)}\sum_{t\sim y}|f(y)-f(t)|^2 \varphi_{U_0}^2(z)dz
        \\ & \leq \sum_{t\sim y}|f(y)-f(t)|^2 m_{U_0}(y)
        \\ & = \delta f(y)^2 m_{U_0}(y),
    \end{align*}
    which implies 
    $$I\lesssim \frac{1}{\varepsilon^2}\sum_{y\in B(x,m)}\delta f(y)^2 m_{U_0}(y)\lesssim m^2\|\delta f\|^2_{2,B(x,C''m)}.$$
    Combining the estimates for the terms $I$ and $II$, we get
    $$\sum_{y\in B(x,m)}|f(y)-f_{B(x,m)}|^2 m_{U_0}(y)=\inf_{\alpha\in \mathbb{R}}\sum_{y\in B(x,m)}|f(y)-\alpha|^2 m_{U_0}(y)\lesssim m^2 \|\delta f\|^2_{2,B(x,C''m)},$$
    as desired.
\end{proof}

We now give the proof of $\varphi_U^2$-Poincar\'{e} inequalities (Theorem \ref{PIthm}) for scales $r\gtrsim \varepsilon$. 

\begin{proof}
    Given a ball $B_U(x,r)$ with $r\gtrsim \varepsilon$, note that
\begin{align}
    \label{cs1}
    B_U(x,r)\subseteq \bigcup_{z\in X\cap B_U(x,r+\varepsilon)}\Big\{B_U(x,r)\cap B_U(z,\varepsilon)\Big\}.
\end{align}
The inclusion (\ref{cs1}) holds since any point in $B_U(x,r)$ is $\varepsilon$-close to some $z\in X$. By (\ref{cs1}), for any $\alpha\in \mathbb{R}$ it follows that
\begin{align*}
    \int_{B_U(x,r)} |f(y)-\alpha|^2 \varphi^2_U(y)dy & \leq \sum_{z\in X\cap B_U(x,r+\varepsilon)}\int_{B_U(z,\varepsilon)}|f(y)-\alpha|^2 \varphi_U^2(y) dy
    \\ & = \sum_{z\in X\cap B_U(x,r+\varepsilon)}\int_{B_U(z,\varepsilon)} |f(y)-\widetilde{f}(z)+\widetilde{f}(z)-\alpha|^2 \varphi_U^2(y)dy
    \\ & \lesssim \sum_{z} \int_{B_U(z,\varepsilon)}|f(y)-\widetilde{f}(z)|^2\varphi_U^2(y)dy+\sum_{z}m_U(z)|\widetilde{f}(z)-\alpha|^2
    \\ & =: I_1 + I_2.
\end{align*}
The first term $I_1$ can be controlled since we already know from Theorem \ref{lierlPI} that (\ref{PIphi}) holds on scale $\varepsilon$; we have
\begin{align}
    \label{cs2}
    I_1\lesssim \varepsilon^2\sum_{z\in X\cap B_U(x,r+\varepsilon)}\int_{B_U(z,C_1\varepsilon)}|\nabla f(y)|^2\varphi_U^2(y)dy\lesssim \varepsilon^2\int_{B_U(x,C_2r)}|\nabla f(y)|^2\varphi_U^2(y)dy.
\end{align}
The second inequality in (\ref{cs2}) holds because of reasoning similar to (\ref{6.4.3}); the balls $B_U(z,C\varepsilon)$ have bounded overlap.

To manage the second term $I_2$, we pick $x_0\in X$ with $d_U(x,x_0)\leq \varepsilon$, and pick an integer $m\asymp r/\varepsilon$ so that the radius of the continuous ball $B_U(x,r)$ is roughly $m$ in the graph distance. Then applying discrete Poincar\'{e} inequality (Lemma \ref{dPI}) and Lemma \ref{weightcomp},
\begin{align}
    \nonumber
     I_2\lesssim m^2\sum_{y\in B(x_0,C_3m)}|\delta \widetilde{f}(y)|^2 m_{U_0}(y)\asymp m^2\sum_{y\in B(x_0,C_3m)}|\delta \widetilde{f}(y)|^2 m_{U}(y)
\end{align}
By Lemma \ref{gradcomp},
\begin{align*}
    I_2\lesssim m^2\varepsilon^2 \int_{B_U(x,C_4\varepsilon m )}|\nabla f(y)|^2 \varphi_U^2(y) dy\lesssim r^2\int_{B_U(x,C_5r)}|\nabla f(y)|^2\varphi_U^2(y)dy,  
\end{align*}
since $r\asymp m\varepsilon$. This proves the desired Poincar\'{e} inequality (\ref{PIphi}) for the ball $B_U(x,r)$.
\end{proof}

\subsection{Eigenfunction perturbation: annuli}
    \label{perturbannulisection}

The main results from Sections \ref{VDsubsection} and \ref{PIsubsection} imply the following.

\begin{Theo}    
    \label{equilibrium}
    Let $U_0\subseteq \mathbb{S}^{n-1}$ be a domain. For $\varepsilon\in (0,1]$, put $U=(1,1+\varepsilon)\times U_0\subseteq \mathbb{R}^n$. Suppose $U_0$ is $(C_0,c_0)$-inner uniform, and $U$ is $(C_0,c_0)$-locally inner uniform up to scale $\varepsilon$. There exist constants $c_1,c_2>0$ such that for all $t\geq \textup{diam}_U^2$,
    \begin{align}
        \label{equilibrium1}
        \sup_{x,y\in U}\Bigg|\frac{e^{t\lambda_U }p^D_U(t,x,y)}{\varphi_U(x)\varphi_U(y)}-1\Bigg|\leq c_1e^{-c_2 t/\textup{diam}_U^2}.
    \end{align}
    The constants $c_1,c_2$ depend only on $n,C_0,c_0,$ and a uniform upper bound on $\lambda(U_0)$.
\end{Theo}

Theorem \ref{equilibrium} is a consequence of $(\varphi_U^2\textup{-VD})$ and $(\varphi_U^2\textup{-PI})$, which we established in Theorems \ref{vdthm} and \ref{PIthm} respectively. It can be proven by following the reasoning in \cite{lscequil}; we omit the details. The fact that volume doubling and Poincar\'{e} inequalities together imply an estimate of the form (\ref{equilibrium1}) holds in great generality on compact domains of metric measure spaces; this implication has been used in, for example, \cite[Theorem ~1.2]{lierllsc}.

A priori, Theorem \ref{equilibrium} tells us that waiting for an amount of time comparable to $\text{diam}_U^2$ is sufficient for $e^{t\lambda_U}p^D_U(t,x,y)/\varphi_U(x)\varphi_U(y)$ to approach equilibrium. In the next lemma, we show that when $U$ is inner uniform, the inner diameter and usual diameter of $U$ are comparable, which allows us to replace $\text{diam}_U$ with $\text{diam}(U)$ in Theorem \ref{equilibrium}.

\begin{Lemma}
    \label{diamlemma}
    Let $U\subseteq \mathbb{S}^{n-1}$ be any $(C_0,c_0)$-inner uniform domain. Then 
    $$\textup{diam}(U)\leq \textup{diam}_{U}\leq  \frac{8}{c_0}\textup{diam}(U).$$
\end{Lemma}

\begin{proof}
    The lower bound is a trivial consequence of the fact that, for all $x,y\in U$, $d(x,y)\leq d_U(x,y)$. It remains to prove the upper bound. 
    
    Since $\overline{U}$ is compact, we can choose $o\in U$ with $d(o,\partial U)=\sup\{d(x,\partial U):x\in U\}$. Let $x\in U$ be arbitrary. From the definition of inner uniformity and (\ref{iu2}), choose a path $\gamma:[0,1]\to U$ with $\gamma(0)=x$ and $\gamma(1)=o$ such that for all $z\in \gamma([0,1])$, $$d(z,\partial U)\geq \frac{c_0}{2}\min\{d_U(x,z),d_U(o,z)\}.$$
    By continuity, we may pick $z\in \gamma([0,1])$ with $d_U(x,z)=d_U(o,z)$. Then 
    $$d_U(x,o)\leq d_U(x,z)+d_U(z,o)=2d_U(z,o)\leq \frac{4}{c_0}d(z,\partial U)\leq \frac{4}{c_0}d(o,\partial U)\leq \frac{4}{c_0}\text{diam}(U).$$
    Thus, for arbitrary $x,y\in U$, the preceding inequality gives
    $$d_U(x,y)\leq d_U(x,o)+d_U(o,y)\leq \frac{8}{c_0}\text{diam}(U). $$
    We finish the proof by taking supremum over all $x,y\in U$.
\end{proof}

Consider two annular domains $A\subseteq B\subseteq \mathbb{R}^n$. Our goal is to combine Theorem \ref{equilibrium} with the perturbation techniques developed in \cite{chaolsc} to prove an inequality of the form $\varphi_A\lesssim \varphi_U\lesssim \varphi_B$ (whenever both sides of the inequality are defined), for any arbitrary domain $U$ with $A\subseteq U\subseteq B$. The main result of this section is Theorem \ref{perturbthm}. For concrete examples of inequalities of the form $\varphi_A\lesssim \varphi_U\lesssim \varphi_B$, see Theorem \ref{illustration} as well as the examples in Section \ref{examples}.

Let $A_0,B_0\subseteq \mathbb{S}^{n-1}$ denote the spherical base of $A$ and $B$, respectively. We introduce several quantitative assumptions on the spherical bases $A_0$ and $B_0$ for our main result to hold. 

\begin{Assum}
    \label{assump1}
    \normalfont
Henceforth, we will assume that all $(C_0,c_0)$-inner uniform domains $U_0\subseteq \mathbb{S}^{n-1}$ of this section satisfy the following property: for all $\varepsilon\in (0,1]$, the domain $(1,1+\varepsilon)\times U_0$ is locally $(C_0,c_0)$-inner uniform up to scale $\varepsilon$.
\end{Assum}

\begin{Assum}
    \label{assump2}
    Let $A_0\subseteq \mathbb{S}^{n-1}$ be a domain and $\{B_0(c)\subseteq \mathbb{S}^{n-1}:c\geq 1\}$ be a family of domains such that $B_0(c)\supseteq A_0$ for all $c\geq 1$. We say that $A_0$ and $\{B_0(c):c\geq 1\}$ satisfy $\Phi$-volume growth if there is a continuous function $\Phi:[1,\infty)\to [1,\infty)$ with $\Phi(1)=1$ such that 
    \begin{align*}
        1\leq \frac{\sigma_{n-1}(B_0(c))}{\sigma_{n-1}(A_0)}\leq \Phi(c)
    \end{align*}
\end{Assum}

\begin{Assum}
    \label{assump3}
    Let $A_0\subseteq \mathbb{S}^{n-1}$ be a domain. We say that $A_0$ has $\Psi$-controlled boundary if either $A_0=\mathbb{S}^{n-1}$, or that for all $\delta\in [0,1/2]$, we have
    \begin{align*}
        \frac{\sigma_{n-1}(\{x\in A_0:\textup{dist}(x,\partial A_0)>\delta \textup{diam}_{A_0}\})}{\sigma_{n-1}(A_0)}\leq \Psi(\delta),
    \end{align*}
    for some continuous decreasing function $\Psi: [0,1/2] \to [0,1]$ with $\Psi(0)=1$.
\end{Assum}

For brevity, we also make the following convention.

\begin{Dfn}
    \normalfont
    In the remainder of this section, we say that an inequality $X\lesssim Y$ has an implied constant depending \textit{only on the key parameters} if $X\leq KY$ for a constant $K>0$ depending only on $n,C_0,c_0,C_1,C_2$ and a uniform upper bound on $\lambda(A_0)$. By \Quote{for all sufficiently small $\varepsilon>0$ depending \textit{only on the key parameters}}, we mean that a certain statement is true whenever $\varepsilon\in (0,\varepsilon_0)$, where $\varepsilon_0$ depends only on the aforementioned key parameters.  
\end{Dfn}

We now state and prove the main result of this section.

\begin{Theo}
    \label{perturbthm}
    Let $A_0\subseteq \mathbb{S}^{n-1}$ and $\{B_0(c)\subseteq \mathbb{S}^{n-1}:c\geq 1\}$ be $(C_0,c_0)$-inner uniform domains satisfying Assumption \ref{assump1} and $\Phi$-volume growth (Assumption \ref{assump2}).
    Let $\{a_{\varepsilon}\}_{\varepsilon>0}$ and $\{b_{\varepsilon}\}_{\varepsilon>0}$ be nonnegative real numbers, and assume that for some constants $C_1,C_2>0$, $a_{\varepsilon}\in [0,C_1\varepsilon^3)$ and $b_{\varepsilon}\in [0,C_2\varepsilon^3)$ for all $\varepsilon \in (0,1)$. Put
    \begin{align*}
        & A = (1,1+\varepsilon)\times A_0\subseteq \mathbb{R}^n,\hspace{0.4in} 
        B_c=(1-a_{\varepsilon},1+\varepsilon+b_{\varepsilon})\times B_0(c)\subseteq \mathbb{R}^n.
    \end{align*}
    \begin{enumerate}
        \item  (Upper bound) For all sufficiently small $\varepsilon>0$ depending only on $C_1$ and $C_2$, for any domain $U\subseteq \mathbb{R}^n$ with $A\subseteq U\subseteq B_c$, we have $\varphi_U(x)\lesssim \varphi_{B_c}(x)$ for all $x\in U$. The implied constant only depends on the key parameters.
        \item (Lower bound) Suppose $A_0$ has $\Psi$-controlled boundary (Assumption \ref{assump3}). For all sufficiently small $\varepsilon>0$ and $c\geq 1$ sufficiently close to $1$, for any domain $U\subseteq \mathbb{R}^n$ with $A\subseteq U\subseteq {B_c}$, we have $\varphi_U(x)\gtrsim \varphi_A(x)$ for all $x\in A$. The implied constant only depends on the key parameters and $\Psi$. Here, $\varepsilon$ and $|c-1|$ are both sufficiently small depending only on the key parameters. Additionally, $|c-1|$ is also taken sufficiently small depending on $\Phi$.
    \end{enumerate}

\end{Theo}

\begin{proof}
    \textit{Part I - upper bound.} The parameter $c$ will not play a role in this part of the proof, so abbreviate $B=B_c$. By taking $\varepsilon>0$ sufficiently small and by a rescaling argument, we may assume that the conclusion (\ref{equilibrium1}) of Theorem \ref{equilibrium} holds for the domain $B$. By domain monotonicity of the Dirichlet heat kernel, Theorem \ref{equilibrium}, and Lemma \ref{diamlemma}, we may fix $t\asymp \text{diam}(\mathbb{S}^{n-1})^2$ and write
    \begin{align*}
        e^{-\lambda_U t}\varphi_U^{2}(x) \leq p^D_U(t,x,x)\leq p^D_{B}(t,x,x) \lesssim e^{-t\lambda_B}\varphi_B^2(x).
    \end{align*}
    The upper bound $\varphi_U\lesssim \varphi_B$ would be obtained if $\lambda_U-\lambda_B>0$ is uniformly bounded above. We write
    \begin{align}
        \label{lUlB}
        \lambda_U-\lambda_B \leq  \lambda_A-\lambda(A_{1-a_{\varepsilon},1+\varepsilon+b_{\varepsilon}}) \leq \lambda(A_{1,1+\varepsilon})-\lambda(A_{1-a_{\varepsilon},1+\varepsilon+b_{\varepsilon}})+\lambda(A_0)
    \end{align}
     The first inequality in (\ref{lUlB}) uses domain monotonicity, and the second inequality uses Lemma \ref{eigvaldecomp}. Thus, Lemma \ref{eiggap} and the assumption that $\lambda(A_0)$ is uniformly bounded above imply that the right-hand side of (\ref{lUlB}) is uniformly bounded. Hence, $\varphi_U(x)\lesssim \varphi_B(x)$.

    \textit{Part II - lower bound, step 1.} For any $t>0$ and $x\in U$,
    \begin{align*}
        e^{-\lambda_U t}\varphi_U(x)=\int_U p^D_U(t,x,y)\varphi_U(y)dy\geq \int_{A}p^D_A(t,x,y)\varphi_U(y)dy.
    \end{align*}
    By Theorem \ref{equilibrium} and Lemma \ref{diamlemma} again, we fix $t\asymp \text{diam}(\mathbb{S}^{n-1})^2$ and obtain
    \begin{align*}
        e^{-\lambda_U t}\varphi_U(x) & \geq \int_{A}e^{-\lambda_A t}\varphi_A(x)\varphi_A(y)\varphi_U(y)dy.
    \end{align*}
    The eigenvalue gap estimate in (\ref{lUlB}) then implies
    \begin{align}
    \label{LB1.1}
        \frac{\varphi_U(x)}{\varphi_A(x)}\gtrsim\int_A \varphi_A(y)\varphi_U(y)dy. 
    \end{align}
    It thus remains to bound the right-hand side of (\ref{LB1.1}) below by a positive constant. 
    
    \noindent \textit{Part II - lower bound, step 2.} For $\delta\in (0,1/2)$, define
    $$A(\delta):=I(\delta)\times A_0(\delta):=(1+\delta\varepsilon,1+\varepsilon-\delta\varepsilon)\times \{x\in A_0:\text{dist}(x,\partial A_0)>\delta\text{diam}_{A_0}\}\subseteq A.$$
        (If $A_0=\mathbb{S}^{n-1}$, then we define $A(\delta)=I(\delta)\times \mathbb{S}^{n-1}$.) The rest of the proof proceeds similarly as \cite[Theorem~4.3]{chaolsc}; we sketch the main details. For any $\delta\in (0,1/2)$, it follows from Lemma \ref{lierlcompare} that
    \begin{align}
        \label{LB1.2}
        \inf_{y\in A(\delta)}\varphi_A(y) \gtrsim \frac{\delta\varepsilon}{\varepsilon^{3/2}}\cdot \frac{1}{H^{1/\delta}\sqrt{\sigma_{n-1}(A_0)}},
    \end{align}
    where $H>1$ is some constant depending only on $C_0,c_0,n$. The exact value of $H$ is unimportant; at the end of the proof we will take $\delta>0$ sufficiently small but bounded away from $0$ so that $H^{-1/\delta}$ is bounded away from $0$. 
    
    Also, by the already proven upper bound $\varphi_U\lesssim \varphi_{B_c}$ and Lemma \ref{lierlcompare}, we have
    \begin{align}
        \label{LB1.3}
        \|\varphi_U\|_{L^{\infty}(A(\delta))}\lesssim \|\varphi_{B_c}\|_{L^{\infty}(B_c)}\lesssim \frac{1}{\sqrt{\varepsilon+a_{\varepsilon}+b_{\varepsilon}}}\cdot \frac{1}{\sqrt{\sigma_{n-1}(B_0(c))}}.
    \end{align}
    By (\ref{LB1.2}) and (\ref{LB1.3}),
    \begin{align}
        \nonumber \int_A\varphi_A(y)\varphi_U(y)dy & \geq \inf_{y\in A(\delta)}\varphi_A(y) \int_{A(\delta)}\varphi_U(y)dy
        \\ \nonumber &\gtrsim \inf_{y\in A(\delta)}\varphi_A(y)\frac{\|\varphi_U\|^2_{L^2(A(\delta))}}{\|\varphi_U\|_{L^{\infty}(A(\delta))}}
        \\ \nonumber & \gtrsim \frac{\delta \varepsilon}{\varepsilon^{3/2}}\cdot \frac{1}{H^{1/\delta}}\sqrt{\frac{\sigma_{n-1}(B_0(c))}{\sigma_{n-1}(A_0)}}\sqrt{\varepsilon+a_{\varepsilon}+b_{\varepsilon}} \|\varphi_U\|^2_{L^2(A(\delta))}
        \\ \label{LB1.4} & \gtrsim \frac{\delta}{H^{1/\delta}}\|\varphi_U\|^2_{L^2(A(\delta))},
    \end{align}
    where the last inequality holds provided $\varepsilon>0$ (depending on $C_1,C_2$) is chosen sufficiently small. 

    \textit{Part II - lower bound, step 3.} In view of (\ref{LB1.4}), we will choose a $\delta>0$ so that $\|\varphi_U\|^2_{L^2(A(\delta))}$ is bounded below. First consider $A$ instead of $A(\delta)$ and note that by the already established upper bound $\varphi_U\lesssim \varphi_{B_c}$ and Lemma \ref{lierlcompare},
    \begin{align*}
        \int_A \varphi_U^2(y)dy=1-\int_{U\backslash A}\varphi_U^2(y)dy \geq 1-C\int_{{B_c}\backslash A}\varphi_{B_c}^2(y)dy\geq 1-C\frac{|B_c\backslash A|}{(\varepsilon+a_{\varepsilon}+b_{\varepsilon})\sigma_{n-1}(B_0(c))}.
    \end{align*}
    This gives us 
    \begin{align*}
        \int_A\varphi_U^2(y)dy&\geq 1-C\Bigg(1-\frac{\sigma_{n-1}(A_0)\int_1^{1+\varepsilon}r^{n-1}dr }{\sigma_{n-1}(B_0(c))\int_{1-a_{\varepsilon}}^{1+\varepsilon+b_{\varepsilon}}r^{n-1}dr }\Bigg)
        \\ & \geq 1-C\Bigg(1-\frac{1}{\Phi(c)}\cdot \frac{\int_1^{1+\varepsilon}r^{n-1}dr }{\int_{1-a_{\varepsilon}}^{1+\varepsilon+b_{\varepsilon}}r^{n-1}dr }\Bigg)
    \end{align*}
     By choosing $\varepsilon>0$  sufficiently small and choosing $c>1$ sufficiently close to $1$, we get 
    \begin{align}
        \label{LB1.5}
        \int_{A}\varphi_U^2(y)dy\geq \frac{1}{2}.
    \end{align}
    \textit{Part II - lower bound, step 4.} We complete the proof by showing that (\ref{LB1.5}) holds with $A$ replaced by $A(\delta)$, and with $1/2$ replaced by a smaller positive constant. Using the upper bound $\varphi_U\lesssim \varphi_{B_c}$ and Lemma \ref{lierlcompare} again,
    \begin{align*}
        \int_{A(\delta)}\varphi_U^2(y)dy&=\int_A\varphi_U^2(y)-\int_{A\backslash A(\delta)} \varphi_U^2(y)dy
        \\ & \geq \frac{1}{2}-C\frac{|A\backslash A(\delta)|}{(\varepsilon+a_{\varepsilon}+b_{\varepsilon})\sigma_{n-1}(B_0(c))}
        \\ & \geq \frac{1}{2} - C\frac{|A\backslash A(\delta)|}{\varepsilon \cdot \Phi(c)\cdot \sigma_{n-1}(A_0)}
        \\ & \geq \frac{1}{2}-C\frac{|A\backslash A(\delta)|}{|A|}.
    \end{align*}
     Because $A_0$ satisfies Assumption \ref{assump3}, the right-hand side can be made greater than (say) $1/4$ by picking $\delta>0$ appropriately depending on $C$ and the function $\Psi$ from Assumption \ref{assump3}. 
     
     In view of (\ref{LB1.4}), the proof is now complete.
\end{proof}

Theorem \ref{perturbthm} immediately implies the following result regarding the stability of the principal Dirichlet eigenfunction when perturbing the domain in the radial direction.

\begin{Cor}
    \label{perturbcor}
   Let $A_0\subseteq \mathbb{S}^{n-1}$ be a $(C_0,c_0)$-inner uniform domain satisfying Assumption \ref{assump1}. Let $\{a_{\varepsilon}\}_{\varepsilon>0}$ and $\{b_{\varepsilon}\}_{\varepsilon>0}$ be nonnegative real numbers such that for some $C_1,C_2>0$, $a_{\varepsilon}\in [0,C_1\varepsilon^3)$ and $b_{\varepsilon}\in [0,C_2\varepsilon^3)$ for all $\varepsilon\in (0,1)$. Put
   $$A=(1,1+\varepsilon)\times A_0\subseteq \mathbb{R}^n,\hspace{0.2in}B=(1-a_{\varepsilon},1+\varepsilon+b_{\varepsilon})\times A_0\subseteq \mathbb{R}^n.$$
   For all sufficiently small $\varepsilon>0$ depending only on the key parameters, for any domain $U$ with $A\subseteq U\subseteq B$, we have
   \begin{align}
        \label{stability1}
       \frac{\min\{r-1,1+\varepsilon-r\}}{\varepsilon^{3/2}}\varphi_{A_0}(\theta)\lesssim \varphi_U(r,\theta)\lesssim \frac{\min\{r-1,1+\varepsilon-r\}}{\varepsilon^{3/2}}\varphi_{A_0}(\theta),
   \end{align}
   for any $ r\in [1+a_{\varepsilon},1+\varepsilon-b_{\varepsilon}]$ and $\theta\in A_0$. The implied constants in both the lower and upper bounds of (\ref{stability1}) depend only on the key parameters. Additionally, the implied constant in the lower bound also depends on $\Psi$.  
\end{Cor}

\subsection{Eigenfunction perturbation: Euclidean boxes}
\label{sectionbox}

In a fixed dimension $n\geq 2$, the collection of all boxes is not inner uniform (consider a box with one side much longer than all of the other sides). However, by Theorem \ref{boxHK}, any box $B$ satisfies ($\varphi_B^2$-VD), ($\varphi_B^2$-PI), and Dirichlet heat kernel estimates (\ref{HKE}) uniformly with all constants depending only on $n$. Theorem \ref{boxHK} is only included for completeness and will not be used subsequently.

\begin{Theo}
    \label{boxHK}
    Let $B=\prod_{i=1}^{n}(-a_i,a_i)\subseteq \mathbb{R}^n$ be a box. Let $\varphi_B$ denote the Dirichlet Laplacian eigenfunction of $B$, normalized so that $\|\varphi_B\|_{L^2(B)}=1$. Then $B$ satisfies ($\varphi_B^2$-$\textup{VD}$) and ($\varphi_B^2$-$\textup{PI}$), with constants depending only on the dimension $n$. Consequently, $B$ satisfies Dirichlet heat kernel estimates of the form (\textup{\ref{HKE}}) with constants depending only on the dimension $n$. 
\end{Theo}

\begin{proof}
    Note that the collection of bounded intervals in $\mathbb{R}$ is uniformly inner uniform, so by (for example) results of \cite{lierllsc}, $I=I_i:=(-a_i,a_i)$ satisfies ($\varphi^2_{I}$-VD) and ($\varphi_{I}^2$-PI). The rest of the proof is just a routine argument showing that volume doubling and Poincar\'{e} inequalities are preserved under Cartesian products; we leave out the details. (To make sure that the balls in the definitions of ($\varphi_B^2$-\text{VD}) and ($\varphi_B^2$-\text{PI}) factor into balls of $I_i$, we observe that the Euclidean distance $|x-y|$ on $B$ is comparable to $\max_i |x_i-y_i|$.) 
\end{proof}

For the remainder of this section, we aim to prove Theorem \ref{rectanglecomparison}. Recall that every Dirichlet Laplacian eigenfunction of $B=\prod_{i=1}^{n}(-a_i,a_i)$ is of the form
    $$(x_1,x_2,...,x_n)\mapsto \prod_{i=1}^{n}\frac{1}{\sqrt{a_i}}\sin\Big(\frac{N_i\pi (x_i+a_i)}{2a_i}\Big),\hspace{0.1in}N_1,N_2,...,N_n\in \{1,2,3,...\}.$$
    In particular, $\varphi_B$ corresponds to $N_i=1$ for all $i$, and is given by (\ref{boxphilambda}).

We start with an analog of Theorem \ref{equilibrium} for Euclidean boxes. 

\begin{Lemma}
    \label{equilibriumbox}
    Let $B=\prod_{i=1}^{n}(-a_i,a_i)\subseteq \mathbb{R}^n$ be a box. For all $x,y\in B$ and $t>0$, we have
    \begin{align}
        \label{equilibriumbox1}
     \frac{e^{\lambda_B t}p^D_B(t,x,y)}{\varphi_B(x)\varphi_B(y)}\lesssim \prod_{i=1}^{n}\Big\{1+\Big(\frac{a_i}{\sqrt{t}}\Big)^3\Big\}.
    \end{align}
    Moreover, if $t\geq a_i^2$ for all $i=1,2,...,n$, then
    \begin{align}
        \label{equilibriumbox2}
    \prod_{i=1}^{n}\Big\{1-\Big(\frac{a_i}{\sqrt{t}}\Big)^3\Big\}\lesssim \frac{e^{\lambda_B t}p^D_B(t,x,y)}{\varphi_B(x)\varphi_B(y)}.
    \end{align}
    All of the implied constants depend only on $n$.
\end{Lemma}

\begin{proof}
    Suppose first the dimension $n=1$ and $B=(-a,a)$. Let $\lambda_j(B)=(j\pi)^2/(2a)^2$ denote the Dirichlet eigenvalues of $B$ in increasing order, and let $\varphi_j(x)=a^{-1/2}\sin(j\pi (x+a)/(2a))$ be the corresponding eigenfunctions. For all $t>0$,
    \begin{align*}
        \Bigg|\frac{e^{\lambda_B t}p^D_B(t,x,y)}{\varphi_B(x)\varphi_B(y)}-1\Bigg|& \leq  \sum_{j=2}^{\infty}e^{-\lambda_j(B)t+\lambda_1(B)t}\Bigg\{\sup_B\Bigg |\frac{\varphi_j}{\varphi_B}\Bigg|\Bigg\}^2
        \\ & = \sum_{j=2}^{\infty}\exp\Big(t\Big(\frac{\pi}{2a}\Big)^2(1-j^2)\Big)j^2
        \\ & \leq \sum_{j=2}^{\infty}\exp\Big(-t\Big(\frac{\pi}{2a}\Big)^2 \frac{3j^2}{4}\Big)j^2.
    \end{align*}
    An integral estimate yields
    \begin{align}
        \label{intestimate}\Bigg|\frac{e^{\lambda_B t}p^D_B(t,x,y)}{\varphi_B(x)\varphi_B(y)}-1\Bigg|\lesssim \int_0^{\infty}\exp\Big(-\frac{3\pi^2 t x^2}{16a}\Big)x^2 dx\lesssim \Big(\frac{a}{\sqrt{t}}\Big)^3,
    \end{align}
    with all implied constants absolute. When $n=1$, the estimate (\ref{intestimate}) implies (\ref{equilibriumbox1}) and (\ref{equilibriumbox2}). For a general box $B=\prod_{i=1}^{n}(-a_i,a_i)$ in $\mathbb{R}^n$, note that we have the identity
    $$\frac{e^{\lambda _Bt}p^D_B(t,x,y)}{\varphi_B(x)\varphi_B(y)}=\prod_{i=1}^{n}\frac{e^{\lambda_{(-a_i,a_i)}}p^D_{(-a_i,a_i)}(t,x_i,y_i)}{\varphi_{(-a_i,a_i)}(x_i)\varphi_{(-a_i,a_i)}(y_i)},$$
    which together with (\ref{intestimate}) gives the desired estimates (\ref{equilibriumbox1}) and (\ref{equilibriumbox2}). 
\end{proof}

Now we give the proof of Theorem \ref{rectanglecomparison}.

\begin{proof}
    The proof is structurally similar to that of Theorem \ref{perturbthm}, so we only focus on the difference between the two proofs.

    For the upper bound, domain monotonicity of the Dirichlet heat kernel and Lemma \ref{equilibriumbox} imply that for all $t>0$
    \begin{align}
        \label{rect1}
        e^{-\lambda_{B_1}t}\varphi_U^2(x)\leq e^{-\lambda_U t}\varphi_U^2(x)\lesssim e^{-\lambda_{B_2}t} \varphi_{B_2}^2(x)\cdot \prod_{i=1}^{n}\Bigg\{1+\Big(\frac{b_i}{\sqrt{t}}\Big)^3\Bigg\}.
    \end{align}
    Fixing $t\asymp \max_i b_i^2$, assumption (\ref{recthyp1}) of the theorem gives
    \begin{align}
        \label{rect2}
        (\lambda_{B_1}-\lambda_{B_2})t\lesssim (\max_ib_i^2)\sum_{i=1}^{n}\Big(\frac{1}{a_i^2}-\frac{1}{b_i^2}\Big)\leq C_1n.
    \end{align}
    Combining (\ref{rect1}) and (\ref{rect2}) gives the upper bound. 
    
    For the lower bound, the inequality (\ref{rect2}) and the upper bound $\varphi_U\lesssim \varphi_{B_2}$ are sufficient to adapt the proof of Theorem \ref{perturbthm}; we get the inequalities
    \begin{align}
        \label{adapt1}
        \frac{\varphi_U(x)}{\varphi_{B_1}(x)}\gtrsim \cos\Big(\frac{\pi(1-\delta)}{2}\Big)^n\int_{(1-\delta)B_1}\varphi_U^2(y)dy
    \end{align}
    and 
    \begin{align}
        \label{adapt2}
        \int_{B_1(\delta)} \varphi_U^2(y)dy&\geq 1-C\int_{B_2\backslash B_1(\delta) }\varphi_{B_2}^2(y)dy
        \\ \label{adapt3} & \geq 1-C'\Big(\frac{(b_1\cdots b_n)-(1-\delta)^n (a_1\cdots a_n)}{b_1\cdots b_n} \Big)
        \\ \label{adapt4} & \geq 1-C'\Big(1-\frac{(1-\delta)^n}{C_2^n}\Big)
    \end{align}
    We used assumption (\ref{recthyp2}) to go from (\ref{adapt3}) to (\ref{adapt4}). The implied constant in (\ref{adapt1}) and the constants $C,C'>1$ in (\ref{adapt2}), (\ref{adapt3}) all only depend on $n$ and $C_1$. Thus we can choose $C_2=C_2(C_1,n)$ sufficiently close to $1$ so that the right-hand side of (\ref{adapt4}) is positive for some $\delta\in (0,1)$. For this particular value of $\delta$, combining (\ref{adapt1}) and (\ref{adapt4}) gives the desired lower bound.
\end{proof}

\section{Appendix}

\subsection{A partition of unity adapted to the \texorpdfstring{$\varepsilon$}{TEXT}-net \texorpdfstring{$X$}{TEXT}}\label{partitionofunity}

Given an $\varepsilon$-net $X$ as in Section \ref{PIsubsection}, we construct a partition of unity $\{\theta_x\}_{x\in X}$ with the properties (\ref{pou}). The construction of such a partition of unity can be found in, for example \cite[Page~235]{kanai}, and has been used by many authors including \cite{coulsc} and  \cite{muruchain}. For the existence of $\{\theta_x\}_{x\in X}$ in a metric measure space, the reader may find \cite[Lemma~2.5]{muruchain} helpful. We now explain the construction of $\{\theta_x\}_{x\in X}$ in our context. First, we note that for any subset $V\subseteq U$,  $\|\nabla d_U(\cdot,V)\|_{\infty}\leq 1$, see the discussion on \cite[Page~40]{gyryalsc}. 
For each $x\in X$, we consider the bump function
$$\overline{\theta}_x(\cdot):=\Big(1-\frac{d_U(\cdot,B_U(x,\varepsilon/2))}{\varepsilon}\Big)_+,$$
and we put $\theta_x(\cdot) =\overline{\theta}_x(\cdot)/\sum_{x'\in X}\bar{\theta}_{x'}(\cdot )$. This gives a partition of unity satisfying (\ref{pou}).

\subsection{Neumann heat kernel estimates for thin annular domains}\label{neumann}

All of our main results as well as examples in Section \ref{examples} are with Dirichlet boundary conditions, but some of them also hold under Neumann boundary conditions. For a bounded domain $U\subseteq \mathbb{R}^n$, the smallest eigenvalue of the Neumann Laplacian $-\Delta$ equals zero, corresponding to eigenfunction $\varphi_U^{N}\equiv 1/\sqrt{|U|}$, where $|U|$ is the Lebesgue measure of $U$. The proofs of volume doubling and Poincar\'{e} inequalities with respect to the Dirichlet eigenfunction $\varphi_U$ are easily adapted for the Neumann eigenfunction $\varphi^{N}_U$ (and in fact much simpler because all principal eigenfunctions are constants). This gives us the following result.

    \begin{Theo}
        \label{neumannHK}
        Let $U_0=\mathbb{S}^{n-1}$ or let $U_0\subseteq \mathbb{S}^{n-1}$ be a $(C_0,c_0)$-inner uniform domain. Fix $K>1$ and consider an annular domain $U=(a,b)\times U_0$ with $b/a\in (1,K)$. Suppose $U$ is locally $(C_0,c_0)$-inner uniform at scale $b-a$. Let $\varphi_U=\varphi_U^{N}=1/\sqrt{|U|}$ denote the first Neumann eigenfunction of $U$. Then $U$ satisfies ($\varphi_U^2$\textup{-VD}) and ($\varphi_U^2$\textup{-PI}) with constants depending only on $n,C_0,c_0$ and $K$. Consequently, the Neumann heat kernel $p^N_U(t,x,y)$ admits the estimates
        \begin{align}
            \label{neumannHK1}
            \frac{c_1\exp(-\frac{d_U(x,y)^2}{c_2t})}{\sqrt{|B_U(x,\sqrt{t})|}\sqrt{|B_U(y,\sqrt{t})|}}\leq p^N_U(t,x,y)\leq \frac{c_3\exp(-\frac{d_U(x,y)^2}{c_4t})}{\sqrt{|B_U(x,\sqrt{t})|}\sqrt{|B_U(y,\sqrt{t})|}},
        \end{align}
        where $c_1,c_2,c_3,c_4$ are constants depending only on $n,C_0,c_0,K$. In (\ref{neumannHK1}), $d_U$ is the geodesic distance on $U$, and $|B_U(x,\sqrt{t})|$ denotes the Lebesgue measure of a ball $B_U(x,\sqrt{t})\subseteq U$ with respect to $d_U$. 
    \end{Theo}   

    In the regime $b/a\to \infty$, Neumann heat kernel estimates (\ref{neumannHK1}) still hold uniformly as long as the class of domains under consideration is uniformly inner uniform; see \cite{gyryalsc} and \cite{lierllsc}. Theorem \ref{neumannHK} gives many families of thin annular domains which are not inner uniform but satisfy Neumann heat kernel estimates of the form (\ref{neumannHK1}) uniformly. Hence, Theorem \ref{neumannHK} complements Neumann heat kernel estimates on inner uniform domains studied by Gyrya and Saloff-Coste \cite{gyryalsc}. For example, Theorem \ref{neumannHK} applies to the domains in Examples \ref{harmonicexample}, \ref{unitcircle}, \ref{2sphere}, and \ref{vonkochexample}.  Following an argument similar to the proof of Theorem \ref{boxHK}, we can also deduce that Neumann heat kernel estimates of the form (\ref{neumannHK1}) hold uniformly for all boxes in $\mathbb{R}^n$, with implied constants depending only on $n$. 

    We do not write out the proof of Theorem \ref{neumannHK}, but just briefly comment on how it is similar to the case of Dirichlet boundary conditions. To prove volume doubling and Poincar\'{e}, it suffices to (by a dilation argument) handle the case when the inner radius $a=1$. The hypothesis that $b/a$ is uniformly bounded above in Theorem \ref{neumannHK} implies that $d_U$ is comparable to the maximum of the metrics on each of its factors $(a,b)$ and $U_0$. This yields an analog of Lemma \ref{distcomp0}, which in turn allows us to prove volume doubling as in Theorem \ref{vdthm}.
    Next, the hypothesis of local inner uniformity implies Poincar\'{e} inequalities $(\varphi_U^2\text{-PI})$ (Definition \ref{pidefn}) at scale $b-a$. Since we have Neumann boundary conditions, this assertion about ($\varphi_U^2$-PI) is not easy to find directly in the literature, but it essentially follows from \cite[Theorem~3.34]{gyryalsc}. For clarity, we provide a short proof at the end of this subsection based on the results of \cite{lierllsc}.  
    
    These observations allow us to imagine the \Quote{thin spherical shell} $U$ as \textit{quasi-isometric} to the spherical domain $U_0$, and they suffice to carry out again the discretization argument of Coulhon and Saloff-Coste \cite{coulsc} of Section \ref{PIsubsection}. This way, we obtain ($\varphi_U^2$-PI) as in Theorem \ref{PIthm}. Combining ($\varphi_U^2$-VD) and ($\varphi_U^2$-PI) then gives us (\ref{HKE}).

    Of course, because the first Neumann eigenfunction $\varphi_U$ is a constant function, the Neumann heat semigroup is unchanged after a $h$-transform by $\varphi_U$. So the properties ($\varphi_U^2$-VD) and ($\varphi_U^2$-PI) with respect to the weighted measure $\varphi_U^2dx$ should really be thought of as unweighted inequalities just with respect to Lebesgue measure.
    We still choose to explain Neumann heat kernel estimates in terms of ($\varphi_U^2$-VD) and ($\varphi_U^2$-PI) to keep the exposition aligned with our framework.    

    \begin{Prop}
        Let $U_0\subseteq \mathbb{S}^{n-1}$ be a domain. In polar coordinates, consider the domain $U=(1,1+\varepsilon)\times U_0$. Suppose $U$ is $(C_0,c_0)$-locally inner uniform up to scale $\varepsilon$ (see Definition \ref{liu}). Let $\varphi_U\equiv 1/\sqrt{|U|}$ denote the first Neumann eigenfunction of the domain $U$. Then, $U$ satisfies $(\varphi_U^2$-\textup{VD}) and $(\varphi_U^2$-\textup{PI}) up to scale $\varepsilon$. That is, fixing an auxiliary integer $N$, the properties $(\varphi_U^2$-\textup{VD}) and $(\varphi_U^2$-\textup{PI}) are satisfied on all balls of radius $r\in (0,N\varepsilon)$, with implicit constants depending only on $n,C_0,c_0,N$.    
    \end{Prop}

    \begin{proof} 
        When $\varphi_U$ is the first Dirichlet eigenfunction of the domain $U$, the desired result already follows from \cite[Proposition~6.12(ii)]{lierllsc}. The key ingredient of that proposition relies on a \textit{Carleson-type estimate} proven in \cite[Proposition~5.10]{lierllsc} of the same paper. In our setting, this type of estimate informally states the following: let $U$ be a domain which is locally $(C_0,c_0)$-inner uniform at scale $\varepsilon$. Take a ball $B_U(x,r)$ with center $x\in \partial U$ and radius $r\in (0,\varepsilon)$, $r\asymp \varepsilon$. Then, for all $y\in B_U(x,r)$, we have $\varphi_U(y)\lesssim \varphi_U(x_r)$, where $x_r$ is any point with $d_U(x,x_r)=r/4$ and $d(x_r,\partial U)\geq c_0r/8$. Such points $x_r$ exist because of \cite[Lemma~3.20]{gyryalsc}. In other words, the values of $\varphi_U$ on boundary balls are uniformly controlled by the value $\varphi_U(x_r)$ at a point $x_r$ uniformly separated from the boundary at scale $r$. This is the main analytic inequality that is needed to verify ($\varphi_U^2$-VD) and ($\varphi_U^2$-PI) for the first Dirichlet eigenfunction $\varphi_U$, and is the primary reason why ($\varphi_U^2$-VD) and ($\varphi_U^2$-PI) are often more difficult to prove in the case of Dirichlet boundary conditions. The argument in \cite{lierllsc} makes it clear that many other nonnegative weight functions on $U$ also yield weighted volume doubling and Poincar\'{e} inequalities, as long as a Carleson-type estimate is satisfied for the weight. In particular, we can certainly consider the simplest possible weight function by choosing the weight $\varphi_U$ to be a positive constant function, such as the first Neumann eigenfunction of $U$, in which case the Carleson-type estimate $\varphi_U(y)\lesssim \varphi_U(x_r)$ is trivial because $\varphi_U\equiv 1/\sqrt{|U|}$ is a constant function. Therefore, after replacing the boundary conditions from Dirichlet to Neumann, a significantly simpler version of the arguments from \cite{lierllsc} yields the desired conclusion. 
    \end{proof}
\subsection{A counterexample to volume doubling}
\label{counterexample}

In this section, we give an explicit example of a family of bounded convex domains $U\subseteq \mathbb{R}^2$ which fails to satisfy $(\varphi_U^2\textup{-VD})$ (Definition \ref{vddefn}) uniformly. This shows that the family of metric measure spaces $(\overline{U},d,\varphi_U^2 dx)$, where $d$ is the Euclidean distance and where $U$ ranges over all convex domains in $\mathbb{R}^2$, fails to be uniformly $\varphi_U^2$-volume doubling. By considering Cartesian products of domains, the same is true with $\mathbb{R}^2$ replaced by $\mathbb{R}^n$ for $n\geq 3$.

For small $\beta>0$, consider the circular sector
$$U_{\beta}:=\{(r,\theta)\in \mathbb{R}^2:r\in (0,1),\theta \in (0,\pi \beta)\}.$$
The first Dirichlet Laplacian eigenfunction of $U_{\beta}$, up to an $L^2(U_{\beta})$-normalizing constant so that $\int_{U_{\beta}}\varphi_{U_{\beta}}^2dx=1$, is given by
$$\varphi_{U_{\beta}}(r,\theta)=J_{1/\beta}(\alpha_{\beta}r)\sin\Big(\frac{\theta}{\beta}\Big),$$
where $J_{1/\beta}(z)$ is the Bessel function of the first kind of order $1/\beta$, i.e.
$$J_{1/\beta}(r)=\sum_{k=0}^{\infty}(-1)^{k}\frac{(r/2)^{1/\beta+2k}}{k!\Gamma(k+1/\beta+1)},$$
and where $\alpha_{\beta}$ is the smallest positive zero of $J_{1/\beta}(r)$. On $U_{\beta}$, let $V_{\beta}(x,r)$ denote the measure of $B(x,r)\cap U_{\beta}$ with respect to $\varphi^2_{U_{\beta}}dx$.

\begin{Lemma}
    \label{bessellemma}\begin{enumerate}[leftmargin=*]
        \item For all $r\in \mathbb{R}$ and $\beta>0$, 
        $$J_{1/\beta}(r)=\frac{1}{\Gamma(1/\beta+1/2)\sqrt{\pi}}\Big(\frac{r}{2}\Big)^{1/\beta}\int_{-1}^{1}(1-t^2)^{\frac{1}{\beta}-\frac{1}{2}}\cos(rt)dt.$$
        \item If $\alpha_{\beta}>0$ is the smallest positive zero of $J_{1/\beta}(r)$, then, as $\beta\to 0$,
        $$\alpha_{\beta}=\frac{1}{\beta}+\frac{c}{\beta^{1/3}}+O(\beta^{1/3}),$$
        where $c\approx 1.855757$ is an absolute constant. 
    \end{enumerate}
\end{Lemma}
\begin{proof}
    The expression for $J_{1/\beta}(r)$ follows from \cite[Page~345]{jones}, while the asymptotics $\alpha_{\beta}$ are from \cite[Page~521]{watson}.
\end{proof}

Now we explicitly construct a sequence of balls in $U_{\beta}$ violating the volume doubling condition. Consider $B_{\beta}:=\{(r,\theta)\in \mathbb{R}^2:0<r<1/(2\alpha_{\beta}),\theta\in (0,\pi \beta)\}\subseteq U_{\beta}$. It follows from Proposition \ref{besVD}, to be given below, that $V_{\beta}(2B_{\beta})/V_{\beta}(B_{\beta})\to \infty$ as $\beta\to 0$, which implies that volume doubling with respect to $\varphi_{U_{\beta}}^2 dx$ fails to uniformly hold for the domains $\{U_{\beta}\}_{\beta>0}$. 

\begin{Prop} 
    \label{besVD}
    As $\beta\to 0$,
    $$V_{\beta}\Big(0,\frac{1}{\alpha_{\beta}}\Big)\sim \frac{\beta^4}{2\pi J_{\frac{1}{\beta}+1}(\alpha_{\beta})^2}\Big(\frac{e\beta}{2}\Big)^{2/\beta}\hspace{0.1in}\text{and}\hspace{0.2in}V_{\beta}\Big(0,\frac{1}{2\alpha_{\beta}}\Big)\sim \frac{\beta^4}{8\pi J_{\frac{1}{\beta}+1}(\alpha_{\beta})^2}\Big(\frac{e\beta}{4}\Big)^{2/\beta}$$ 
\end{Prop}

\begin{proof}
    We calculate that as $\beta\to 0$,
\begin{align}
    & \label{bes1}\int_{0}^{\pi \beta}\int_{0}^{1/\alpha_{\beta}} J^2_{1/\beta}(\alpha_{\beta}r)\sin^2(\theta/\beta)r dr d\theta 
    \\ \nonumber & = \frac{ \pi \beta}{2} \int_{0}^{ 1/\alpha_{\beta}} J_{1/\beta}^2(\alpha_{\beta}r) r dr
    %The exact value for the \theta integral is \pi \beta/2
    \\ \nonumber & = \frac{\beta}{2} \int_0^{1/\alpha_{\beta}} r\frac{1}{\Gamma(1/\beta+1/2)^2}\Big(\frac{\alpha_{\beta}r}{2}\Big)^{2/\beta}\Big(\int_{-1}^{1}(1-t^2)^{\frac{1}{\beta}-\frac{1}{2}}\cos(\alpha_{\beta}rt)dt\Big)^2 dr
    \\ \nonumber & = \frac{\beta}{2} \int_0^{1} \frac{r}{\alpha_{\beta}^2}\frac{1}{\Gamma(1/\beta+1/2)^2}\Big(\frac{r}{2}\Big)^{2/\beta}\Big(\int_{-1}^{1}(1-t^2)^{\frac{1}{\beta}-\frac{1}{2}}\cos(rt)dt\Big)^2 dr
    \\ \nonumber & \sim \frac{\beta^3}{4\pi }\Big(\frac{e\beta}{2}\Big)^{2/\beta}\int_0^{1} r^{1+2/\beta}\Big(\int_{-1}^{1}(1-t^2)^{\frac{1}{\beta}-\frac{1}{2}}\cos(rt)dt\Big)^2 dr.
\end{align}
In the last display, we used Stirling's formula and Lemma \ref{bessellemma}. By Laplace's method for asymptotic expansion of integrals, we have, as $\beta \to 0$,
$$\int_{-1}^{1}(1-t^2)^{\frac{1}{\beta}-\frac{1}{2}}\cos(rt)dt\sim \sqrt{\pi \beta},$$
with the convergence being uniform over all $r\in [0,1]$. It follows that 
$$V_{\beta}(0,1/\alpha_{\beta})\sim \frac{\beta^3}{4\pi}\Big(\frac{e\beta}{2}\Big)^{2/\beta}\cdot \frac{\pi\beta^2}{2}=\frac{\beta^5}{8}\Big(\frac{e\beta}{2}\Big)^{2/\beta},$$
as $\beta\to 0$. On the other hand, by \cite[Page~284, Eqn.~(9)]{strauss},
\begin{align}
    \label{bes2} \|\varphi_{U_{\beta}}\|_{L^2(U_{\beta})}^2=\int_{0}^{\pi\beta}\int_0^{1}J_{1/\beta}^2(\alpha_{\beta}r)\sin^2(\theta/\beta)rdrd\theta=\frac{\pi\beta}{2}\cdot \frac{1}{2}J_{\frac{1}{\beta}+1}(\alpha_{\beta})^2.
\end{align}
This gives the first asymptotic in Proposition \ref{besVD} since $V_{\beta}(0,1/\alpha_{\beta})$ equals (\ref{bes1}) divided by (\ref{bes2}). The second asymptotic is by a similar calculation.    
\end{proof}

\begin{Exm}
    \normalfont
    On the $2$-sphere $\mathbb{S}^2$, we give another example of bounded domains $\{U_{\alpha}\}$ that do not satisfy ($\varphi^2_{U_{\alpha}}$-VD) uniformly over $\alpha>0$. Parametrize the sphere by spherical coordinates: $\mathbb{S}^2=\{(\theta,\psi):\theta\in [0,2\pi),\psi\in [0,\pi)\}$, and consider the wedges $U_{\alpha}=\{(\theta,\psi)\in\mathbb{S}^2:\theta\in (0,\alpha)\}$ parametrized by $\alpha>0$. Because of the exact formula
    $$\varphi_{U_{\alpha}}(\theta,\psi)=\frac{\sin(\pi \theta/\alpha)\sin(\psi)^{\pi/\alpha}}{\Big(\int_0^{\alpha}\int_0^{\pi}\sin(\frac{\pi\theta}{\alpha})^2\sin(\psi)^{2\pi/\alpha}\sin(\psi)d\psi d\theta\Big)^{1/2}},$$
    a similar construction to that for the circular sector in $\mathbb{R}^2$ shows the failure of volume doubling. This leads to the following conjecture.

    \begin{Conj}
        \label{conjVD}
        Any family of bounded domains $U$ (of Euclidean space or of a Riemannian manifold) cannot satisfy ($\varphi_U^2$-\textup{VD}) uniformly whenever those domains include at least one angle tending to $0$. See for example Figure \ref{VDfail} below. 
    \end{Conj}

    To give more evidence supporting Conjecture \ref{conjVD}, consider domains $U_{\alpha}\subseteq \mathbb{R}^2$ such as the left picture in Figure \ref{VDfail}. The boundary Harnack principle \cite[Theorem~5.12]{lierllsc} implies that locally near the vertex of an angle,  $\varphi_{U_{\alpha}}$ is comparable to the harmonic function $h_{\alpha}(r,\theta)=r^{\pi/\alpha}\sin(\pi\theta/\alpha)$. The fact that $\varphi_{U_{\alpha}}$ is normalized so that $\|\varphi_{U_{\alpha}}\|_{L^2(U_{\alpha})}=1$ does not matter, since ($\varphi^2_{U_{\alpha}}$-VD) is invariant under multiplying $\varphi_{U_{\alpha}}$ by a constant. It is easily checked that near the vertex of a cone, $U_{\alpha}$ fails to be $h_{\alpha}^2$-volume doubling. This observation does not, however, imply the failure of ($\varphi_{U_{\alpha}}^2$-VD). Indeed, \cite[Theorem~5.12]{lierllsc} requires local inner uniformity, but any family of domains with an angle $\alpha\to 0$ is not even locally inner uniform.    

    \begin{figure}[H]
    
  \centering
  \includegraphics[width=0.40\textwidth]{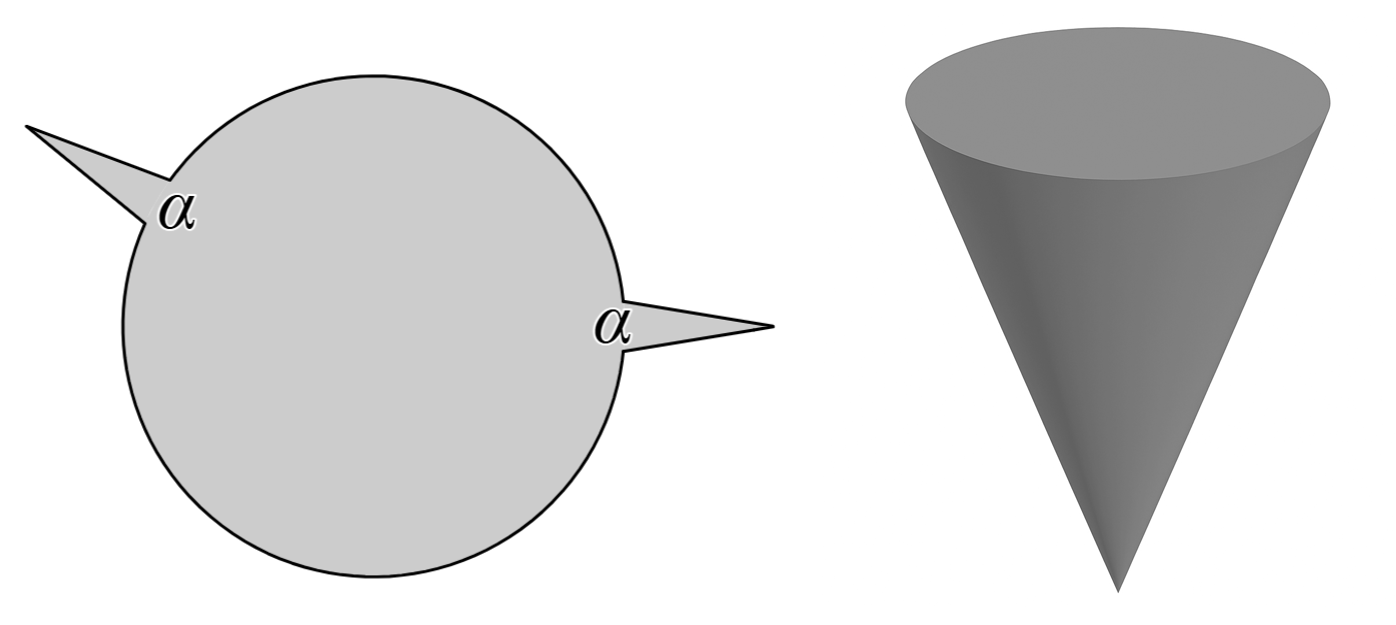}
  \caption{Examples of domains $U_{\alpha}$ parametrized by an angle $\alpha>0$ which conjecturally fail to be $\varphi_{U_{\alpha}}^2$-volume doubling uniformly as $\alpha\to 0$. Depicted on the right is a bounded cone with opening angle $\alpha$, more explicitly $U_{\alpha}=\{x^2+y^2<\tan(\alpha/2)^2z^2,0<z<1\}\subseteq \mathbb{R}^3$.}
  \label{VDfail}
\end{figure}
\end{Exm}

\textbf{Acknowledgements.} The first author (BC) would like to thank Christian Remling for his insightful responses to questions raised by BC on the MathOverflow website, which inspired Lemma \ref{noweight}. We would also like to thank the referee for their careful reading of the manuscript, for their helpful comments, and for bringing the references \cite{lu1} and \cite{lu2} to the authors' attention. BC is in part supported by NSF Grant DGE–2139899. LSC is supported by NSF Grant DMS-2343868.

\author{
  \noindent 
  Brian Chao
  \\ Department of Mathematics, Cornell University, Ithaca, NY 14853, USA.
  \\ E-mail: \texttt{bc492@cornell.edu}
}
\\~\\
\author{
  \noindent 
  Laurent Saloff-Coste
  \\ Department of Mathematics, Cornell University, Ithaca, NY 14853, USA.
  \\ E-mail: \texttt{lsc@math.cornell.edu}
}

\end{document}